\documentclass[pdftex]{ectaart}
\usepackage{amsmath,amsthm,amssymb,amsfonts,mathrsfs,mathtools,dsfont,accents}
\usepackage{geometry,graphicx,pdflscape,color}
\usepackage{natbib}
\RequirePackage[colorlinks=true,linkcolor=magenta,citecolor=cyan,urlcolor=blue]{hyperref}
\RequirePackage{hypernat}
\makeatletter

\newcommand{\Rmnum}[1]{\expandafter\@slowromancap\romannumeral #1@}

\def\widebar{\accentset{{\cc@style\underline{\mskip10mu}}}}
\makeatother

\newcommand{\E}{\mathbb{E}}
\newcommand{\F}{\mathcal{F}}
\newcommand{\T}{\mathrm{T}}
\newcommand{\ds}{\,\mathrm{d}}
\newcommand{\e}{\varepsilon}

\newcommand{\R}{\mathbb{R}}
\newcommand{\s}{\mathcal{S}}

\newcommand{\ee}[2]{e^{-i2\pi #1#2/T}}
\newcommand{\ei}[2]{e^{i2\pi #1#2/T}}
\newcommand{\Fda}[2]{\widehat{F}(\mathrm{d} A_{#1})_{#2}^n}
\newcommand{\Fdm}[2]{\widehat{F}(\mathrm{d} M_{#1})_{#2}^n}
\newcommand{\Fdx}[2]{\widehat{F}(\mathrm{d} X_{#1})_{#2}^n}
\newcommand{\Fc}[2]{\widehat{F}(c_{#1})_{#2}^{n,N}}
\newcommand{\chat}[2]{\widehat{c}^{n,N,M}_{#1}(#2)}
\newcommand{\dd}[2]{d^{n,N}_{#1}(#2)}

\newcommand{\thb}[2]{\theta^n_{#1}(#2)}
\newcommand{\ta}[2]{\tau^{#1}_{#2}}
\newcommand{\taa}[4]{\ta{#1}{#3-1}\vee\ta{#2}{#4-1}}
\newcommand{\tab}[4]{\ta{#1}{#3}\wedge\ta{#2}{#4}}

\newtheorem{assumption}{Assumption}
\newenvironment{Assump}[1]
	{\renewcommand\theassumption{#1}\assumption}
	{\endassumption}

\theoremstyle{plain}

\newtheorem{lem}{Lemma}
\newtheorem{thm}{Theorem}
\newtheorem{prop}{Proposition}
\newtheorem{corol}{Corollary}

\theoremstyle{definition}

\newtheorem{remk}{Remark}
\newtheorem{algo}{Algorithm}

\graphicspath{{./pics/}}
\numberwithin{equation}{section}

\begin{document}
\begin{frontmatter}
\title{The Fourier Transform Method for Volatility Functional Inference by Asynchronous Observations}
\runtitle{Volatility Functionals by spectrum}

\begin{aug}
	\author{\fnms{Richard Y.} \snm{Chen}\corref{} \ead[label=e1]{yrchen@uchicago.edu}\thanksref{t1}}
	\thankstext{t1}{The author would like to acknowledge the supports from \textit{National Science Foundation grant DMS 17-13129} and \textit{Stevanovich Fellowship} from the University of Chicago, and gratefully thanks Per A. Mykland for his valuable discussions and unwavering support to pursue this topic.}
	\address{Department of Statistics, University of Chicago, Chicago, IL 60637, U.S.A., \printead{e1}.}
\end{aug}
\runauthor{Chen, Y. R.}
\date{October, 2019}

\begin{abstract}
	We study the volatility functional inference by Fourier transforms. This spectral framework is advantageous in that it harnesses the power of harmonic analysis to handle missing data and asynchronous observations without any artificial time alignment nor data imputation. Under conditions, this spectral approach is consistent and we provide limit distributions using irregular and asynchronous observations. When observations are synchronous, the Fourier transform method for volatility functionals attains both the optimal convergence rate and the efficient bound in the sense of Le Cam and H\'ajek. Another finding is asynchronicity or missing data as a form of noise produces ``interference'' in the spectrum estimation and impacts on the convergence rate of volatility functional estimators. This new methodology extends previous applications of volatility functionals, including principal component analysis, generalized method of moments, continuous-time linear regression models et cetera, to high-frequency datasets of which asynchronicity is a prevailing feature.
\end{abstract}
\end{frontmatter}

\section{Introduction}
Volatility inference from high-frequency financial data have drawn vigorous academic efforts since the beginning of this new millennium. Volatility is a pivotal measure of risk, and the pillar of many financial models. High-frequency datasets offer promising venues for accurate volatility proxies. Early developments focused on integrated volatility \citep{abdl01, abdl03, bs02, bs04}. More recent literature advanced toward the estimation problems of spot volatility \citep{fw07,k10,apps12,aj14,mmr15,bhmr19}. Based on nonparametric estimates of spot volatility, estimation of the integrated volatility functional
\begin{equation}\label{def.S(g)}
	S(g)_T = \int_0^T g(c(t))\ds t
\end{equation}
is possible by plugging in nonparametric spot estimates into the functional $g(\cdot)$ and forming a Riemann sum.

In this paper, we employ the ideas and tools from harmonic analysis to study the statistical inference for volatility functionals defined as (\ref{def.S(g)}) when the data is observed both irregularly and asynchronously from a multivariate It\^o semimartingale. Specifically, $c(t)$ is the spot\\
 volatility matrix (instantaneous covariance at time $t$) of a continuous multivariate It\^o semimartingale; $g(\cdot)$ is a smooth functional of econometric interests, for instance, $g$ could be a map from the spot volatility matrix to spot betas, or from the volatility matrix to its distinct eigenvalues or the corresponding eigenvectors, etc.

\textit{What are the motivations for studying volatility functionals?} Many financial time series applications can be formulated as volatility functionals. Their importance and multifarious utilities lie in the fact that volatility is one of the central concepts in modern-day financial theories and practices, many empirical investigations are conducted by measuring volatility, examples include but not limited to measuring market risk, model calibration, portfolio selection, option pricing. Recent applications include principal component analysis \citep{ax19}, linear regression \citep{ltt17}, specification tests \citep{ltt16}, generalized method of moments \citep{lx16}. Volatility functionals can also be used in quantifying statistical uncertainties of various volatility estimators, such as quarticity.

\textit{What have been done about volatility functional inference?} Previously, \cite{jr13} proposed a nonparametric plug-in methodology for functional estimation, where the plug-ins are finite differences of realized variances (hereafter RV). The Riemann sum of functionals of plug-ins with explicit bias correction satisfy an asymptotic theory which is rate-optimal and semiparametrically efficient. Recently, \cite{llx19} introduced jackknife for bias correction and a simulation-based method for variance estimation, which are derivative-free and greatly facilitate applications. \cite{ll17} studied efficient functional estimation when volatility exhibits long-memory property, \cite{y18} utilized matrix calculus to ease the burden of computing derivatives and removed various bias terms to allow a more flexible range of the tuning parameter. More recently, \cite{c19a} uses pre-averaging method to provide noise-robust and rate-optimal functional estimators and extends the inferential theory of volatility functionals to the setting of noisy data.

So far, the methods and inferential theories of volatility functionals rely on the setup that the data are synchronous, thus previous applications of volatility functionals required some synchronization procedures before calculating the estimators. Some synchronization procedures, including the previous-tick method, result in reduced sample sizes and possible synchronization bias such as the Epps effect. The loss of data and bias become more pronounced for data of illiquid assets. 

\textit{What are new in this paper?} The methodology of this paper enables valid volatility functional inference using irregular and asynchronous observations from multiple time series. The methodology of this paper employs the Fourier transform to translate the information contained in the data from the time domain to the frequency domain. The most significant advantages of the frequency-domain methodology over the time-domain counterparts include the following:
\begin{itemize}
	\item operations in the frequency domain bypass the troublesome asynchronicity encountered in the time domain and circumvent data synchronization and imputation, thereby offer an elegant approach to harness more prevalent asynchronous data;
	\item the frequency-domain technique for estimating spot volatility are based on the integration-type operations (Fourier transform is a integral transform) rather than finite differences of RVs used by time-domain techniques, thus it is numerically more stable.
\end{itemize}
The handling of asynchronicity by the Fourier transform method will be further discussed in the rest of this paper. To demonstrate the numerical stability of the Fourier transform method, figure \ref{fig.chatcurve} compares the estimates of one volatility trajectories by finite differences of RV and the Fourier transform method. In contrast to the estimates based on RV, the Fourier transform method has much less jiggling and shows high fidelity to the true sample path of volatility.
\begin{figure}
	\center
	\caption{Spot volatility estimator: realized variance and Fourier transform}
	\includegraphics[width=.8\textwidth]{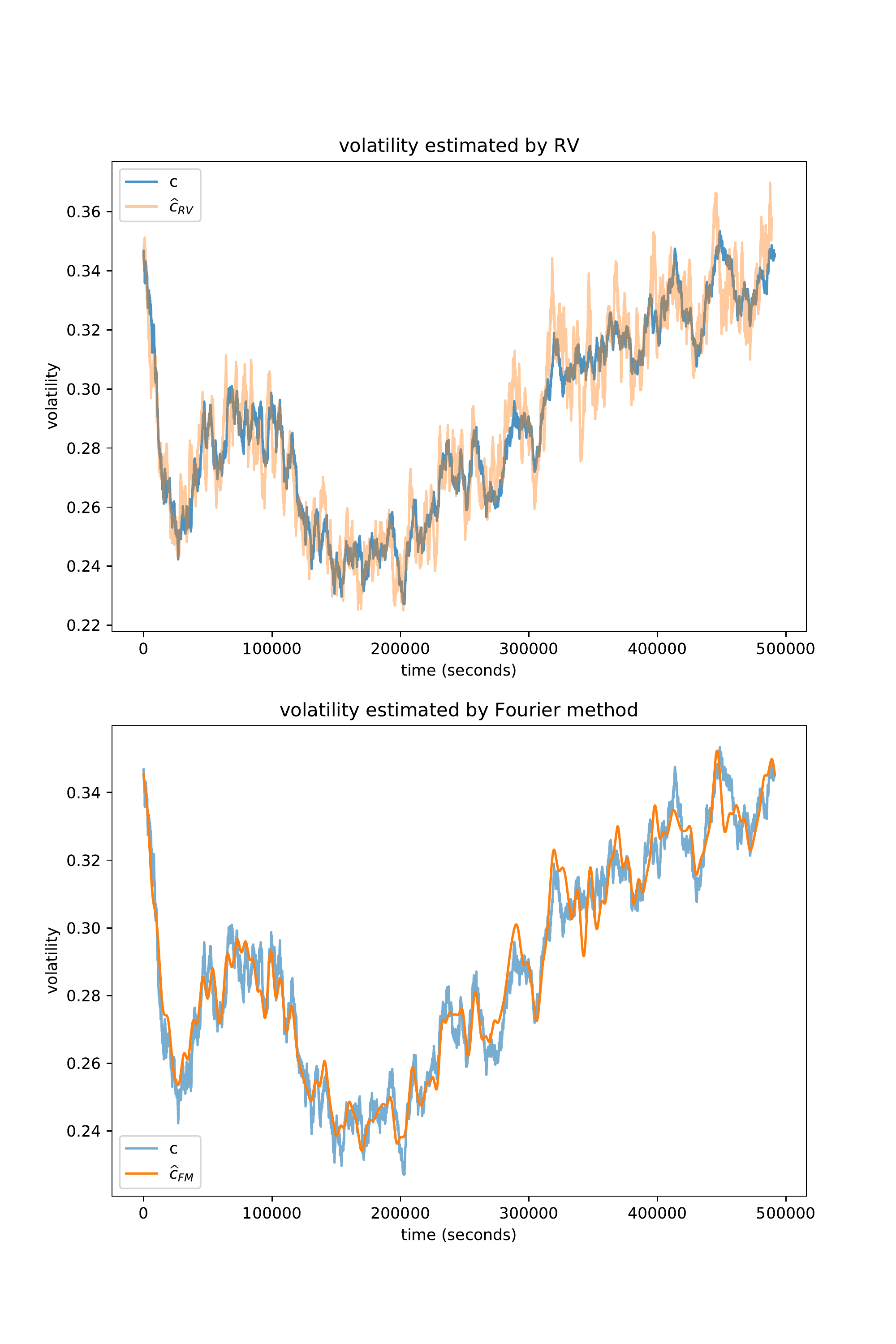}\\
	The sample path of volatility in the simulation is labeled by $c$, it is driven by a fractional Brownian motion with the Hurst parameter $H=.56$; the estimate $\widehat{c}_{RV}$ is by finite differences of RVs which form a rough path; the estimates $\widehat{c}_{FM}$ is by the Fourier transform method.
	\label{fig.chatcurve}
\end{figure}

In addition to the methodological contribution, this paper includes the following results:
\begin{enumerate}
	\item identifying the finite-sample mean-square rate of the spot co-volatility estimation when the data are observed asynchronously; due to asynchronous observations the rate is different from the rate established by \cite{mr15} in the univariate setting;
	\item establishing asymptotic distributional results for volatility functional, revealing: (1) convergence rates of functional estimators as determined by a tuning parameter in the frequency domain and (2) asymptotic variances and how they are impacted by the temporal spacing of observations; 
	\item discovering the fundamental limits on the amount of frequency-domain information that can be utilized without producing bias in the asynchronous setting and the resultant convergence rate of the Fourier transform method;
	\item achieving the optimal convergence rate and the efficiency bound as \cite{jr13}, \cite{llx19} when the observations are synchronous; in the case of asynchronous observations we show the functional estimator can converge with the optimal rate but is biased. 
\end{enumerate}

The methodology and its inferential theory of this paper have their roots in several foundational papers. The nonparametric plug-ins are due to \cite{mm02,mm09}, the former proposed a spot volatility estimator using trigonometric series based on the premise of Fourier analysis, the later formulated this method in terms of complex exponentials; the Fourier estimation method for volatility was further developed by \cite{cg11}. The asymptotic results are central limit theorems of the stable type, cf. \cite{j97}; \cite{jp98} provided stable convergence theorems for discretized solutions to some stochastic differential equations.


This paper is organized as follows:
\begin{itemize}
	\item section \ref{sec:setting} sets up notation and states assumptions for the rest of the paper, it serves as a reference session and can be skipped in the first reading; 
	\item section \ref{sec:method} formally introduces and explains the frequency-domain method for volatility matrix; 
	\item section \ref{sec:consistency} establishes the mean-square rate of spot volatility matrix in the asynchronous setting; and shows the consistency of estimators of volatility spectrum, spot volatility and volatility functionals;
	\item section \ref{sec:func} establishes the second-order asymptotic results, particularly section \ref{sec:uni} shows the stable central limit theorem for functionals of an element in volatility matrix, section \ref{sec:multi} gives the result for general functionals of the whole volatility matrix and shows the Fourier transform method enjoys the optimal rate and achieves the efficient bound in the synchronous setting.
\end{itemize}

\section{Setting}\label{sec:setting}
This paper considers time series data from a fixed time window that can be modeled by a continuous It\^o semimartingale defined on a filtered probability space:
\begin{equation}\label{def.X}
    X(t) = X(0) + \int_0^tb(u)\ds u + \int_0^t\sigma(u)\ds W(u),
\end{equation}
where $b(u)\in\R^d$, $\sigma(u)\in\R^{d\times d'}$ with $d\le d'$, $W$ is a $d'$-dimensional standard Brownian motion; the spot volatility (instantaneous covariance matrix) is $c(u)=\sigma(u)\sigma(u)^\T\in\s^+_d$, where $\s^+_d$ denotes the convex cone of $d\times d$ positive-semidefinite matrices.

Generally, our observations from different components of the multivariate process (\ref{def.X}) are asynchronous, and the sample sizes are different across dimensions. Next, we will introduce the notation for the irregular and asynchronous temporal spacing.

\subsection{Notations}\label{sec.notation}
Here are the mathematical notations we will use. $i=\sqrt{-1}$ is the imaginary unit; for a complex number $z$, $\overline{z}$ is its complex conjugate; for a real number $x$, $\lfloor x\rfloor$ is its integer part; for a matrix $\mathbf{L}$, $\mathbf{L}^\T$ is its transpose; for $x,y\in\R$, $x\wedge y=\min(x,y)$, $x\vee y=\max(x,y)$; $a_n\asymp b_n$ means both $\{a_n/b_n\}$ and $\{b_n/a_n\}$ are bounded sequences. For a function $F$ defined on $\R^{d}$ (resp. $\R^{d\times d}$), $\partial_jF$ (resp. $\partial_{jk}F$) is its derivative with respect to the $j$-th (resp. $(j,k)$-th) argument. $ \overset{\mathbb{P}}{\longrightarrow}$ represents convergence in probability, $\overset{\mathcal{L}-s}{\longrightarrow}$ stands for stable convergence in law, $\mathcal{MN}(\mu,\Sigma)$ represents a mixed normal distribution with random mean $\mu$ and covariance $\Sigma$.

Below are notations for observations and temporal spacing used throughout this paper:
\begin{itemize}
    \item if $U$ is a $\R^d$-valued  process, $U_j$ is the $j$-th component of $U$; if $U$ is a $\R^{d\times d}$ or $\R^{d\times d'}$-valued process, $U_{jk}$ and $U_{j\cdot}$ are the $(j,k)$-th component and the $j$-th row of $U$, respectively;
    \item $\mathcal{T}_j=\{\tau^j_h,\,h=0,\cdots,n_j\}$ is the set of observation times of the $j$-th component process $X_j$; without loss of generality, by time translation let $\min_j\tau^j_0=0$ and $\max_j\tau_{n_j}=T$, i.e. the observation period is $[0,T]$; 
    \item $I^j_h = (\tau^j_{h-1},\tau^j_h]$ is time interval between two consecutive observations of $X_j$;
    \item $\Delta^j_h = \tau^j_h-\tau^j_{h-1}$ is the length of $I^j_h$, let $\Delta^j=\max_h \Delta^j_h$ be the observational mesh of $X_j$;
    \item $\underline{n}=\min_j n_j$ and $n=\max_j n_j$ are the smallest and largest sample sizes among all the dimensions respectively; $\Delta(n)=\max_j \Delta^j$ is the largest mesh size;
    \item $\delta^j_h$ is the first-order difference operator according to the observational times of $X_j$, i.e., given a generic scalar process $U$, $\delta^j_h(U) = U(\tau^j_h) - U(\tau^j_{h-1})$ is the increment of the process $U$ over the time interval $I^j_h$.
\end{itemize}

\subsection{Assumptions}
Here we state the assumptions on the spacing of observational times, the sample paths of the volatility process, and a regularity condition on the functionals.

In this paper, we study the in-fill asymptotics, i.e., we suppose $T$ is a finite constant, and within the finite time interval $[0,T]$ the smallest sample size of all dimensions $\underline{n}\to\infty$ ($\Delta(n)\to0$) in the asymptotic analysis. We assume that, for each $n$ the temporal spacings of all dimensions are of the same magnitude. We also need to assume that the observational times are independent of the sample path (exogenous), this precludes the possibility that the observational times of the process depends on the process itself, for example, hitting times.
\begin{Assump}{T}[temporal spacing]\label{a.T}
$\Delta(n)\to0$ as $n\to\infty$. For some finite constant $K$,
	\[ \frac{\Delta(n)}{\min_{j}\min_{h}{\Delta^j_h}}<K. \]
Let $\F=\sigma(X_t,\,t\in[0,T])$ and $\mathcal{G}=\sigma(\mathcal{T}_j,\,j=1,\cdots d)$, $\forall A\in\F,\,B\in\mathcal{G}$, $A$ and $B$ are independent events.
\end{Assump}

The following assumption is on sample path continuity and local boundedness in a convex subspace. 
\begin{Assump}{U}[volatility continuity and local boundedness]\label{a.U}
There is a sequence of pairs $(\tau_m,\s_m)$, where $\tau_m$ is a stopping time and $\tau_m\nearrow\infty$, each $\s_m\subset\mathcal{S}^+_d$ is a compact subset of positive semidefinite matrices such that
	\[t\in[0,\tau_m] \Rightarrow \|b(t)\|+\|c(t)\|\le m,\, c(t)\in\s_m.\]
\end{Assump}

Given a continuous function $f$, its \textit{modulus of continuity} $\omega_f(\Delta)$ is defined as
\begin{equation}\label{modulus.continuity}
	\omega_f(\Delta) = \sup_{|x-y|\le\Delta}\|f(x)-f(y)\|.
\end{equation}

In order to establish an inferential theory, it is necessary for us to put constraints on the smoothness of the sample paths of volatility.
\begin{Assump}{V-$\alpha$}[volatility regularity]\label{a.V}
The sample path of the volatility $c$ is continuous almost surely. The modulus of continuity of $c$ satisfies
 	\[\omega_c(\Delta) \le \Delta^\alpha,\,\,\alpha>0.\]
\end{Assump}

\begin{remk}
Assumption \ref{a.T}, \ref{a.U} are needed for consistency; assumption \ref{a.V} with $\alpha>1/2$ is further needed for central limit theorems. Assumption \ref{a.U}, \ref{a.V} can be rephrased that the volatility as a function of time belong to the H\"older ball
	\[ H^\alpha(K) \coloneqq \left\{f\in C([0,T])\left| \sup_{t\in[0,T]}\|f(t)\|+ \sup_{t\ne u}\frac{\|f(t)-f(u)\|}{|t-u|^\alpha} \le K \right.\right\} \]
	for some $K>0$.
\end{remk}

We require that the functionals $g:\s^+_d\mapsto\R^r$ satisfy
\begin{equation}\label{cond.g}
	g\in\mathcal{C}^2(\s)
\end{equation}
where $\s$ is a compact subspace of $\s^+_d$, 
$\s\supset\cup_m\s^\epsilon_m$ for some $\epsilon>0$, $\s_m^\epsilon=\big\{A\in\mathcal{S}^+_d: \inf_{M\in\s_m}\|A-M\|\le\epsilon\big\}$ is the $\epsilon$-enlargement of the subspace $\s_m$ 
and $\s_m$ is identified in assumption \ref{a.U}. Note that $c(t)\in\s_m$ if $t\le\tau_m$, hence any consistent estimation of $c(t)$ lies in the subspace $\s_m^\epsilon$ in large samples.

For instance, differentiable functions whose derivatives are of polynomial growth satisfy (\ref{cond.g}), i.e., if for some constants $K>0$ and $r\ge2$,
\begin{equation*}
	\|\partial^hg(c)\|\le K(1+\|c\|^{r-h}), \qquad h=0,1,2,
\end{equation*}
then $g\in\mathcal{C}^2(\s)$.

\section{Fourier Method}\label{sec:method}
Given a functional of econometric interest and a nonparametric estimator of spot volatility, we construct our functional estimator via the plug-in framework of \cite{jr13}. In this framework, computing a functional estimator entails (i) computing the nonparametric estimates of spot volatility at various time points; (ii) plugging the nonparametric estimates into the functional and computing the Riemann sum.

The spot volatility estimator is a crucial element in volatility functional estimation. For a given functional, the large sample properties of the functional estimator largely relies on the asymptotics of the nonparametric estimator of spot volatility. In this paper, to cope with asynchronicity and generalize the framework of \cite{jr13}, we choose the Fourier method to compute nonparametric estimates of spot volatility. The Fourier method for volatility function estimation comprises of 3 steps:
\begin{enumerate}
	\item Estimate the Fourier coefficients of volatility (volatility spectrum);
	\item Estimate the spot volatility from the estimates of its Fourier coefficients;
	\item Plug in the estimates of spot volatility and evaluate the functionals.
\end{enumerate}

\subsection{Volatility spectrum by Bohr convolution}
Before presenting and explaining the Fourier method, we give a quick review on Fourier transform, Fourier series, and a result in approximation theory.
Given a function $f$ defined on $[0,T]$, for $q\in\mathbb{N}^+$, define its \textit{Fourier transform} and \textit{Fourier-Stieltjes transform} as
\begin{equation}\label{def.FT}
\begin{array}{lcl}
	F(f)_q           &\coloneqq& \int_0^T f(t)\,\ee{q}{t}\ds t \nonumber\\
	F(\mathrm{d}f)_q &\coloneqq& \int_0^T\ee{q}{t}\ds f(t)
\end{array}
\end{equation}
If $f(0)=f(T)$, it can be expanded into \textit{Fourier series}:
\begin{equation}\label{def.IFT}
	f(t) = \frac{1}{T}\sum_{q=-\infty}^\infty F(f)_q\,\ei{q}{t}
\end{equation}

Define the following function approximation involving finite Fourier series\footnote{It is the \textit{Ces\`aro sum} on Fourier series. We are using Ces\`aro sum to ensure uniform convergence of (\ref{approximation}) and the Fourier method (\ref{def.chat}) for spot volatility; see section \ref{sec:consis.spot}. Its effect can be expressed through \textit{Fej\'er kernel}. Fej\'er kernel is a summability kernel (also a \textit{delta sequence}) whose basic properties are summarized in appendix \ref{apdx:trig}.}
\begin{equation}\label{def.fhat}
	\widehat{f}^M(t) \coloneqq \frac{1}{T}\sum_{q=-M+1}^{M-1}\Big(1-\frac{|q|}{M}\Big)F(f)_q\,\ei{q}{t}
\end{equation}
By (\ref{rep.fhat}) and lemma \ref{lem.Fejer.uc}, we have for some $K>0$,
\begin{equation}\label{approximation}
	\sup_{t\in[0,T]}\big|\widehat{f}^M(t) - f(t)\big| \le K\omega_f(1/M),
\end{equation}
where $\omega_f$ is the modulus of continuity defined in (\ref{modulus.continuity}).

In statistical applications, according to \cite{mm02, mm09}, we can approximate the Fourier-Stieltjes transform of $X$ and the Fourier transform of the process $c$ over $[0,T]$ by the following quantities:
\begin{eqnarray}
	\Fdx{j}{s} &\equiv& \sum_{h=1}^{n_j}\delta^j_h(X_j)\,\ee{s}{\tau^j_{h}}, \label{def.FShat}\\
	\Fc{jk}{q} &\equiv& \frac{1}{2N+1}\sum_{s=-N}^N\Fdx{j}{q-s}\times\Fdx{k}{s}. \label{def.Fhat}
\end{eqnarray}
We call (\ref{def.Fhat}) as \textit{spectrum estimator}. The available frequency coordinates for $\Fdx{j}{s}$ are $0,\pm1,\cdots$, $\pm\lfloor n_j/2\rfloor$, and given $N\le\lfloor n_k/2\rfloor$ the available frequency coordinates for $\Fc{jk}{q}$ are $0,\pm1,\cdots$, $\pm(\lfloor n_j/2\rfloor-N)$.

\begin{remk}
\cite{phl16} generalized the Bohr convolution in volatility spectrum estimation by applying a spectral kernel function $\Phi$,
\begin{equation*}
	\Fc{jk}{q} \equiv \frac{1}{2N+1}\sum_{s=-N}^N \Phi\Big(\frac{s}{N}\Big)\, \Fdx{j}{q-s} \times \Fdx{k}{s}
\end{equation*}
where the spectral kernel function $\Phi$ satisfies that 
\begin{equation}\label{cond.Phi}
\left\{\begin{array}{l}
	\Phi(w)\ge 0,\, w\in[-1,1] \\
	\int_{-1}^1\Phi(w)\ds w=1 \\
	\int_{-1}^1w\Phi(w)\ds w=0 \\
	\int_{-1}^1|\Phi(w)|^2+|w^p\Phi(w)|^2\ds w<\infty,\,p=1,2
\end{array}\right.
\end{equation}
The spectrum estimator (\ref{def.Fhat}) is a special case with the constant spectral kernel function $\Phi=1$. For our purpose of volatility functional estimation, we choose to use (\ref{def.Fhat}). The reason is that the constant spectral kernel function minimizes the asymptotic variance of our volatility functional estimator. The variance-minimizing property is implied by (\ref{cond.Phi}) and \textit{Parseval's identity}.
\end{remk}

\subsection{Spot volatility by Fourier transform}
Based on the Fourier coefficient estimates $\Fc{jk}{q}$'s, the spot volatility can be estimated by \textit{Fourier-Fej\'er inversion}
\begin{equation}\label{def.chat}
    \chat{jk}{t} \equiv \frac{1}{T}\sum_{q=-M+1}^{M-1}\Big(1-\frac{|q|}{M}\Big)\Fc{jk}{q}\,\ei{q}{t}
\end{equation}
where $M\le\lfloor n_j/2\rfloor-N+1$. In the rest of this paper, we call (\ref{def.chat}) as \textit{spot estimator}.

By defining the following vector and matrix
\begin{eqnarray}\label{def.Fc.mat}
	\Fdx{}{s} &\equiv& \big[\Fdx{1}{s},\cdots,\Fdx{d}{s}\big]^\T \nonumber\\
	\Fc{}{q}  &\equiv& \frac{1}{2N+1}\sum_{s=-N}^N\Fdx{}{q-s}\cdot\widehat{F}(\mathrm{d}X)_s^{n,\T},
\end{eqnarray}
we can express the elementwise defined estimator $\chat{}{t}=\big[\chat{jk}{t}\big]_{jk}$ as 
\begin{equation*}
	\chat{}{t} = \frac{1}{T}\sum_{q=-M+1}^{M-1}\Big(1-\frac{|q|}{M}\Big)\Fc{}{q}\,\ei{q}{t}
\end{equation*}

\begin{remk}
The estimator $\widehat{F}(\mathrm{d}X_j)^n_s$ is the discrete Fourier transform (hereafter DFT) of the increments of $X_j$; the spectrum estimator $\Fc{jk}{q}$ is based on the idea akin to that of \textit{Bohr convolution}, i.e., a scaled convolution of the finite sequences $\Fdx{j}{s}$'s and $\Fdx{k}{s}$'s; the spot estimator $\chat{jk}{t}$ is the $M$-order Ces\`aro sum of inverse discrete Fourier transforms (hereafter IDFT) of the Fourier coefficient estimates via Fej\'er kernel.
\end{remk}

\begin{remk}\label{remk.tuning}
We have 1 tuning parameter $N$ for the spectrum estimator $\Fc{jk}{q}$ and 2 tuning parameters $N$ and $M$ for the spot estimator $\chat{jk}{t}$. 
\begin{itemize}
	\item $N$ is the ``level of averaging'', it dictates how many Fourier-Stieltjes transform estimates we use in estimating Fourier coefficients of volatility; by the law of large a higher $N$ leads to a more accurate Bohr convolution as an estimator, but since each item in the Bohr convolution carries a discretization error as an estimator of the corresponding Fourier-Stieltjes transform, $N$ can not be too large; this is quantitatively discussed to the first order in section \ref{sec:consis.spec};
	\item $M$ is the number of Fourier coefficient estimates used in approximating the spot volatility; the more harmonics are taken into account the better function approximation can be, yet we only possess estimates of Fourier coefficients, so we requires $M$ to be large enough but not too large; see (\ref{cond.NM}).
\end{itemize}

\begin{remk}
There are some fundamental constraints on $N$ and $M$ (the limited number of Fourier coefficients that can be estimated and used) due to discrete observations, namely,
\begin{equation*}
\left\{\begin{array}{lcl}
	N &\le& \lfloor n_k/2\rfloor \wedge \big(\lfloor n_j/2\rfloor-M+1\big)\\
	M &\le& \lfloor n_j/2\rfloor-N+1.
\end{array}\right.
\end{equation*}
	
First, we start from the Fourier-Stieltjes transform (\ref{def.FShat}), in order to avoid aliasing due to discrete sampling of the continuous-time signal (i.e., the digital signal versus its analog counterpart), we only compute $\Fdx{j}{s}$ for $|s|\le\lfloor n_j/2\rfloor$\footnote{It is the so-called \textit{Nyquist frequency (folding frequency)}, which is the highest frequency coordinate without the effect of aliasing.}.

Second, due to the form of convolution in the definition (\ref{def.Fhat}) and the constraint by the Nyquist frequency, it follows that $N\le\lfloor n_k/2\rfloor$.

Third, according to the definition (\ref{def.chat}) and the Nyquist frequency, the number of Fourier coefficient satisfies $M\le\lfloor n_j/2\rfloor-N+1$.\footnote{An interesting modification is to use different $N$'s for different frequency coordinate $q$'s. For instance, when estimating a low frequency we can use a relative large $N$ to increase accuracy. It does no alter any asymptotic property but does improve finite-sample accuracy. However, it is a technical complication to the current Fourier transform method and left to future study.}
\end{remk}

To ensure the consistency of the spectrum estimator and the asymptotic mixed normality of volatility functional estimators, we require
\begin{equation}\label{cond.NM}
\left\{\begin{array}{ll}
N &\to \infty \\
N &\le \lfloor\underline{n}/2\rfloor - M+1,
\end{array}\right.\qquad
\left\{\begin{array}{ll}
M/N^{1/(1+2\alpha)} & \to \infty \\
M/N^{1/2} &\to 0,
\end{array}\right.
\end{equation}
where $\alpha$ is specified in assumption \ref{a.V}.
\end{remk}

\subsubsection{The relation with kernel methods and advantages}
Here we provide some intuition for the spot estimator in the simplest case $d=1$ and compare it with kernel estimators of volatility.

Suppose we observe the univariate process at times $\{\tau_0,\tau_1,\cdots,\tau_n\}$, and let $\delta_j=\delta^1_j$ be the first-order difference operator. Note
\begin{equation*}
	\Fc{}{q} = \sum_{h=1}^n\ee{q}{\tau_{h}}\delta_h(X)^2 + \frac{1}{2N+1}\sum_{|s|\le N}\sum_{h\ne v}\ee{q}{\tau_{h}}\ei{s}{(\tau_h-\tau_v)}\delta_h(X)\delta_v(X)
\end{equation*}
we can write
\begin{equation*}
	\widehat{F}(c)^{n,N}_q = \sum_{h=1}^n\ee{q}{\tau_h}\delta_h(X)^2 + \sum_{h\ne v} \ee{q}{\tau_h} D^N\Big(\frac{\tau_h-\tau_v}{T}\Big) \frac{\delta_h(X)\delta_v(X)}{2N+1}
\end{equation*}
where $D^N(\cdot)$ is a kernel function defined later in (\ref{def.Dirichlet}). Furthermore, based on the definition (\ref{def.chat}), one has
\begin{equation}\label{chat.rep}
	\chat{}{t} = \frac{1}{T}\sum_{h=1}^n F^M\Big(\frac{t-\tau_h}{T}\Big)\delta_h(X)^2 + \frac{1}{T}\sum_{h\ne v} F^M\Big(\frac{t-\tau_h}{T}\Big) D^N\Big(\frac{\tau_h-\tau_v}{T}\Big) \frac{\delta_h(X)\delta_v(X)}{2N+1}
\end{equation}
where $F^M(\cdot)$ is another kernel function defined later in (\ref{def.Fejer}). 

Figure \ref{fig.dirichlet.fejer} shows some examples of the kernels $D^N(\cdot)$ and $F^M(\cdot)$. There are some wiggles away from the origin due to the fact that they are trigonometric polynomials. As $N$ and $M$ become large, the kernels concentrate more around the origin.

\begin{figure}
	\caption{Dirichlet kernels and Fej\'er kernels}
	\includegraphics[width=.5\textwidth]{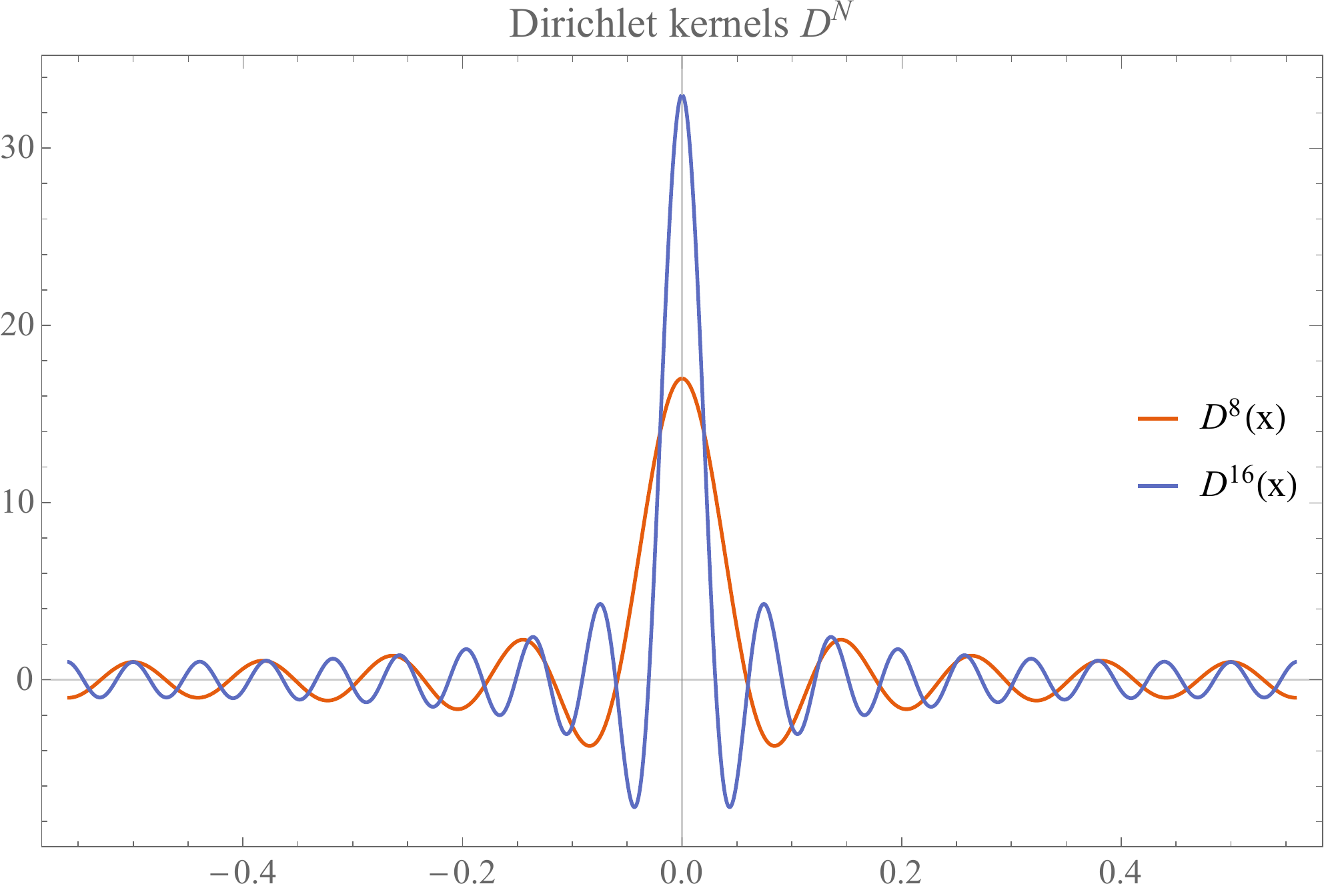}\includegraphics[width=.5\textwidth]{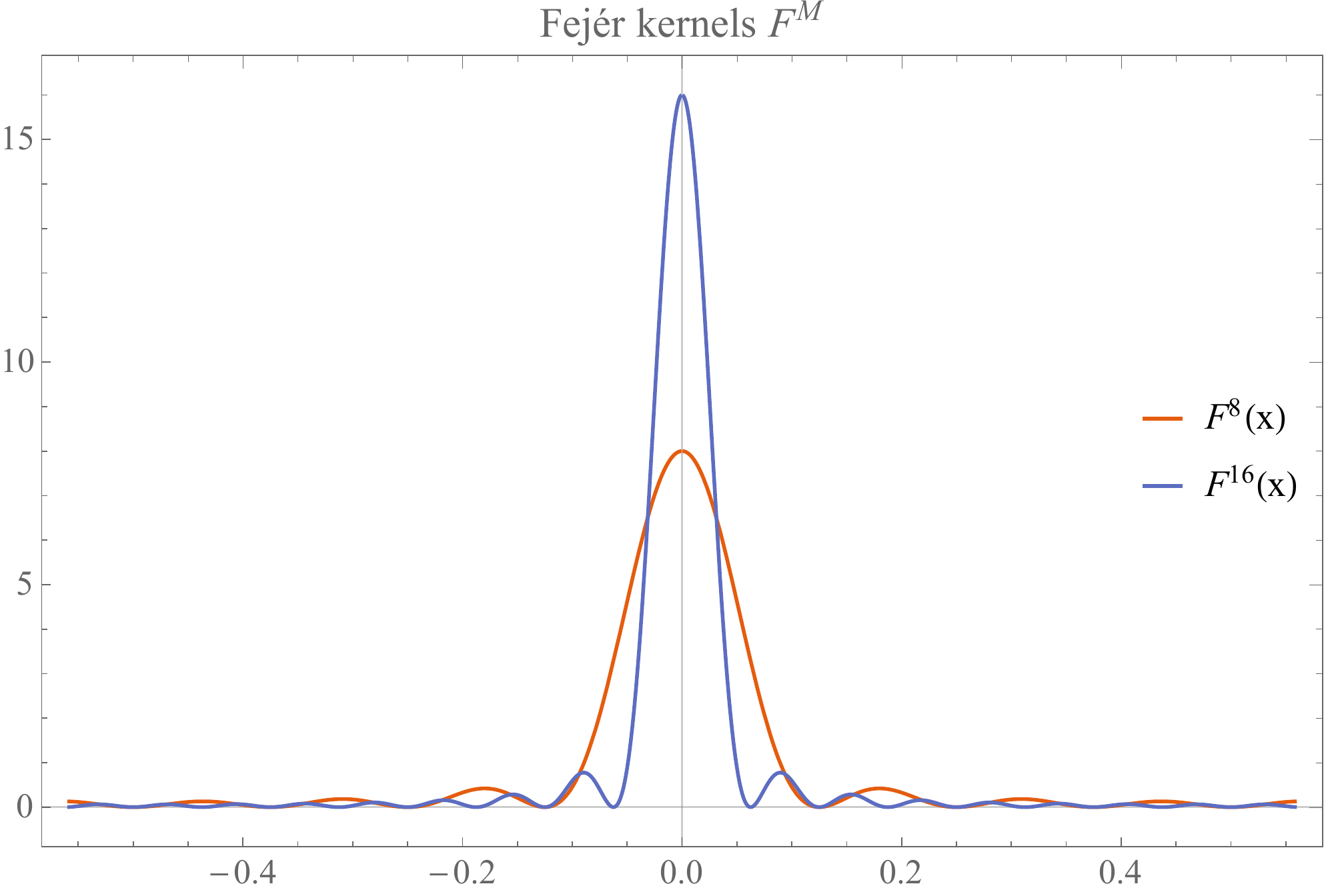}
	\label{fig.dirichlet.fejer}
\end{figure}

If one interprets $\delta_h(X)^2$ as a proxy of $c(\tau_h)$, then $\widehat{F}(c)^{n,N}_q$ as an estimator of the Fourier coefficient is a combination of the DFT of the proxies $\delta_h(X)^2$'s and cross terms involving the sample auto-covariances weighted by the kernel $D^N(\cdot)$:
\begin{equation*}
	\widehat{F}(c)^{n,N}_q = \text{DFT of volatility proxies} + \text{weighted sum of sample auto-covariance}
\end{equation*}
similarly, $\widehat{c}^{n,N,M}(t)$ can be interpreted as a kernel estimator plus cross terms. The kernel is $F^M(\cdot)$ and the cross terms are sample auto-covariances weighted by both $D^N(\cdot)$ and $F^M(\cdot)$:
\begin{equation*}
	\chat{}{t} = \text{a kernel estimator} + \text{weighted sum of sample auto-covariance}
\end{equation*}

The cross term of weighted sum brings additional variation to the estimator (\ref{def.chat}) as opposed to the kernel estimators; see \cite{mmr15}. Naturally we shall ask: given the possible variations from the cross terms, why not just use the DFT of volatility proxies $\delta_j(X)$'s to estimate the Fourier coefficients? Why not just use the kernel estimator to estimate the spot volatility?

Here are 2 significant merits of the spot estimator (\ref{def.chat}) as compared to kernel estimators:
\begin{itemize}
	\item In multivariate settings, (\ref{def.chat}) can estimate the spot co-volatility when different processes are observed asynchronously, because the Bohr convolution is computed in the frequency domain and one does not need to worried about data misalignment and temporal irregularity in the time domain; however many other estimators require data alignment as a prerequisite.
	\item When the sampling frequency is high enough so that microstructure noise $\e$ is required in the model, the estimators (\ref{def.Fhat}) and (\ref{def.chat}) are still consistent with smaller choices of $N$ and $M$; whereas $\delta^j(X+\e)^2$ can no longer be a good proxy for volatility.
\end{itemize}

\subsection{Volatility functionals estimator by Fourier plug-ins}
Based on the plug-in framework for functional estimation, we define the estimator of volatility functionals as
\begin{equation}\label{def.Shat}
	\widehat{S}(g)^{n}_T \equiv \sum_{h=1+L}^{B-L} g\Big(\widehat{c}^{n,N,M}\Big(\frac{hT}{B}\Big)\Big) \frac{T}{B}
\end{equation}
We call (\ref{def.Shat}) as \textit{functional estimator}. We have 4 tuning parameters for the functional estimators, namely $N,M,B,L$. The tuning parameters $N$ and $M$ are inherited from the spot estimator $\widehat{c}^{n,N,M}$, see remark \ref{remk.tuning}. The tuning parameters $B$ and $L$ dictate how to construct the functional estimators:
\begin{itemize}
	\item $B$ is the number of plug-ins in the Riemann sum; a higher $B$ results in a more accurate approximation to the integral, with the cost of higher computational load;
	\item $L$ is the bandwidth at the boundaries of the time window, in which no spot estimate will be taken in the Riemann sum.
\end{itemize}
The boundary values of a volatility sample path $c(0)$ and $c(T)$ are different in general. However, the spot estimator (\ref{def.chat}) is based on trigonometric series and is periodic by construction, hence it holds $\chat{}{0}=\chat{}{T}$. Because of this artifact, no spot estimate near the boundaries will be used in the functional estimator (\ref{def.Shat}).

We require the tuning parameters $B$ and $L$ satisfy
\begin{eqnarray}\label{cond.BL}
\left\{\begin{array}{ll}
	B/N^{1/2} \to \infty & \\
	L = 0 & \text{if } c(0)=c(T)\\
	L \asymp B/M & \text{if } c(0)\ne c(T).
\end{array}\right.
\end{eqnarray}

To summarize, the Fourier method for volatility functionals can be implemented in the following algorithm in pseudo-code:
\begin{algo}
Read in the data vector $\{X_j(\tau^j_h)\}_{h=0,\cdots,n_j},\; j=1\cdots,d$;
\begin{enumerate}
	\item Input the tuning parameters $N$, $M$ satisfying (\ref{cond.NM});
	\item for $j=\{1,\cdots,d\}$:
	\begin{itemize}
		\item Compute $\Fdx{j}{s}$, $s=0,\pm1,\cdots,\pm\lfloor n_j/2\rfloor$ according to (\ref{def.FShat}), by the FFT algorithm;
	\end{itemize}
	\item for $q=\{0,\pm1,\cdots,\pm\lfloor\min_{j}n_j/2\rfloor\}$:
	\begin{itemize}
		\item Compute the complex-valued matrix $\Fc{}{q}$ by (\ref{def.Fc.mat});
	\end{itemize}
	\item Input the tuning parameters $B$, $L$ satisfying (\ref{cond.BL});
	\item Zero-pad\footnote{Note that the required length of spot estimates is $B$ and is usually higher than the number of Fourier coefficients. Zero-padding plus FFT decrease the computational cost from $O(BN)$ of a na\"ive algorithm to $O(B\ln(B))$.} symmetrically each sequence $\{\Fc{jk}{q}\}$ to be of length $B$, and use the FFT algorithm to compute $\chat{jk}{hT/B},\,h=0,\cdots,B-1$;
	\item Plug in $\chat{}{hT/B}$ to compute (\ref{def.Shat}).
\end{enumerate}
\end{algo}

\section{Consistency}\label{sec:consistency}
\subsection{Consistent estimation of volatility spectrum}\label{sec:consis.spec}
In this section, we discuss the convergence of $\Fc{}{q}$ to the true spectrum. Consistency of the spectrum estimation has been shown by \cite{mm09}. Here we will take a closer look at the various components and causes of its estimation error.

First, we introduce a short-hand representation of (\ref{def.X}):
    \[ X = X(0) + A + M, \]
where $A(t) = \int_0^tb(u)\ds u$, $M(t) = \int_0^t\sigma(u)\ds W(u)$.

We can write
\begin{equation}\label{Fhat-F.decomp}
    \Fc{jk}{q} - F(c_{jk})_q = R(0)^{n,N}_{jk,q} + R(1)^{n,N}_{jk,q} + R(2)^{N}_{jk,q},
\end{equation}
where
\begin{eqnarray*}
    R(0)^{n,N}_{jk,q} &=& \frac{1}{2N+1}\sum_{|s|\le N}\big[\Fdx{j}{q-s}\Fdx{k}{s} - \Fdm{j}{q-s}\Fdm{k}{s}\big]\\
    R(1)^{n,N}_{jk,q} &=& \frac{1}{2N+1}\sum_{|s|\le N}\left[\Fdm{j}{q-s}\Fdm{k}{s} - F(\mathrm{d}M_j)_{q-s}F(\mathrm{d}M_k)_s\right]\\
    R(2)^{N}_{jk,q} &=& \frac{1}{2N+1}\sum_{|s|\le N}F(\mathrm{d}M_j)_{q-s}F(\mathrm{d}M_k)_s - F(c_{jk})_q.
\end{eqnarray*}

Essentially, (\ref{Fhat-F.decomp}) decomposes the error in volatility spectrum estimation into 3 effects: 
\begin{itemize}
	\item $R(0)^{n,N}_{jk,q}$ is the effect of the drift term;
	\item $R(1)^{n,N}_{jk,q}$ is the effect of discrete observations of the continuous-time model (discretization \& asynchronicity errors);
	\item $R(2)^{N}_{jk,q}$ is the effect due to finite Bohr convolution (statistical error).
\end{itemize}
Under assumption \ref{a.T}, \ref{a.U}, according to the proof in appendix \ref{apdx:spec}, $\exists K>0$, we have
\begin{equation}\label{errors.Fchat}
\left.\begin{array}{lcl}
	\E\big(|R(0)^{n,N}_{jk,q}|\big) &\le& K TN^{-3/4} \\
	\E\big(|R(1)^{n,N}_{jk,q}|\big) &\le& K N\Delta(n)\\ 
	\E\big(|R(2)^{n,N}_{jk,q}|\big) &\le& K TN^{-1/2}
\end{array}\right\}.
\end{equation}	

\begin{prop}\label{prop.Fc.error}
If assumption \ref{a.T}, \ref{a.U} hold, then $\exists K>0$, such that
\begin{equation*}
	\E\Big(\big|\Fc{jk}{q} - F(c_{jk})_q\big|\Big) \le K\big[N\Delta(n) + TN^{-1/2}\big].
\end{equation*}	
\end{prop}
By (\ref{cond.NM}) and \textit{Markov's inequality}, we have the following corollary.
\begin{corol}\label{corol.Fhat}
    If $N$ satisfies (\ref{cond.NM}), under assumption \ref{a.T}, \ref{a.U},
    \[ \Fc{jk}{q} \overset{\mathbb{P}}{\longrightarrow} F(c_{jk})_q, \qquad q=0,\pm1,\cdots,\pm(\lfloor n_j/2\rfloor-N). \]
\end{corol}

\begin{remk}\label{remk.Fc.error}
For volatility spectrum estimation on a finite time horizon, we summarize the magnitudes (\ref{errors.Fchat}) of various error terms in table \ref{table.Fchat.error}, where $K$ is some finite positive real number.
\begin{table}[!ht]
\caption{approximate magnitude: estimation errors of volatility spectrum}
\begin{tabular}{l|c|c|c}
	              &                  & discretization \&      & \\
	error sources & drift effect     & asynchronicity errors & statistical error\\
	\hline 
	magnitudes    & $\le KN^{-3/4}$  & $\le KN\Delta(n)$      & $\asymp N^{-1/2}$ 
\end{tabular}
\label{table.Fchat.error}
\end{table}

\begin{itemize}
	\item The discretization effect summarizes how irregular and asynchronous observations bear on the spectrum estimator. As it turns out, as long as $N\Delta(n)\to0$, the effect of temporal irregularity and asynchronicity is asymptotically negligible; this is consistent with the finding of \cite{cg11};
	\item The size of the drift effect is dominated by other terms regardless of the choice $N$. In subsequent asymptotic analysis of the volatility spectrum estimator, we can safely assume, without loss of generality,
	\begin{equation}\label{def.X.mtg}
		X(t) = X(0) + \int_0^t\sigma(u)\ds W(u).
	\end{equation}
\end{itemize}
\end{remk}

\begin{remk}\label{remk.N}
Table \ref{table.Fchat.error} indicates that the size of $N$ determines the convergence rate of the spectrum estimator $\Fc{jk}{q}$:
\begin{itemize}
	\item $N=o(\Delta(n)^{-2/3})=o(\underline{n}^{2/3})$ is a sufficient (not necessary) condition under which the asynchronicity effect is (asymptotically) negligible compared with the statistical error; in this scenario, the rate of convergence is $N^{1/2}$ and is dictated by the statistical error of the finite Bohr convolution;
	\item 
	if one take all the available information in the frequency domain by letting $N=\lfloor\underline{n}/2\rfloor-M+1$, the spectrum estimator is biased due to asynchronicity, although the estimator converges with a bias with the rate $\underline{n}^{1/2}$.
\end{itemize}
To avoid the asynchronicity bias, the convergence rate is less that $\underline{n}^{1/2}$. We call this phenomenon \textit{the curse of asynchronicity}. For volatility functionals using the Fourier transform method, we provide a sufficient and necessary conditions for both consistency in section \ref{sec:consis.spot} and unbiased asymptotic normality in section {\ref{sec:uni}}.
\end{remk}
\subsection{Consistent  estimation of spot volatility and its functionals}\label{sec:consis.spot}
In this section, we first state a result on the mean square rate of the spot volatility estimation, then based on this mean square rate, we can guarantee the consistency both the spot estimator (\ref{def.chat}) and the functional estimator (\ref{def.Shat}).

\cite{mr15} proved the asymptotic normality and convergence rate for univariate spot volatility. The next proposition extends their result on the mean square rate to the multivariate and asynchronous settings.
\begin{prop}\label{prop.msr}
Under (\ref{def.X}) and assumption \ref{a.T}, \ref{a.U}, \ref{a.V}, there exists a finite positive constant $K$ such that $\forall j,k=1,\cdots,d$,
\begin{equation*}
	\sup_{t\in[M^{-1},T-M^{-1}]} \E\big|\chat{jk}{t} - c_{jk}(t)\big|^2 \le K \Big(\frac{N^4}{\underline{n}^4}\mathds{1}_{\{j\ne k\}} + M^{-2\alpha} + \frac{M}{N}\Big);
\end{equation*}
additionally, if $c(0)=c(T)$,
\begin{equation*}
	\sup_{t\in[0,T]} \E\big|\chat{jk}{t} - c_{jk}(t)\big|^2 \le K \Big(\frac{N^4}{\underline{n}^4}\mathds{1}_{\{j\ne k\}} + M^{-2\alpha} + \frac{M}{N}\Big).
\end{equation*}
\end{prop}

\begin{remk}
The various terms in the upper bound in proposition \ref{prop.msr} arise from estimation errors of different natures, cf. (\ref{decomp.chat.error}) and (\ref{decomp.Omega}). The sources of these estimation errors are:
\begin{itemize}
\item asynchronous observations;
\item approximation by convolution with the Fej\'er kernel (one type of delta sequences in Fourier analysis);
\item statistical error in the form of stochastic integrals of the Fej\'er kernel with respect to Brownian motion.
\end{itemize}
The magnitude of these estimation errors are summarized in table \ref{table.chat.error}.
\begin{table}[!ht]
\caption{approximate magnitude: estimation errors of spot volatility}
\begin{tabular}{l|c|c|c}
	error sources & asynchronicity error & delta sequence approximation & statistical error\\
	\hline 
	magnitudes    & $\asymp N^2\Delta(n)^2$ & $\asymp M^{-\alpha}$ & $\asymp \sqrt{M/N}$ 
\end{tabular}\label{table.chat.error}
\end{table}
\end{remk}

According to proposition \ref{prop.msr}, we have the following corollary on the uniform consistency of the spot estimator. This corollary generalizes Theorem 3.4 in \cite{mm09} to the case where $c(0)\ne c(T)$. In the spirit of Theorem 2 of \cite{phl16},  it provides a more accurate result on the ``\textit{border effect}'' as a result of $c(0)\ne c(T)$.
\begin{corol}\label{corol.consis.spot}
If (\ref{def.X}) and assumption \ref{a.T}, \ref{a.U} are true, $N$ and $M$ satisfy (\ref{cond.NM}), then 
	\[ \sup_{t\in[M^{-1},T-M^{-1}]} \big\|\chat{}{t} - c(t)\big\| \overset{\mathbb{P}}{\longrightarrow} 0; \]
additionally, if $c(0)=c(T)$,
	\[ \sup_{t\in[0,T]} \big\|\chat{}{t} - c(t)\big\| \overset{\mathbb{P}}{\longrightarrow} 0. \]
\end{corol}

Therefore, we have the consistency of the functional estimator.
\begin{corol}\label{corol.consis.func}
Assume (\ref{def.X}) and assumption \ref{a.T}, \ref{a.U}, (\ref{cond.g}), $N$ and $M$ satisfy (\ref{cond.NM}), $B$ and $L$ satisfy (\ref{cond.BL}), then
	\[ \widehat{S}(g)^n_T \overset{\mathbb{P}}{\longrightarrow} S(g)_T. \]
\end{corol}

\section{Stable Convergence and Asymptotic Normality}\label{sec:func}
In this section, we provide the asymptotic distributions of the functional estimator (\ref{def.Shat}) based on Fourier series. We present limit theorems from the relative simple to the complicated: the univariate setting, the bivariate setting, the multivariate setting with synchronous and asynchronous observations, with $N=o(n)$ and $N\asymp n$.

\subsection{Functionals of univariate volatility}\label{sec:uni}
First of all, let's consider a simple case - estimating functionals of one element in the volatility matrix. The object to estimate is
\begin{equation*}
	S(g)_{jk,T} = \int_0^\T g\big(c_{jk}(t)\big)\ds t,
\end{equation*}
where $j,k=1,\cdots,d$. The estimator is
\begin{equation*}
	\widehat{S}(g)^n_{jk,T} \equiv \sum_{h=1+L}^{B-L} g\Big(\widehat{c}^{n,N,M}_{jk}\Big(\frac{hT}{B}\Big)\Big) \frac{T}{B}.
\end{equation*}

We first present the result for diagonal elements, and in this case the temporal spacing is easier to deal with since the only issue is irregularity in the univariate setting. Then we present the result for off-diagonal elements.

There is no particular reason to favor a regular time grid like $\{hT/B\}_h$ into which to plug spot estimates. For the diagonal elements, one could also use the irregular observation times as the time grid on which to compute the spot estimates:
\begin{equation*}
	\widetilde{S}(g)^{n}_{jj,T} \equiv \sum_{h=1+L}^{n_j-L} g\big(\chat{jj}{\tau_h}\big)\Delta^j_h.
\end{equation*}

The univariate functional estimators $\widetilde{S}(g)^{n}_{jj,T}$ and $\widehat{S}(g)^{n}_{jj,T}$ share the same asymptotic distribution with the same convergence rate.
\begin{thm}\label{thm.uni}
Assume (\ref{def.X}), (\ref{cond.g}), assumption \ref{a.T}, \ref{a.U}, and assumption \ref{a.V} with $\alpha>1/2$, $c(0)=c(T)$. For $j=1,\cdots,d$, if we choose $N=\lfloor n_j/2\rfloor-M+1$ and the other tunning parameters in accordance with (\ref{cond.NM}), (\ref{cond.BL}), then
\begin{eqnarray*}
	n_j^{1/2}\big[\widetilde{S}(g)^{n}_{jj,T} -S(g)_{jj,T}\big] &\overset{\mathcal{L}-s}{\longrightarrow}& \mathcal{MN}\big(0,V(g)_{jj,T}\big)\\
	n_j^{1/2}\big[\widehat{S}(g)^{n}_{jj,T} -S(g)_{jj,T}\big]   &\overset{\mathcal{L}-s}{\longrightarrow}& \mathcal{MN}\big(0,V(g)_{jj,T}\big),
\end{eqnarray*}
where
\begin{equation*}
	V(g)_{jj,T} = T\int_0^T \big[\partial g(c_{jj}(t))\,c_{jj}(t)\big]^2\ds t.
\end{equation*}
\end{thm}

\subsection{Scaled Dirichlet kernel} 
As trigonometric functions form the basis of Fourier analysis, the asymptotic bivariate and multivariate results are based on trigonometric polynomials. The assumptions on temporal spacing are formulated in terms of \textit{Dirichlet kernels}. The $q$-order Dirichlet kernel is defined as
\begin{equation}\label{def.Dirichlet}
	D^q(x) = \sum_{|s|\le q}e^{i2\pi sx},
\end{equation}
and we have
\begin{equation}\label{rep.Dirichlet}
	D^q(x) = \left\{
	\begin{array}{cc}
		\frac{\sin[\pi(2q+1)x]}{\sin(\pi x)}, & x\notin\mathbb{N}\\
		2q+1, & x\in\mathbb{N}.
	\end{array}\right.
\end{equation}
Define step functions of time, for $j=1,\cdots,d$, 
\begin{equation}\label{def.theta_n}
\begin{array}{ll}
	\thb{j}{t} &= \inf\big\{\ta{j}{h},\, t \le \ta{j}{h}\big\} \wedge \ta{j}{n_j}.
\end{array}
\end{equation}
Based on the Dirichlet kernel and the step functions, define the shifted and scaled Dirichlet kernel
\begin{equation}\label{def.dirichlet.scale}
	\dd{jk}{t,u} = \frac{1}{2N+1}D^N\Big(\frac{\thb{j}{t}-\thb{k}{u}}{T}\Big),\,\,\, j,k=1,\cdots,d.
\end{equation}
The function $\dd{jk}{t,u}$ was introduced in \cite{cg11} and it is indispensable to the asymptotic analysis in this paper.
As \cite{cg11}, we formulate the assumption on the irregular and asynchronous observation times through the shifted and scaled Dirichlet kernel.
\begin{Assump}{F}[Fej\'er kernels of time]\label{a.dd}
For $j,k,l,m=1,\cdots,d$, the quadratic integrals of $d^{n,N}_{jk}$ and $d^{n,N}_{lm}$ converge as $n,N\to\infty$. Specifically, $\exists$ $L^1$ functions $\tilde{\theta}_{jk,lm}$, $\acute{\theta}_{jk,lm}$, $\check{\theta}_{jk,lm}$, $\grave{\theta}_{jk,lm}$, such that $\forall t\in[0,T]$,
\begin{eqnarray*}
	\int_0^t N\int_0^u \dd{jk}{u,v}\,\dd{lm}{u,v} \ds v\ds u &\overset{\mathbb{P}}{\longrightarrow}& \int_0^t \tilde{\theta}_{jk,lm}(u) \ds u \\
	\int_0^t N\int_0^u \dd{jk}{u,v}\,\dd{lm}{v,u} \ds v\ds u &\overset{\mathbb{P}}{\longrightarrow}& \int_0^t \acute{\theta}_{jk,lm}(u) \ds u \\
	\int_0^t N\int_0^u \dd{jk}{v,u}\,\dd{lm}{u,v} \ds v\ds u &\overset{\mathbb{P}}{\longrightarrow}& \int_0^t \check{\theta}_{jk,lm}(u) \ds u \\
	\int_0^t N\int_0^u \dd{jk}{v,u}\,\dd{lm}{v,u} \ds v\ds u &\overset{\mathbb{P}}{\longrightarrow}& \int_0^t \grave{\theta}_{jk,lm}(u) \ds u.
\end{eqnarray*}
\end{Assump}
\begin{remk}
In order to have an intuitive understanding of assumption \ref{a.dd}, let's look at the case $j=k=l=m$. Assumption \ref{a.dd} implies $N\int_0^t\dd{jj}{t,u}^2\ds u \overset{\mathbb{P}}{\longrightarrow} \tilde{\theta}_{11,11}(t)$. Based on (\ref{def.dirichlet.scale}), (\ref{rep.Fejer}) and Riemann summation, we have
\begin{equation*}
	\tilde{\theta}_{jj,jj}(t) = \lim_{N\to\infty} \frac{N}{2N+1}\int_0^t F^{2N+1}\Big(\frac{t-u}{T}\Big)\ds u,
\end{equation*}
where $F^{2N+1}$ is the \textit{Fej\'er kernel} defined by (\ref{def.Fejer}). 
by the proof of lemma \ref{lem.dd.syn},
\begin{equation}\label{theta.jjjj}
	\tilde{\theta}_{jj,jj}(t) = \acute{\theta}_{jj,jj}(t) = \check{\theta}_{jj,jj}(t) = \grave{\theta}_{jj,jj}(t) = T/4.
\end{equation}
(\ref{theta.jjjj}) holds under assumption \ref{a.T} and is independent of assumption \ref{a.dd}. Assumption \ref{a.dd} is the condition by which (\ref{theta.jjjj}) can be generalized to the multivariate setting.
\end{remk}

\begin{remk}
The irregularity and asynchronicity in many time-domain techniques are formulated in a notion called \textit{quadratic variation of time} by \cite{mz06}. In the simplest case $j=k=l=m=1$, the quadratic variation of time $H$ is a function of time defined as a limit (which is assumed to exist and is differentiable):
\begin{equation*}
	H(t) = \lim_{n_1\to\infty}\frac{n_1}{T}\sum_{\ta{1}{h}\le t}\big(\ta{1}{h} - \ta{1}{h-1}\big)^2;
\end{equation*}
in the case $t\in[\ta{1}{h-1},\ta{1}{h+1})$ and $T=\pi$, 
assumption \ref{a.dd} implies
\begin{equation*}
	\tilde{\theta}_{11,11}(t) = \lim_{N,n_1\to\infty} \frac{N}{(2N+1)^2}\sum_{\ta{1}{h}\le\thb{1}{t}} \big(\ta{1}{h} - \ta{1}{h-1}\big)\frac{\sin[(2N+1)(t-\ta{1}{h})]^2}{\sin(t-\ta{1}{h})^2}.
\end{equation*}
Both limits above are defined in probability. The counterpart of $H(t)$ in assumption \ref{a.dd} is $\widetilde{\Theta}_{11,11}(t)\coloneqq\int_0^t\tilde{\theta}_{11,11}(u)\ds u$. As we will see later, just as the time derivative $H'(t)$, the time derivative $\widetilde{\Theta}'_{11,11}(t)=\tilde{\theta}_{11,11}(t)$ appears in the asymptotic variances.
\end{remk}

\subsection{Functionals of bivariate volatility}\label{sec:bi}
Now we move on to the limit theorem for functionals of co-volatilities (off-diagonal elements in the volatility matrix). As it turns out, the asymptotic normality requires different conditions when the observations are synchronous and asynchronous.

\begin{Assump}{ST}[synchronous observations]\label{a.ST}
	$n_1=\cdots=n_d$ and $\min_j\ta{j}{h}=\max_j\ta{j}{h}$ for $h=1,\cdots, n_1$.
\end{Assump}

\begin{thm}\label{thm.bi}
Assume (\ref{def.X}), (\ref{cond.g}), assumption \ref{a.T}, \ref{a.dd}, \ref{a.U}, and assumption \ref{a.V} with $\alpha>1/2$, $c(0)=c(T)$, and choose the tunning parameters according to (\ref{cond.NM}) and (\ref{cond.BL}) with 
\begin{itemize}
	\item $N\le\lfloor(n_j\wedge n_k)/2\rfloor -M+1$ if assumption \ref{a.ST} holds,
	\item $N=o((n_j\wedge n_k)^{4/5})$ if assumption \ref{a.ST} does not hold,
\end{itemize}
then for $j,k=1,\cdots,d$,
\begin{equation*}
	N^{1/2}\big[\widehat{S}(g)^n_{jk,T} -S(g)_{jk,T}\big] \overset{\mathcal{L}-s}{\longrightarrow} \mathcal{MN}\big(0,V(g)_{jk,T}\big),
\end{equation*}
where
\begin{equation*}
	V(g)_{jk,T} = \int_0^T \partial g(c_{jk}(t))^2 \times \Big\{\big[\tilde{\theta}_{jk,jk}(t)+\grave{\theta}_{jk,jk}(t)\big] c_{jj}(t)\,c_{kk}(t) + 2\check{\theta}_{jk,jk}(t)\, c_{jk}(t)^2\Big\} \ds t.
\end{equation*}
\end{thm}

\begin{remk}\label{remk.rate}
Compare theorem \ref{thm.uni} and theorem \ref{thm.bi}, we can see that the asymptotic properties of functionals acting on diagonal and off-diagonal elements are drastically different.
\begin{itemize}
	\item \textit{Convergence rates}. For functionals of diagonal elements, the convergence rate can be $n_j^{1/2}$ by choosing $N$ as large as $\lfloor n_j/2\rfloor-M+1$; for functionals of off-diagonal elements, due to the impact of asynchronous observations (also see remark \ref{remk.N}), in order that the limit distribution is a centered mixed normal, we can only choose $N$ smaller than $(n_j\wedge n_k)^{4/5}$ and the attendant convergence rate is strictly less than $(n_j\wedge n_k)^{2/5}$;
	\item \textit{Asymptotic variances}. For functionals of diagonal elements, asymptotic variances are independent of the temporal spacing; on the contrary for functionals of off-diagonal elements, the temporal spacing leaves its imprint on asymptotic variances through $\tilde{\theta}_{jk,jk},\,\acute{\theta}_{jk,jk},\,\check{\theta}_{jk,jk},\,\grave{\theta}_{jk,jk}$ defined in assumption \ref{a.dd}.
\end{itemize}
\end{remk}

\subsection{Functionals of multivariate volatility}\label{sec:multi}
Now, let's look at the fully-fledged result. The object to estimate is $S(g)_T$ defined in (\ref{def.S(g)}), its estimator is $\widehat{S}(g)^n_T$ defined in (\ref{def.Shat}). Our goal is to seek central limit theorems for functionals of a whole volatility matrix in various circumstances.

\begin{thm}\label{thm.multi}
Assume (\ref{def.X}), (\ref{cond.g}), assumption \ref{a.T}, \ref{a.dd}, \ref{a.U}, and assumption \ref{a.V} with $\alpha>1/2$, $c(0)=c(T)$, and choose the tunning parameters according to (\ref{cond.NM}), (\ref{cond.BL}) with
\begin{itemize}
	\item $N\le\lfloor\underline{n}/2\rfloor - M+1$ if assumption \ref{a.ST} holds,
	\item $N=o(\underline{n}^{4/5})$ if assumption \ref{a.ST} does not hold,
\end{itemize}
then we have
\begin{equation*}
	N^{1/2}\big[\widehat{S}(g)^{n}_T -S(g)_T\big] \overset{\mathcal{L}-s}{\longrightarrow} \mathcal{MN}\big(0,V(g)_T\big),
\end{equation*}
where
\begin{multline}\label{def.V(g)}
	V(g)_T = \sum_{j,k,l,m=1}^d\int_0^T \partial_{jk} g(c(t))\,\partial_{lm} g(c(t)) \times \\ 
	\Big\{\big[\tilde{\theta}_{jk,lm}(t)+\grave{\theta}_{jk,lm}(t)\big] c_{jl}(t)\,c_{km}(t) + \big[\acute{\theta}_{jk,lm}(t)+\check{\theta}_{jk,lm}(t)\big] c_{jm}(t)\,c_{kl}(t)\Big\} \ds t.
\end{multline}
\end{thm}

We have the following corollary which immediately follows from theorem \ref{thm.multi} and lemma \ref{lem.dd.syn}. It states that when different time series are observed synchronously, the functional estimator based on the Fourier transform method can be rate optimal and efficient.  
\begin{corol}\label{corol.syn}
Assume (\ref{def.X}), (\ref{cond.g}), assumption \ref{a.T}, \ref{a.ST}, \ref{a.U}, \ref{a.V} with $\alpha>1/2$, $c(0)=c(T)$. Choose the tunning parameters according to (\ref{cond.NM}), (\ref{cond.BL}) with $N=\lfloor n/2\rfloor-M+1$, we have
\begin{equation*}
	\Delta(n)^{-1/2}\big[\widehat{S}(g)^{n}_T -S(g)_T\big] \overset{\mathcal{L}-s}{\longrightarrow} \mathcal{MN}\big(0,V(g)^*_T\big),
\end{equation*}
where
\begin{equation*}
	V(g)^*_T = \sum_{j,k,l,m=1}^d\int_0^T \partial_{jk} g(c(t))\, \partial_{lm} g(c(t)) \times \big[ c_{jl}(t)\,c_{km}(t) +  c_{jm}(t)\,c_{kl}(t) \big] \ds t.
\end{equation*}
\end{corol}
The convergence rate $\Delta(n)^{-1/2} \asymp n^{1/2}$ is optimal and the asymptotic variance $V(g)^*_T$ achieves the efficiency bound, cf. \cite{jr13} and \cite{cdg13}.

We provide an estimator of the asymptotic variance (\ref{def.V(g)}), which is defined as
\begin{equation}\label{def.AVARest}
	\widehat{V}(g)^{n}_T = \sum_{j,k,l,m=1}^d \left[\widehat{V}(0)^{n,N,M,B}_{jk,lm,T} + \widehat{V}(1)^{n,N,M,B}_{jk,lm,T}\right],
\end{equation}
where
\begin{eqnarray*}
	\widehat{V}(0)^{n,N,M,B}_{jk,lm,T} &=& \frac{T}{B} \sum_{h=1}^B \partial_{jk}g(\chat{}{t_h})\,\partial_{lm}g(\chat{}{t_h})\, \chat{jl}{t_h}\, \chat{km}{t_h} \\ 
	&& \times N\delta(n) \sum_{v=1}^{\lfloor t_h/\delta(n)\rfloor} \Big[\dd{jk}{t_h,\vartheta_v}\,\dd{lm}{t_h,\vartheta_v} + \dd{jk}{\vartheta_v,t_h}\,\dd{lm}{\vartheta_v,t_h}\Big] \\
	\widehat{V}(1)^{n,N,M,B}_{jk,lm,T} &=& \frac{T}{B} \sum_{h=1}^B \partial_{jk}g(\chat{}{t_h})\,\partial_{lm}g(\chat{}{t_h})\, \chat{jm}{t_h}\, \chat{kl}{t_h} \\ 
	&& \times N\delta(n) \sum_{v=1}^{\lfloor t_h/\delta(n)\rfloor} \Big[\dd{jk}{t_h,\vartheta_v}\,\dd{lm}{\vartheta_v,t_h} + \dd{jk}{\vartheta_v,t_h}\,\dd{lm}{t_h,\vartheta_v}\Big],
\end{eqnarray*}
and $t_h = hT/B$ with $B$ satisfying (\ref{cond.BL}), $\vartheta_v = v\delta(n)$, $\delta(n)=\min_j\min_h\Delta^j_h$.

According to (\ref{cond.g}), corollary \ref{corol.consis.spot}, and the choices of $t_h$ and $\vartheta_v$, it immediately follows that under the conditions of theorem \ref{thm.multi},
\begin{equation*}
\widehat{V}(g)^{n}_T \overset{\mathbb{P}}{\longrightarrow} V(g)_T,
\end{equation*}
hence we have the following corollary.
\begin{corol}
	Under the conditions of theorem \ref{thm.multi}, on the event $\big\{\widehat{V}(g)^{n}_T\text{ is positive semidefinite}\big\}$,
	\begin{equation*}
		N^{1/2}\big(\widehat{V}(g)^{n}_T\big)^{-1/2}\big[\widehat{S}(g)^{n}_T -S(g)_T\big] \overset{\mathcal{L}}{\longrightarrow} N(0,I).
	\end{equation*}
\end{corol}

\subsection{Asychronicity biases and wave interference}
When different time series are observed synchronous, or the objects are univariate volatility functionals, by theorem \ref{thm.uni} and corollary \ref{corol.syn}, the functional estimators are not only rate-optimal but also efficient by taking $N=\lfloor\underline{n}/2\rfloor-M+1$. 

However, in the presence of asynchronous observations, the condition $N=o(\underline{n}^{4/5})$ in theorem \ref{thm.bi}, \ref{thm.multi} means the convergence rate is strictly less than $\underline{n}^{2/5}$. If we allow the limit distribution to be non-centered, we can improve the convergence rate to be exact $\underline{n}^{2/5}$. To formulate this non-centered asymptotic result, we define ``\textit{cubic variation of time}'' as $P^n_{jk}(t)\coloneqq \underline{n}^2\int_0^t \big[\thb{j}{u} - \thb{k}{u}\big]^2\ds u$, note
\begin{eqnarray*}
	P^n_{jk}(t)
	&=& \underline{n}^2\sum_{I^j_h\cap I^k_v\ne\emptyset} \Big[\big(\tab{j}{k}{h}{v}-\taa{j}{k}{h}{v}\big)^2\,|\ta{j}{h-1}-\ta{k}{v-1}|\, \mathds{1}_{\{\ta{j}{h-1}\wedge\ta{k}{v-1}\le t;\, I^j_h\not\subseteq I^k_v \text{ and } I^k_v\not\subseteq I^j_h\}} \\
	&&\hspace{15mm} + \big(\ta{j}{h}\vee\ta{k}{v}-\taa{j}{k}{h}{v}\big)^2\,|\ta{j}{h-1}-\ta{k}{v-1}|\, \mathds{1}_{\{\ta{j}{h-1}\wedge\ta{k}{v-1}\le t;\, I^j_h\subseteq I^k_v \text{ or } I^k_v\subseteq I^j_h\}} \\
	&&\hspace{15mm} + \big(\tab{j}{k}{h}{v}-\taa{j}{k}{h}{v}\big)\,|\ta{j}{h}-\ta{k}{v}|^2\, \mathds{1}_{\{\taa{j}{k}{h}{v}\le t\}} \Big],
\end{eqnarray*}
and under assumption \ref{a.ST}, $P^n_{jk}(t)=0$ uniformly.

\begin{Assump}{$\Theta$}[cubic variations of time]\label{a.theta}
$\forall j,k=1,\cdots,d$, $\exists$ an integrable function $\varrho_{jk}$, such that $\forall t\in[0,T]$, as $n\to\infty$
\begin{equation*}
	P^n_{jk}(t) \overset{\mathbb{P}}{\longrightarrow} \int_0^t \varrho_{jk}(u)\ds u.
\end{equation*}
\end{Assump}

The next proposition states a limit result with exact rate $\underline{n}^{2/5}$, and the cubic variation of time emerges as the bias in the asymptotic distribution.

\begin{prop}\label{prop.noncenter}
Assume (\ref{def.X}), (\ref{cond.g}), assumption \ref{a.T}, \ref{a.dd}, \ref{a.theta}, \ref{a.U}, and assumption \ref{a.V} with $\alpha>1/2$, $c(0)=c(T)$. Choose the tunning parameters according to (\ref{cond.NM}), (\ref{cond.BL}) with $N=\lfloor \kappa\underline{n}^{4/5}\rfloor\land(\lfloor \underline{n}/2\rfloor-M+1)$,
we have
\begin{equation*}
	\underline{n}^{2/5}\big[\widehat{S}(g)^{n}_T -S(g)_T\big] \overset{\mathcal{L}-s}{\longrightarrow} \mathcal{MN}\big(\mu(g)_T,V(g)_T\big),
\end{equation*}
where $V(g)_T$ is defined as (\ref{def.V(g)}) and
\begin{equation*}
	\mu(g)_T = -\frac{2\pi^2\kappa^{5/2}}{3T^2}\sum_{j,k=1}^d\int_0^T \partial_{jk}g\big(c(t)\big)\, c_{jk}(t)\, \varrho_{jk}(t) \ds t,
\end{equation*}
with $\varrho_{jk}$ being defined in assumption \ref{a.theta}.
\end{prop}

The bias in the second order can be estimated by
\begin{equation*}
	\widehat{\mu}(g)_T = -\frac{2\pi^2\kappa^{5/2}}{3T^2}\sum_{j,k=1}^d\sum_{h=1}^B \partial_{jk}g\big(\chat{}{t_h}\big)\, \chat{jk}{t_h}\, \big[P^n_{jk}(t_h) - P^n_{jk}(t_{h-1})\big].
\end{equation*}

In the asynchronous scenario, if $N=\lfloor\underline{n}/2\rfloor-M+1$, the functional estimator (\ref{def.Shat}) generally is no longer consistent. However, there is still an asymptotic result with optimal convergence rate and a new limit. Before state this result, we need an additional assumption on the almost everywhere convergence of the Dirichlet kernel of time gaps.
\begin{Assump}{D}[Dirichlet kernels of time]\label{a.d}
	$\forall j,k=1,\cdots,d$, $\exists$ an integrable function $r_{jk}$, such that $\forall t\in[0,T]$, as $n\to\infty$
	\begin{equation*}
	\int_0^t d^{n,\lfloor\underline{n}/2\rfloor}_{jk}\big(u,\thb{j}{u}\big)\ds u \overset{\mathbb{P}}{\longrightarrow} \int_0^t r_{jk}(u)\ds u.
	\end{equation*}
\end{Assump}

Define
\begin{eqnarray*}
	\underline{c}^{n,N}(t) &\coloneqq& \big[\dd{jk}{t,t}\,c_{jk}(t)\big]_{jk} \\
	\underline{S}(g)^{n,N}_T &\coloneqq& \int_0^T g\big(\underline{c}^{n,N}(t)\big)\ds t,
\end{eqnarray*}
when $N=\lfloor\underline{n}/2\rfloor-M+1$, under other conditions, $\widehat{S}(g)^n_T - \underline{S}(g)^{n,N}_T$ converges rate-optimally to a mixed normal distribution.
\begin{prop}\label{prop.sqrtn}
Assume (\ref{def.X}), (\ref{cond.g}), assumption \ref{a.T}, \ref{a.dd}, \ref{a.d}, \ref{a.U}, \ref{a.V} with $\alpha>1/2$, $c(0)=c(T)$, and choose the tunning parameters according to (\ref{cond.NM}), (\ref{cond.BL}) with $N=\lfloor \underline{n}/2\rfloor-M+1$, we have
\begin{equation*}
	\underline{n}^{1/2}\big[\widehat{S}(g)^{n}_T - \underline{S}(g)^{n,N}_T\big] \overset{\mathcal{L}-s}{\longrightarrow} \mathcal{MN}\big(0,\underline{V}(g)_T\big),
\end{equation*}
where
\begin{multline*}
	\underline{V}(g)_T = 2\sum_{j,k,l,m=1}^d\int_0^T \partial_{jk} g\big(r\circ c(t)\big)\,\partial_{lm} g\big(r\circ c(t)\big) \times \\ 
	\Big\{\big[\tilde{\theta}_{jk,lm}(t)+\grave{\theta}_{jk,lm}(t)\big] c_{jl}(t)\,c_{km}(t) + \big[\acute{\theta}_{jk,lm}(t)+\check{\theta}_{jk,lm}(t)\big] c_{jm}(t)\,c_{kl}(t)\Big\} \ds t,
\end{multline*}
$r(t)=[r_{jk}(t)]_{jk}$, and $r\circ c$ is the Hadamard product.
\end{prop}

\begin{remk}
	The biases in proposition \ref{prop.noncenter} and proposition \ref{prop.sqrtn} arise from asynchronicity. Note $\Fdx{j}{s}$ defined in (\ref{def.FShat}) can be regarded as a wave function, the multiplication term $\Fdx{j}{q-s}\times\Fdx{k}{s}$ in the spectrum estimator (\ref{def.Fhat}) can be interpreted as a ``superposition'' of two waves. When the observation times of the $j$-th and $k$-th components are asynchronous, the waves $\Fdx{j}{q-s}$ and $\Fdx{k}{s}$ are out of phase. This results in the scaled and shifted Dirichlet kernel $\dd{jk}{t,t}$ and is the source of asynchronicity biases.
\end{remk}

In view of (\ref{rep.Dirichlet}),
\begin{equation*}
\dd{jk}{t,t} = \left\{
\begin{array}{cl}
\frac{\sin\big((2N+1)\pi[\thb{j}{t}-\thb{k}{t}]/T\big)}{(2N+1)\sin\big(\pi[\thb{j}{t}-\thb{k}{t}]/T\big)} & \text{if } \thb{j}{t}-\thb{k}{t} \ne 0, \\
1 & \text{if } \thb{j}{t}-\thb{k}{t} = 0,
\end{array}\right.
\end{equation*}
so we have the following:
\begin{itemize}
	\item $\dd{jj}{t,t}=1,\, \forall j=1,\cdots,d$;
	\item if the $j$-th and $k$-th components are observed synchronously, $\dd{jk}{t,t}=1$;
	\item if the $j$-th and $k$-th components are observed asynchronously but $N$ is chosen in a way such that $N\Delta(n)\to0$, $\dd{jk}{t,t}\overset{\mathbb{P}}{\to}1$.
\end{itemize}
In all the scenarios above, we have $r_{jk}(t)=1$ uniformly over time and $j,k=1,\cdots d$. When the observations are synchronous, i.e., assumption \ref{a.ST} holds, we also have $\underline{S}(g)^{n,N}_T=S(g)_T$. When the observations are asynchronous, however, there is not any workable approach so far to conduct bias correction for general temporal spacings.


\section{Monte Carlo and Empirical Study}
\subsection{Monte Carlo}
We adopt the following simulation model:
\begin{equation*}
\left\{\begin{array}{lcl}
	\mathrm{d}X(t) &=& .03\ds t + \sqrt{c(t)}\ds W(t) \\
	c(t) &=& \widetilde{c}(t) - [\widetilde{c}(T)-\widetilde{c}(0)]\,t/T \\
	\mathrm{d}\widetilde{c}(t) &=& 6(.16-\widetilde{c}(t))\ds t + .5\sqrt{\widetilde{c}(t)}\ds B(t),
\end{array}\right.
\end{equation*}
where $\E[(W_{t+\Delta}-W_t)(B_{t+\Delta}-B_t)]=-.6\Delta$. Each simulation employs $23400\times21$ data points with $\Delta(n)=1s$. 

In the first simulation experiment, we simulate synchronous observations and compute estimators for functionals $g(c)=c^2$, $g(c)=c^{-1}$, $g(c)=\log(c)$ based on the realized variance and the Fourier methods. For all these three functionals, the tunning parameters are $N=\lfloor\underline{n}^{.75}\rfloor$, $M=\lfloor\underline{n}^{.3}\rfloor$, and $k_n=\lfloor\Delta_n^{-.45}\rfloor$ (cf. (3.6) in \cite{jr13}).
The empirical densities of studentized estimators are shown in figure \ref{fig.MC}.
\begin{figure}
	\centering
	\caption{Volatility functional estimators based on Realized Variances \& the Fourier method}\label{fig.MC}
	\includegraphics[width=.5\textwidth]{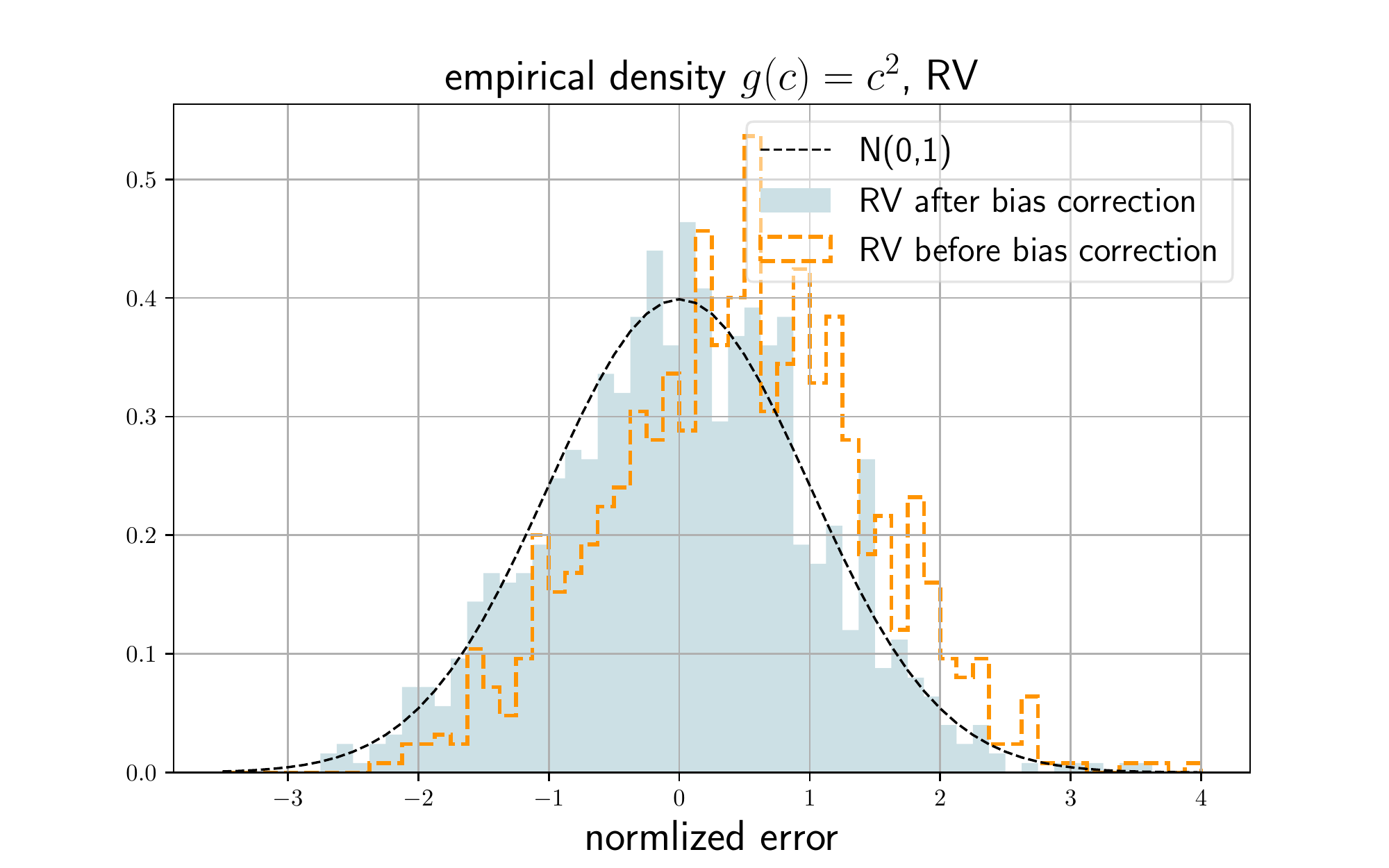}\includegraphics[width=.5\textwidth]{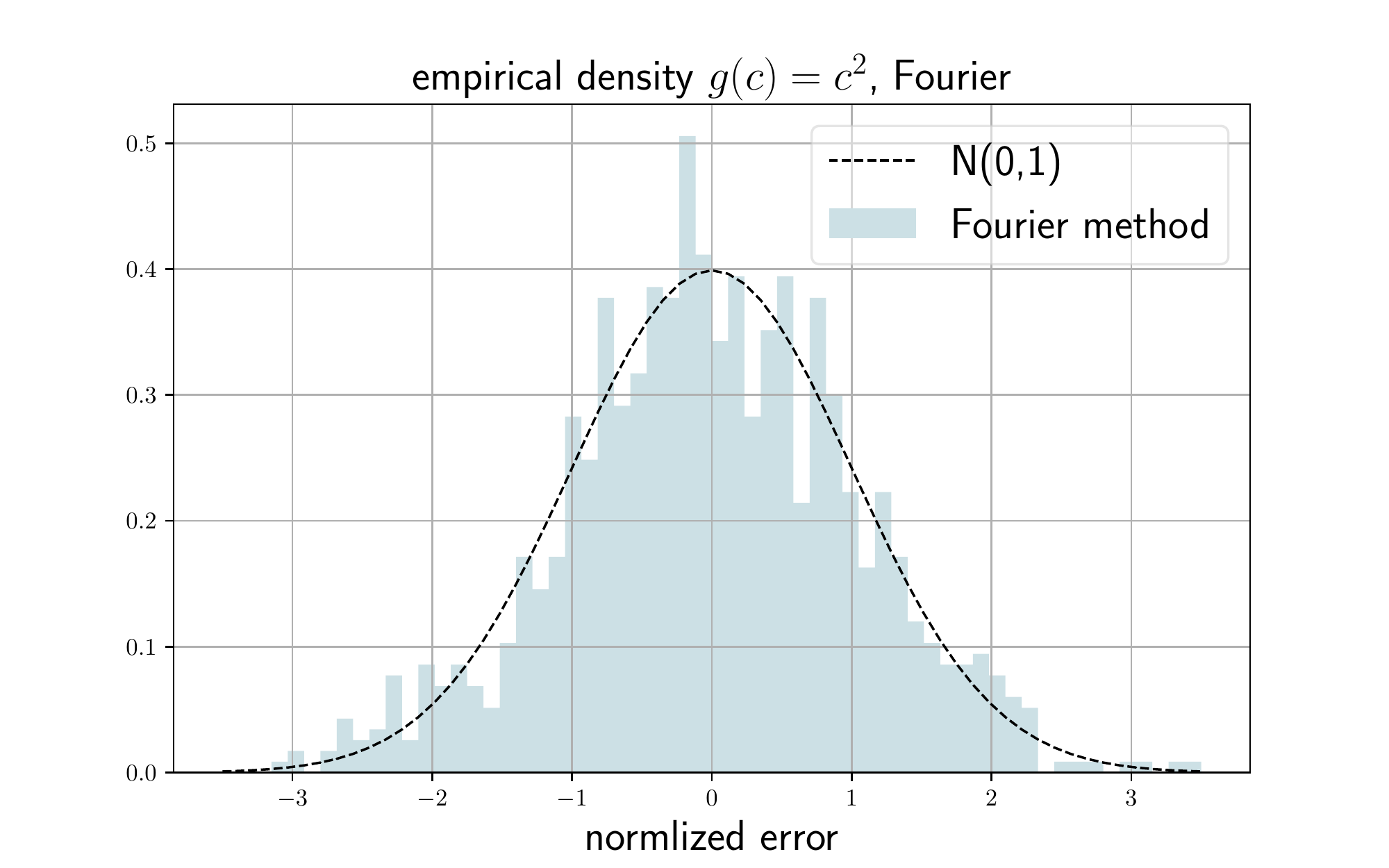}
	\includegraphics[width=.5\textwidth]{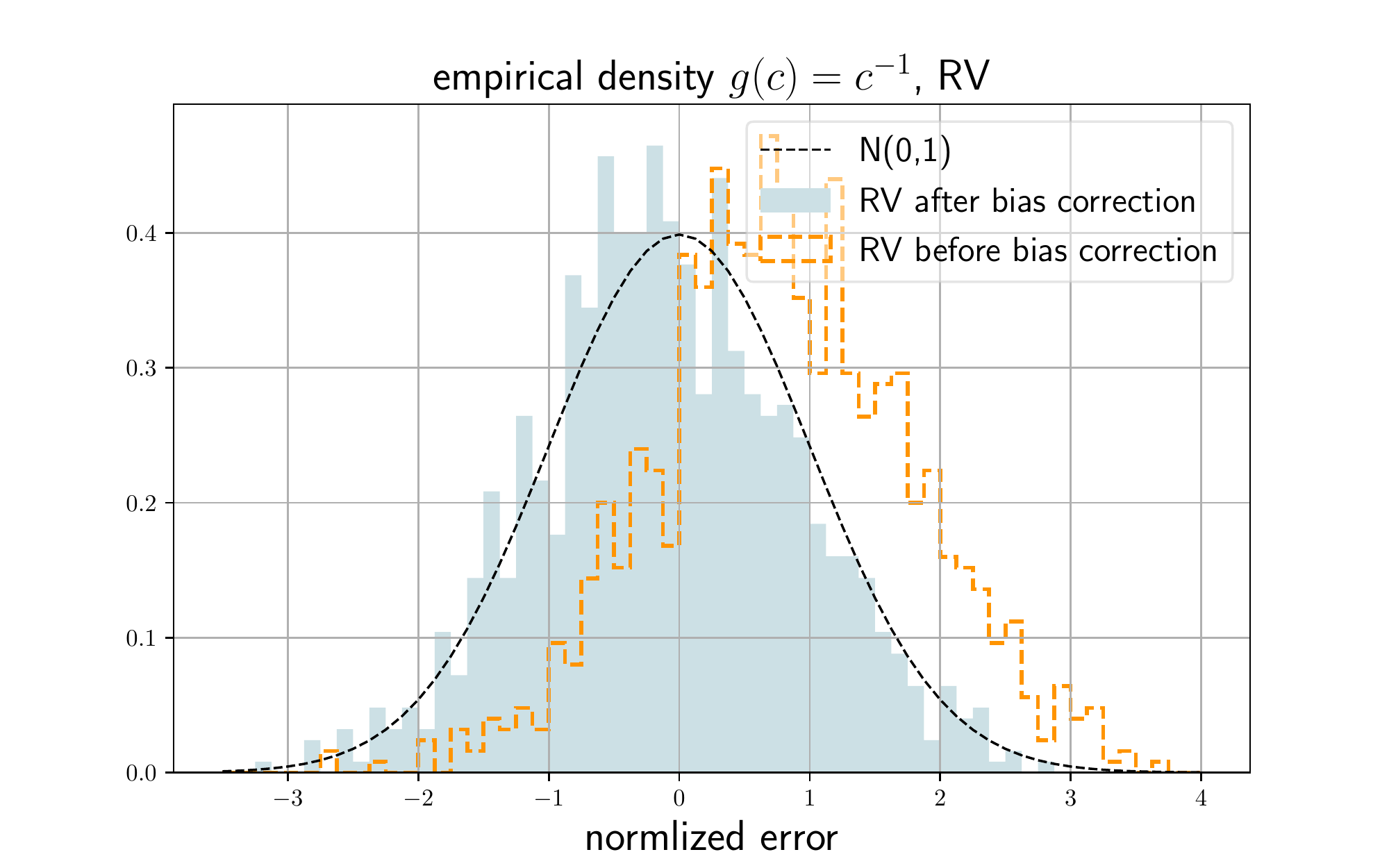}\includegraphics[width=.5\textwidth]{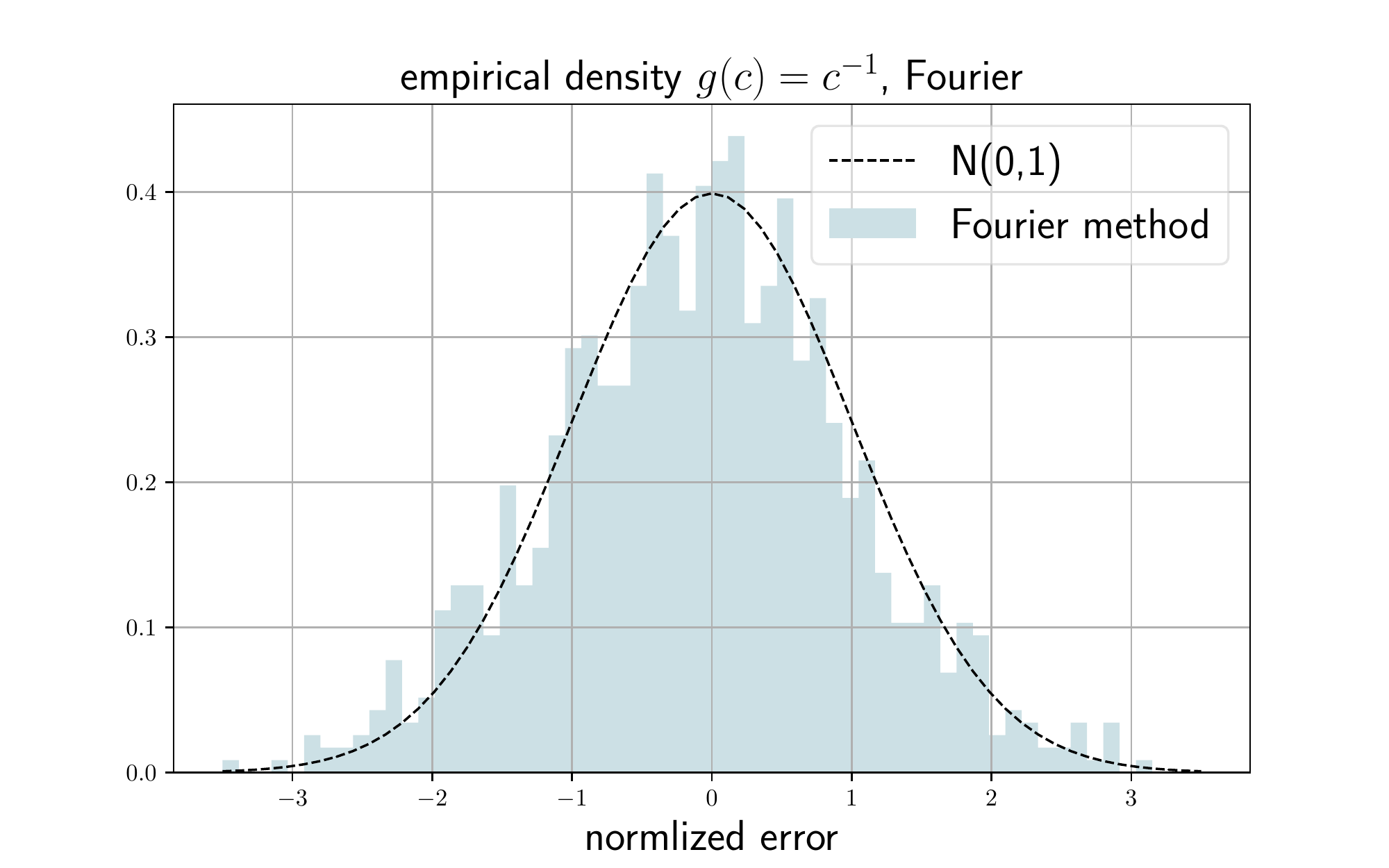}
	\includegraphics[width=.5\textwidth]{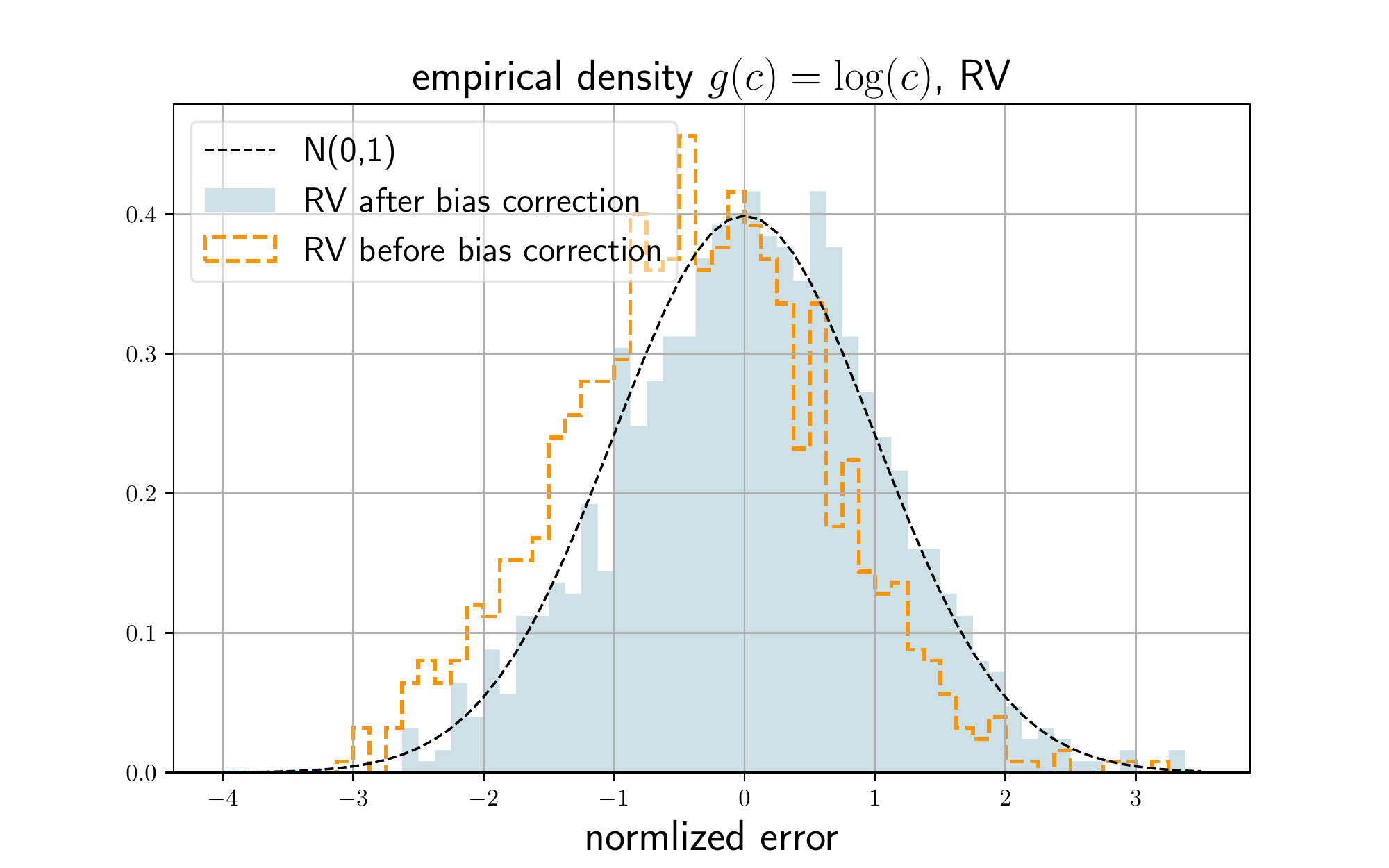}\includegraphics[width=.5\textwidth]{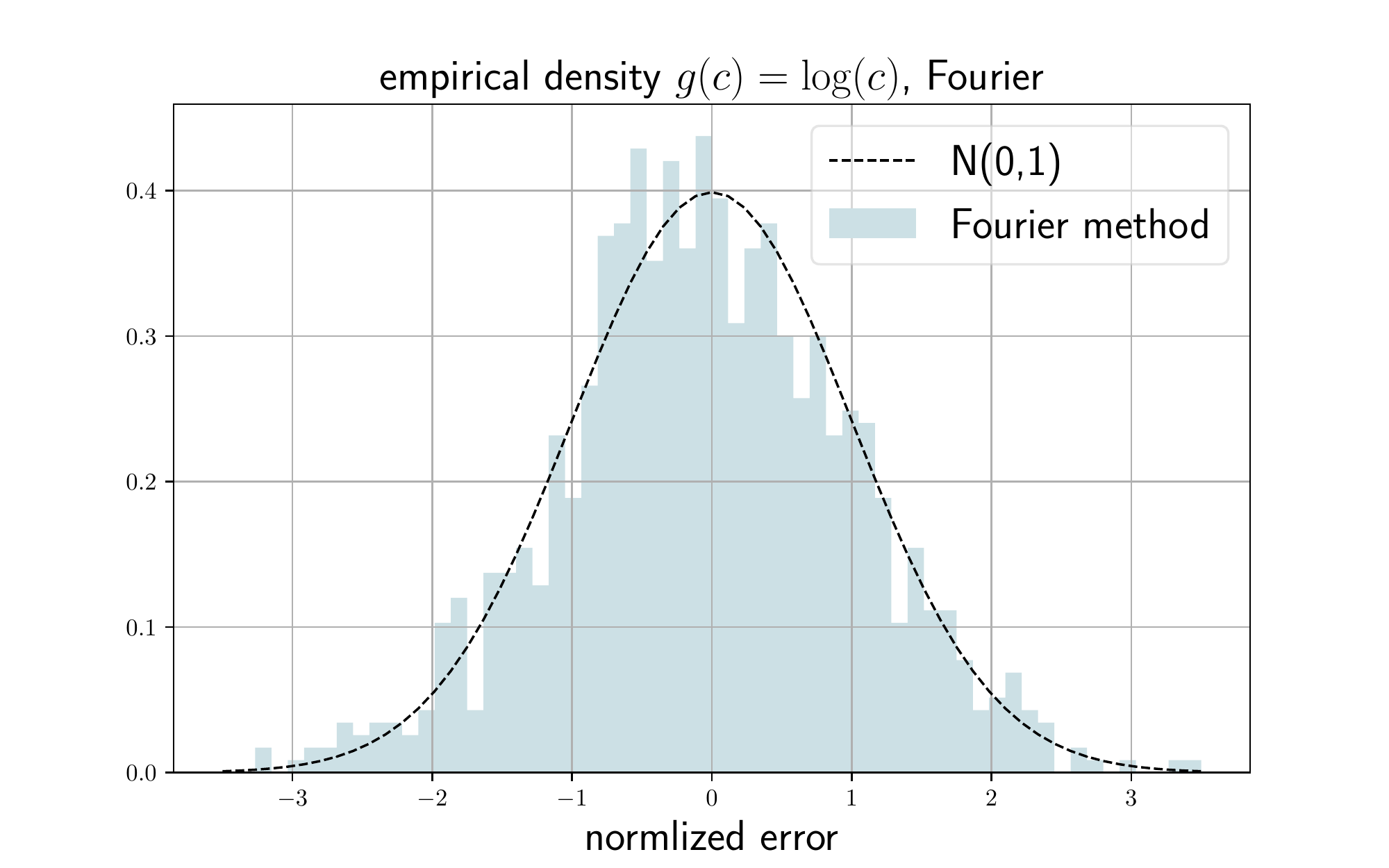}
\end{figure}


\section{Concluding Remark}
Using observations that are of high-frequency and asynchronous, this paper studies the inference problem of volatility functionals based on spot volatility plug-ins. The nonparametric method to estimate spot volatility is based on harmonic analysis. One one hand, the Fourier transform method is numerically more stable for spot volatility estimation than the finite differences of realized variances. On the other hand, more significantly, the frequency-domain operations circumvents the need for time alignment or data imputation.

We first show the consistency of the volatility spectrum estimator through the lens of drift effect, asynchronicity effect and statistical error due to finite Bohr convolution; we then showed the mean-square rate and consistency of the spot volatility estimator based on the result of volatility spectrum estimation and a delta sequence. To establish an inferential theory, we provide the asymptotic distributional results for functionals of one element in the volatility matrix and the functionals of whole volatility matrix. Interestingly, the results reveal how the asynchronicity as a form of noise impacts the convergence rate and the asymptotic variance of the volatility functional estimators.

This paper offers an elegant framework to cope with asynchronous observations and missing data of multiple time series. For the applications of volatility functionals, such as principal component analysis, specification tests, linear regression, now we have a methodological framework with solid statistical guarantees to utilize the more prevailing high-frequency datasets that were asynchronously observed.

\newpage
\appendix
\section{Trigonometric polynomials}\label{apdx:trig}
Now we formally introduce some properties of trigonometric polynomials useful in Fourier analysis.

Based on Dirichlet kernels (\ref{def.Dirichlet}), we define Fej\'er kernel of order $M$ as
\begin{equation}\label{def.Fejer}
    F^M(x) = \frac{1}{M}\sum_{q=0}^{M-1} D^q(x).
\end{equation}
Note for $x\notin\mathbb{N}$,
\begin{multline*}
    MF^M(x) = \sum_{q=0}^{M-1}\frac{\sin[\pi(2q+1)x]}{\sin(\pi x)} = \frac{1}{\sin(\pi x)}\mathrm{Im}\Big(\sum_{q=0}^{M-1}e^{i2\pi(q+1/2)x}\Big) \\
    = \frac{1}{\sin(\pi x)}\mathrm{Im}\Big(\frac{e^{i2\pi Mx}-1}{e^{i\pi x}-e^{-i\pi x}}\Big) = \frac{1-\cos(2\pi Mx)}{2\sin(\pi x)^2},
\end{multline*}
hence we have
\begin{equation*}
    F^M(x) = \left\{
\begin{array}{cc}
    \frac{\sin(\pi Mx)^2}{M\sin(\pi x)^2}, & x\notin\mathbb{N} \\
    M, & x\in\mathbb{N},
\end{array}\right.
\end{equation*}
and by (\ref{rep.Dirichlet})
\begin{equation}\label{rep.Fejer}
    F^{2M+1}(x) = \frac{1}{2M+1}D^M(x)^2.
\end{equation}

According to (\ref{def.Dirichlet}) and (\ref{def.Fejer}), $\forall M\in\mathbb{N}^+$, 
and an interval $I$ with $|I|=T$,
\begin{equation}\label{Fejer.integral}
	\int_I F^M(x/T)\ds x = \frac{1}{M}\sum_{q=0}^{M-1}\int_I\sum_{|s|\le q}\ei{s}{x}\ds x = \frac{1}{M}\sum_{q=0}^{M-1}\int_I\ds x = T.
\end{equation}

Furthermore, we have the following result,
\begin{equation}\label{Jackson}
	\frac{1}{M}\int_0^T F^M\Big(\frac{t}{T}\Big)^2\ds t \to \frac{2T}{3},
\end{equation}
cf. remark 5.2 in \cite{ct15}.

What are so interesting about Dirichlet kernels and Fej\'er kernel in Fourier analysis? Given a function $f$ on $[0,T]$, define its \textit{truncated Fourier inversion} as $\widebar{f}^q(x) = T^{-1}\sum_{|s|\le q}F(f)_s\,e^{i2\pi sx/T}$. One can express $\widebar{f}^q$ as the convolution between $f$ and the Dirichlet kernel of order $q$,
\begin{equation*}
	\widebar{f}^q(x) = \frac{1}{T}\int_0^T f(u)\sum_{|s|\le q}\ei{s}{(x-u)}\ds u = \frac{1}{T}\int_0^T f(u)D^q[(x-u)/T]\ds u.
\end{equation*}
Recall $\widehat{f}$ defined in (\ref{def.fhat}), we have, by (\ref{def.Fejer}), 
\begin{equation}\label{rep.fhat}
	\widehat{f}^M(t) = \frac{1}{TM}\sum_{q=0}^{M-1}\sum_{|s|\le q}F(f)_s\,\ei{s}{t} = \frac{1}{M}\sum_{q=0}^{M-1}\widebar{f}^q(t) = \frac{1}{T}\int_0^T f(u)F^M[(t-u)/T]\ds u.
\end{equation}

\begin{lem}\label{lem.Fejer.uc}
If the function $f$ is continuous, then
\begin{eqnarray*}
	\sup_{t\in[0,T]}\left|\frac{1}{T}\int_0^TF^M\Big(\frac{t-u}{T}\Big)\,f(u)\ds u - f(t)\right| \le K\omega_f(1/M), && \text{if } f(0)=f(T); \\
	\sup_{t\in[1/M,T-1/M]}\left|\frac{1}{T}\int_0^TF^M\Big(\frac{t-u}{T}\Big)\,f(u)\ds u - f(t)\right| \le K\omega_f(1/M), && \text{if } f(0)\ne f(T).
\end{eqnarray*}
\end{lem}
\begin{proof}	
By (\ref{Fejer.integral}),
\begin{equation*}
	\delta^M(t) \coloneqq T\left|\frac{1}{T}\int_0^TF^M\Big(\frac{t-u}{T}\Big)\,f(u)\ds u - f(t)\right| 
	= \left|\int_0^T F^M\Big(\frac{t-u}{T}\Big) \big[f(t)-f(u)\big]\ds u\right|.
\end{equation*}
	
\textbf{(1)} If $f(0)=f(T)$, by periodization, we can extend the definition of $f$ to the real line and retain its modulus of continuity. By a change of variable and the periodicity of $F^M$ and $f$,
\begin{multline*}
	\delta^M(t) = \left|\int_{t-T}^t F^M(u/T) \big[f(t)-f(t-u)\big]\ds u\right| \le \int_{-T/2}^{T/2} F^M(u/T)\cdot|f(t)-f(t-u)|\ds u\\ 
	= \Big(\int_{-1/M}^{1/M} + \int_{1/M\le|u|\le T/2}\Big) F^M(u/T)\cdot|f(t)-f(t-u)|\ds u.
\end{multline*}
since $F^M(x) \le \frac{M}{1+M^2x^2},\, x\in[-1/2,1/2]$,
\begin{multline*}
	\int_{1/M\le|u|\le T/2} F^M(u/T)\cdot|f(t)-f(t-u)|\ds u
	\le \frac{K}{M}	\int_{1/M\le|u|\le T/2} |u|^{\alpha-2}\ds u \le K\big[\omega_f(M^{-1}) + M^{-1}\big].
\end{multline*}
moreover,
\begin{equation*}
	\int_{-1/M}^{1/M} F^M(u/T)\cdot|f(t)-f(t-u)|\ds u \le K\omega_f(M^{-1})\int_{-1/M}^{1/M} F^M(u/T)\ds u \le K\omega_f(M^{-1}),
\end{equation*}
then this lemma in the case $f(0)=f(T)$ is proved.
	
\textbf{(2)} If $f(0)\ne f(T)$, then for $t\in[1/M,T-1/M]$,
\begin{equation*}
	\delta^M(t) \le \Big(\int_0^{t-1/M} + \int_{t-1/M}^{t+1/M} + \int_{t+1/M}^T\Big) F^M\Big(\frac{t-u}{T}\Big)\cdot|f(t)-f(u)|\ds u,
\end{equation*}
then by a similar argument, the lemma in the case $f(0)\ne f(T)$ can also be proved.
\end{proof}

The following lemma is similar to Lemma 5.1 in \cite{ct15}. The proof emulates that of \cite{ct15}, but we generalize the result to irregular observations, replace the time period $2\pi$ with $T$, and change the conclusion for our own need.
\begin{lem}\label{lem.Fejer.Riemann}
If $M=o(n)$, for $\thb{j}{t}$ defined in (\ref{def.theta_n}), $\forall f\in C([0,T])$, $\forall t\in[0,T]$, 
\begin{equation*}
	\left|\int_0^T F^M\Big(\frac{t-\thb{j}{u}}{T}\Big)\,f(u)\ds u - \int_0^T F^M\Big(\frac{t-u}{T}\Big)\,f(u)\ds u\right| \le KT \frac{M}{n}.
\end{equation*}
\end{lem}
\begin{proof}
Denote the L.H.S. by $D(t)^{n,M}_{j,T}$, note
\begin{multline*}
	D(t)^{n,M}_{j,T}\le K\int_0^T \Big|F^M\Big(\frac{t-\thb{j}{u}}{T}\Big) - F^M\Big(\frac{t-u}{T}\Big)\Big|\ds u \\
	= K\sum_{h=1}^{n_j}\int_{I^j_h} \Big|F^M\Big(\frac{t-\ta{j}{h}}{T}\Big) - F^M\Big(\frac{t-u}{T}\Big)\Big|\ds u = K\sum_{h=1}^{n_j}\Delta^j_h \Big|F^M\Big(\frac{t-\ta{j}{h}}{T}\Big) - F^M\Big(\frac{t-u^j_h}{T}\Big)\Big|,
\end{multline*}
where $u^j_h\in I^j_h$ for each $h$ by mean value theorem.
	
Let $J_b=\Big(\big(t+\frac{bT}{M}\big)\vee0,\big(t+\frac{(b+1)T}{M}\big)\wedge T\Big]$, $B_0=\inf\big\{b,t+\frac{(b+1)T}{M}>0\big\}$, $B_1=\sup\big\{b,t+\frac{bT}{M}<T\big\}$, then
	\[ D(t)^{n,M}_{j,T} \le K\sum_{b=B_0}^{B_1}\sum_{\ta{j}{h}\in J_b}\Delta^j_h \Big|F^M\Big(\frac{t-\ta{j}{h}}{T}\Big) - F^M\Big(\frac{t-u^j_h}{T}\Big)\Big|. \]
Based on mean value theorem, $\exists v^j_h\in[u^j_h,\ta{j}{h}]$ for each $h$ such that $F^M\Big(\frac{t-\ta{j}{h}}{T}\Big) - F^M\Big(\frac{t-u^j_h}{T}\Big) = (\ta{j}{h}-v^j_h)\,\partial F^M\Big(\frac{t-v^j_h}{T}\Big)$, so
\begin{multline*}
	\sum_{\ta{j}{h}\in J_b}\Delta^j_h \Big|F^M\Big(\frac{t-\ta{j}{h}}{T}\Big) - F^M\Big(\frac{t-u^j_h}{T}\Big)\Big| \le \Delta(n)^2 \sum_{\ta{j}{h}\in J_b}\Big|\partial F^M\Big(\frac{t-v^j_h}{T}\Big)\Big| \\
	\le K \frac{n\Delta(n)^2}{M} \sup_{v\in J_b}\Big|\partial F^M\Big(\frac{t-v}{T}\Big)\Big|.
\end{multline*}
	
Based on (\ref{rep.Fejer}),
	\[ \sup_{v\in J_b}\Big|\partial F^M\Big(\frac{t-v}{T}\Big)\Big| \le KM\sup_{v\in J_b}F^M\Big(\frac{t-v}{T}\Big) \le \widetilde{K}M^2\int_{J_b} F^M\Big(\frac{t-u}{T}\Big)\ds u, \]
thus we have
\begin{equation*}
	D(t)^{n,M}_{j,T} \le K Mn\Delta(n)^2 \sum_{b=B_0}^{B_1}\int_{J_b} F^M\Big(\frac{t-u}{T}\Big)\ds u \asymp KT\frac{M}{n},
\end{equation*}
from which this lemma follows.	
\end{proof}

The following lemma is a modified adaptation of Lemma 3 in \cite{cg11}.	We generalize it to time period $T$ and a more general $N$. Its purpose is to investigate the $L^p$ norm of the shifted and scaled Dirichlet kernel. 
\begin{lem}\label{lem.dp}
Assume assumption \ref{a.T}, we have for $p>1$, $N\le\lfloor (n_j\land n_k)/2\rfloor$, $\exists K_p<\infty$,  
\begin{equation*}
	\sup_{j,k}\sup_{t\in[0,T]} \int_0^T \big|\dd{jk}{t,u}\big|^p\ds u \le K_pN^{-1}
\end{equation*}
\end{lem}
\begin{proof}
By the definition (\ref{def.dirichlet.scale}) and the shape of the Dirichlet kernel, it suffices to study
	\[ \sup_j\sup_{a\in[0,T]} \int_{(a-T/2)\vee0}^{(a+T/2)\wedge T} \Big|\frac{1}{2N+1}D^N\Big(\frac{\thb{j}{t}-a}{T}\Big)\Big|^p\ds t \le K_pN^{-1}. \]
Note that
\begin{equation*}
	\Big|\frac{1}{2N+1}D^N(x/T)\Big| \le 1\land\frac{2T}{(2N+1)|x|},
\end{equation*}
and $\forall a\in[0,T]$, $\forall j=1\cdots,d$, $|t-a|>|\thb{j}{t}-a|-|\thb{j}{t}-t|$, thus
\begin{multline*}
	\int_{(a-T/2)\vee0}^{(a+T/2)\wedge T} \Big|\frac{1}{2N+1}D^N\Big(\frac{\thb{j}{t}-a}{T}\Big)\Big|^p\ds t \\
	\le \Big(\int_{(a-T/2)\vee0}^{\big[a-\frac{2T}{2N+1}-\Delta(n)\big]\vee0}+\int_{\big[a+\frac{2T}{2N+1}+\Delta(n)\big]\wedge T}^{(a+T/2)\wedge T}\Big) \Big|\frac{1}{2N+1}D^N\Big(\frac{\thb{j}{t}-a}{T}\Big)\Big|^p\ds t + \frac{4T}{2N+1}+2\Delta(n) \\
	\le \Big(\int_{(a-T/2)\vee0}^{\big[a-\frac{2T}{2N+1}-\Delta(n)\big]\vee0}+\int_{\big[a+\frac{2T}{2N+1}+\Delta(n)\big]\wedge T}^{(a+T/2)\wedge T}\Big) \Big|\frac{2T}{(2N+1)(t-a)}\Big|^p\ds t + \frac{4T}{2N+1}+2\Delta(n),
\end{multline*}
via a change of variable,
\begin{equation*}
	\Big(\int_{(a-T/2)\vee0}^{\big[a-\frac{2T}{2N+1}-\Delta(n)\big]\vee0}+\int_{\big[a+\frac{2T}{2N+1}+\Delta(n)\big]\wedge T}^{(a+T/2)\wedge T}\Big) \Big|\frac{2T}{(2N+1)(t-a)}\Big|^p\ds t \le \frac{4T}{2N+1}\int_1^\infty x^{-p}\ds x,
\end{equation*}
thereby this lemma is proved.
\end{proof}

For $j,k,l,m=1\cdots,d$, by Fubini's theorem and H\"older's inequality,
\begin{multline*}
	N^2\int_0^T\int_0^T\ds t\ds u\, \Big[\int_0^{t\wedge u} \dd{jk}{t,v}\,\dd{jk}{u,v}\ds v \int_0^v \dd{lm}{v,\vartheta}^2\ds \vartheta\Big] \\
	\le N^2\int_0^T\ds v\, \Big[\int_{v}^T\dd{jk}{t,v}\ds t\, \int_v^T\dd{jk}{u,v}\ds u \int_0^v \dd{lm}{v,\vartheta}^2\ds \vartheta \Big]\\
	\le N^2\int_0^T\ds v\, \Big(\int_{v}^T\big|\dd{jk}{t,v}\big|\ds t\Big)^2 \Big(\int_0^v \big|\dd{lm}{v,\vartheta}\big|^2\ds \vartheta\Big) \\
	\le T^{\frac{3p-2}{p}} N\Big(\int_0^T\big|\dd{jk}{t,v}\big|^p\ds t\Big)^{2/p} \Big(N\int_0^T\big|\dd{lm}{v,\vartheta}\big|^2\ds \vartheta\Big).
\end{multline*}
according to lemma \ref{lem.dp}, we have for $p>1$,
\begin{multline}\label{bdd.dd}
	\sup_{j,k,l,m} N^2\int_0^T\int_0^T\ds t\ds u\, \Big[\int_0^{t\wedge u} \dd{jk}{t,v}\,\dd{jk}{u,v}\ds v \int_0^v \dd{lm}{v,\vartheta}^2\ds \vartheta\Big] 
	\le T^{\frac{3p-2}{p}} N^{1-2/p}
\end{multline}

The following lemma follows from assumption \ref{a.dd} and is a straightforward generalization of Lemma 4 in \cite{cg11} from the bivariate setting to the multivariate setting.
\begin{lem}\label{lem.dd}
Under assumption \ref{a.T}, \ref{a.dd}, then $\forall f_0,f_1\in C([0,T])$, we have $\forall t\in[0,T]$,
\begin{eqnarray*}
	\int_0^t f_0(u)\ds u\, N\int_0^u \dd{jk}{u,v}\,\dd{lm}{u,v}\, f_1(v)\ds v &\overset{\mathbb{P}}{\longrightarrow}& \int_0^t \tilde{\theta}_{jk,lm}(u)\, f_0(u)f_1(u) \ds u \\
	\int_0^t f_0(u)\ds u\, N\int_0^u \dd{jk}{u,v}\,\dd{lm}{v,u}\, f_1(v)\ds v &\overset{\mathbb{P}}{\longrightarrow}& \int_0^t \acute{\theta}_{jk,lm}(u)\, f_0(u)f_1(u) \ds u \\
	\int_0^t f_0(u)\ds u\, N\int_0^u \dd{jk}{v,u}\,\dd{lm}{u,v}\, f_1(v)\ds v &\overset{\mathbb{P}}{\longrightarrow}& \int_0^t \check{\theta}_{jk,lm}(u)\, f_0(u)f_1(u) \ds u \\
	\int_0^t f_0(u)\ds u\, N\int_0^u \dd{jk}{v,u}\,\dd{lm}{v,u}\, f_1(v)\ds v &\overset{\mathbb{P}}{\longrightarrow}& \int_0^t \grave{\theta}_{jk,lm}(u)\, f_0(u)f_1(u) \ds u.
\end{eqnarray*}	
\end{lem}

The following lemma reveals the limiting behavior of the shifted and scaled Dirichlet kernels when the temporal spacings are synchronous (but possibly irregular) across different dimensions.
\begin{lem}\label{lem.dd.syn}
Assume $n_1=n_2=\cdots=n_d$ and $\min_j\ta{j}{h}=\max_j\ta{j}{h}$ for $h=1,\cdots, n_1$, then $\forall t\in[0,T]$ and $\forall f_0,f_1\in C([0,T])$,
\begin{eqnarray*}
	\int_0^T f_0(u)\ds u\, N\int_0^u \dd{jk}{u,v}\,\dd{lm}{u,v}\, f_1(v) \ds v &\overset{\mathbb{P}}{\longrightarrow}& \frac{T}{4}\int_0^T f_0(u)f_1(u) \ds u \\
	\int_0^T f_0(u)\ds u\, N\int_0^u \dd{jk}{u,v}\,\dd{lm}{v,u}\, f_1(v) \ds v &\overset{\mathbb{P}}{\longrightarrow}& \frac{T}{4}\int_0^T f_0(u)f_1(u) \ds u \\
	\int_0^T f_0(u)\ds u\, N\int_0^u \dd{jk}{v,u}\,\dd{lm}{u,v}\, f_1(v) \ds v &\overset{\mathbb{P}}{\longrightarrow}& \frac{T}{4}\int_0^T f_0(u)f_1(u) \ds u \\
	\int_0^T f_0(u)\ds u\, N\int_0^u \dd{jk}{v,u}\,\dd{lm}{v,u}\, f_1(v) \ds v &\overset{\mathbb{P}}{\longrightarrow}& \frac{T}{4}\int_0^T f_0(u)f_1(u) \ds u.
\end{eqnarray*}
\end{lem}
\begin{proof}
Because of synchronous observations, $\dd{jk}{u,v}\,\dd{lm}{u,v} = \dd{11}{u,v}^2$, then by (\ref{def.dirichlet.scale}) and (\ref{rep.Fejer}),
\begin{multline*}
	\int_0^Tf_0(u)\ds u\, N\int_0^u \dd{11}{u,v}^2\, f_1(v) \ds v = \frac{N}{2N+1}\int_0^T f_0(u)\ds u \int_0^u F^{2N+1}\Big(\frac{\thb{j}{u}-\thb{k}{v}}{T}\Big) f_1(v)\ds v,
\end{multline*}
by Riemann summation,
\begin{equation*}
	\int_0^u F^{2N+1}\Big(\frac{\thb{j}{u}-\thb{k}{v}}{T}\Big) f_1(v)\ds v = \int_0^u F^{2N+1}\Big(\frac{u-v}{T}\Big) f_1(v)\ds v + O_p(n^{-1}),
\end{equation*}
via changes of variables,
\begin{equation*}
	\int_0^T f_0(u)\ds u \int_0^u F^{2N+1}\Big(\frac{u-v}{T}\Big) f_1(v)\ds v = T^2\int_0^1 f_0(Tu)\ds u \int_0^u F^{2N+1}(u-v) f_1(Tv)\ds v,
\end{equation*}
note that $F^{2N+1}$ is a delta sequence, as $N\to\infty$, $\int_0^u F^{2N+1}(u-v) f_1(Tv)\ds v \to f_1(Tu)/2$, then this lemma follows from a change of variable.
\end{proof}

\section{Proof of proposition 1}\label{apdx:spec}
In all the proofs of this paper, $K$ represents a positive finite real number, and may vary from line to line.

By assumption \ref{a.U} and a localization argument (cf. section 4.4.1 in \cite{jp12}), without loss of generality we can assume a stronger assumption in all the following proofs:
\begin{Assump}{SU}[global boundedness]\label{a.SU}
The spot volatility matrix $c$ has continuous sample path almost surely. Moreover, there exists a finite constant $K$ and a compact subset $\s$ of positive semidefinite matrices such that
	\[ \|b(t)\|+\|c(t)\|\le K,\, c(t)\in\s,\; \forall t\in[0,T].\]
\end{Assump}

\noindent\textbf{Study of $R(1)^{n,N}$}\\
We have
\begin{equation*}
    \Fdm{j}{q} - F(\mathrm{d}M_j)_q = \int_0^T \beta^n_{j,q}(t)\ds M_j(t)
\end{equation*}
where
\begin{equation}\label{def.beta}
	\beta^n_{j,q}(t) = \sum_{h=1}^{n^j}\ee{q}{\tau^j_{h}}\big[1-\ee{q}{(t-\tau^j_{h})}\big]\mathds{1}_{I^j_h}(t).
\end{equation}
note $|\beta^n_{j,q}(t)|\le KT^{-1}q\Delta(n)$, by \textit{Burkholder-Davis-Gundy inequality},
\begin{eqnarray*}
    \E\big(|\Fdm{j}{q-s}|^4\big) &\le& KT^2 \\
    \E\big(|F(\mathrm{d}M_j)_q|^4\big) &\le& KT^2 \\
    \E\big(|\Fdm{j}{q} - F(\mathrm{d}M_j)_q|^4\big) &\le& KT^{-2}q^4\Delta(n)^4
\end{eqnarray*}
by \textit{Cauchy-Schwarz inequality},
\begin{multline*}
    \E\big(|\Fdm{j}{q-s}\Fdm{k}{s} - F(\mathrm{d}M_j)_{q-s}F(\mathrm{d}M_k)_s|^2\big)\\
    \le 2\Big[ \E\big(|\Fdm{j}{q-s}|^4\big)^{1/2}\cdot\E\big(|\Fdm{k}{s} - F(\mathrm{d}M_k)_s|^4\big)^{1/2}\\ + \E\big(|F(\mathrm{d}M_k)_s|^4\big)^{1/2}\cdot\E\big(|\Fdm{j}{q-s} - F(\mathrm{d}M_j)_{q-s}^n|^4\big)^{1/2} \Big]
\end{multline*}
so by \textit{Jensen's inequality},
\begin{equation}\label{Fhat2-F2}
    \E\Big(\big|R(1)^{n,N}_{jk,q}\big|\Big) \le K N\Delta(n)
\end{equation}

\noindent\textbf{Study of $R(2)^N$}\\
Define a $\mathbb{C}$-valued martingale $\Gamma_q^j(t) = \int_0^t \ee{q}{u}\ds M_j(u)$ for $j=1,\cdots,d$.  We see that $\Gamma_q^j(T) = F(\mathrm{d}M_j)_q$. By It\^o's formula,
\begin{equation*}
    \Gamma_{q-s}^j(T)\cdot\Gamma_s^k(T) = F(c_{jk})_q + \int_0^T\Gamma_{q-s}^j(t)\ds\Gamma_s^k(t) + \int_0^T\Gamma_s^k(t)\ds\Gamma_{q-s}^j(t)
\end{equation*}
hence
\begin{eqnarray*}
    R(2)^N_{jk,q} = \Lambda(1)^N_{jk,q} + \Lambda(2)^N_{jk,q}
\end{eqnarray*}
where 
\begin{equation*}
\begin{array}{lcl}
    \Lambda(1)^N_{jk,q} &=& \frac{1}{2N+1}\sum_{|s|\le N}\int_0^T\Gamma_{q-s}^j(t)\ds\Gamma_s^k(t) \\
    \Lambda(2)^N_{jk,q} &=& \frac{1}{2N+1}\sum_{|s|\le N}\int_0^T\Gamma_s^k(t)\ds\Gamma_{q-s}^j(t)
\end{array}
\end{equation*}
By (\ref{def.Dirichlet}), we have
\begin{eqnarray*}
    \Lambda(1)^N_{jk,q} &=& \int_0^T\sigma_{k\cdot}(t)\ds W(t) \,\int_0^t \ee{q}{u}\frac{1}{2N+1}D^N\Big(\frac{u-t}{T}\Big)\sigma_{j\cdot}(t)\ds W(u)\\
    \Lambda(2)^N_{jk,q} &=& \int_0^T\ee{q}{t}\sigma_{j\cdot}(t)\ds W(t) \,\int_0^t\frac{1}{2N+1}D^N\Big(\frac{u-t}{T}\Big)\sigma_{k\cdot}(t)\ds W(u)
\end{eqnarray*}
by \textit{It\^o isometry}\footnote{a.k.a. It\^o energy identity.}, (\ref{rep.Fejer}) and (\ref{Fejer.integral}),
\begin{multline*}
    \E\Big(|\Lambda(1)^N_{jk,q}|^2\Big) = \E\left[\int_0^T \Big(\int_0^t \ee{q}{u}\frac{1}{2N+1}D^N\Big(\frac{u-t}{T}\Big)\ds X_j(u)\Big)^2 c_{kk}(t)\ds t \right] \\
    \le \frac{K}{2N+1}\int_0^T \int_0^t F^{2N+1}\Big(\frac{u-t}{T}\Big) \ds u\ds t \asymp \frac{KT^2}{N} 
\end{multline*}
the term $\Lambda(2)^N_{jk,q}$ can be bounded by a similar argument, so
\begin{equation}\label{F2-Fc}
    \E\Big(\big|R(2)^N_{jk,q}\big|\Big) \le KTN^{-1/2}
\end{equation}

\noindent\textbf{Study of $R(0)^{n,N}$}\\
For a generic scalar process, we can write $\widehat{F}(\mathrm{d}U)^n_q=\int_0^T e^n_{j,q}(t)\ds U(t)$ where $e^n_{j,q}(t)=\sum_{h=1}^{n^j}\ee{q}{\tau^j_{h}}\mathds{1}_{I^j_h}(t)$, for $j=1,\cdots,d$.

By linearity of discrete Fourier transform, $\Fdx{j}{q} = \widehat{F}(\mathrm{d}A^j)_q^n + \Fdm{j}{q}$, so 
\begin{multline*}
    \Fdx{j}{q-s}\Fdx{k}{s} - \Fdm{j}{q-s}\Fdm{k}{s} = \\
    \Fda{j}{q-s}\Fda{k}{s} + \Fda{j}{q-s}\Fdm{k}{s} + \Fda{k}{s}\Fdm{j}{q-s}
\end{multline*}

By Parseval's identity,
\begin{equation*}
    \sum_{s=-\infty}^\infty |F(\mathrm{d}A_j)_s|^2 = \int_0^T|b_j(t)|^2\ds t <\infty
\end{equation*}
note
\begin{equation*}
    \Fda{j}{q} - F(\mathrm{d}A_j)_q = \int_0^T \beta^n_{j,q}(t)\, b_j(t)\ds t
\end{equation*}
where $\beta^n_{j,q}(t)$ is defined in (\ref{def.beta}).
By Cauchy-Schwarz inequality,
\begin{multline*}
    \big|R(0)^{n,N}_{jk,q}\big| \le \frac{K}{2N+1}\int_0^T\|b(t)\|^2\ds t + \Big(\frac{K}{2N+1}\int_0^T\|b(t)\|^2\ds t\Big)^{1/2}\times\\
    \Bigg[\Big(\frac{1}{2N+1}\sum_{|s|\le N}|F(\mathrm{d}M_k)_s|^2\Big)^{1/2} + \Big(\frac{1}{2N+1}\sum_{|s|\le N}|F(\mathrm{d}M_j)_{q-s}|^2\Big)^{1/2} \Bigg]
\end{multline*}

From the study of the term $R(2)^N_{jk,q}$, we know
\begin{equation*}
    \sum_{|s|\le N}F(\mathrm{d}M_k)_s^2 = \int_0^TD^N\Big(\frac{2t}{T}\Big)\,c_{kk}(t)\ds t + 2\int_0^T\sigma_{k,\cdot}(t)\ds W(t) \int_0^tD^N\Big(\frac{t+u}{T}\Big)\,\sigma_{k,\cdot}(u)\ds W(u)
\end{equation*}
by Cauchy-Schwarz inequality, (\ref{rep.Fejer}), (\ref{Fejer.integral})
\begin{equation*}
    \frac{1}{2N+1}\int_0^TD^N\Big(\frac{2t}{T}\Big)\,c_{kk}(t)\ds t \\
    \le \frac{1}{\sqrt{2N+1}}\Big(\int_0^TF^{2N+1}\Big(\frac{2t}{T}\Big)\ds t\Big)^{1/2} \Big(\int_0^Tc_{kk}(t)^2\ds t\Big)^{1/2} \le KTN^{-1/2}
\end{equation*}
by Jensen's inequality, Burkholder-Gundy inequality, (\ref{rep.Fejer}), (\ref{Fejer.integral}),
\begin{multline*}
    \E\Big[\frac{1}{2N+1}\int_0^T\sigma_{k\cdot}(t)\ds W(t) \int_0^tD^N\Big(\frac{t+u}{T}\Big)\sigma_{k\cdot}(u)\ds W(u)\Big] \\
    \le \frac{1}{2N+1}\E\Big[\Big(\int_0^T\sigma_{k\cdot}(t)\ds W(t) \int_0^tD^N\Big(\frac{t+u}{T}\Big)\sigma_{k\cdot}(u)\ds W(u)\Big)^2\Big]^{1/2} \\
    \le \frac{K}{\sqrt{2N+1}}\Big[\int_0^T\int_0^tF^{2N+1}\Big(\frac{t+u}{T}\Big)\ds u\ds t\Big]^{1/2} \le KTN^{-1/2}
\end{multline*}
hence $\frac{1}{2N+1}\sum_{|s|\le N}|F(\mathrm{d}M_k)_s|^2\le KTN^{-1/2}$. 

Similarly, $\frac{1}{2N+1}\sum_{|s|\le N}|F(\mathrm{d}M_k)_{q-s}|^2\le KTN^{-1/2}$. Thus
\begin{equation}\label{FX2-FM2}
    \E\Big(\big|R(0)^{n,N}_{jk,q}\big|\Big) \le KTN^{-3/4}.
\end{equation}

\section{Some martingale structures}
Now let's prepare some It\^o martingales that are indispensably useful in the incoming asymptotic analysis.

For $j,k=1,\cdots,d$, define the following It\^o martingales:
\begin{eqnarray}\label{def.U}
	U^{n,N}_{jk}(t) &=& \int_0^t \dd{jk}{t,u}\, \sigma_{k\cdot}(u) \ds W(u) \nonumber\\
	\widetilde{U}^{n,N,M}_{jk}(t) &=& \int_0^t \dd{jk}{u,t}\, \widehat{\rho}_{jk}^M(\thb{j}{u})\, \sigma_{j\cdot}(u) \ds W(u) \\
	\widehat{U}^{n,N,M}_{jk,T}(t,u) &=& \int_0^u F^M\Big(\frac{t-\thb{j}{v}}{T}\Big)\,\dd{jk}{v,u}\,\sigma_{j\cdot}(v)\ds W(v), \nonumber
\end{eqnarray}
and for $j,k,l,m=1,\cdots,d$,
\begin{eqnarray}\label{def.Z}
	Z^{n,N}_{jk,lm}(t) &=& \int_0^t \dd{jk}{t,u}\, \sigma_{k\cdot}(u)\, U^{n,N}_{lm}(u) \ds W(u) \nonumber\\
	\breve{Z}^{n,N}_{jk,lm}(t) &=& \int_0^t \dd{jk}{t,u}\, \sigma_{k\cdot}(u)\, \widetilde{U}^{n,N,M}_{lm}(u) \ds W(u) \nonumber\\
	\mathring{Z}^{n,N}_{jk,lm}(t) &=& \int_0^t \dd{lm}{u,t}\, \widehat{\rho}_{lm}^M(\thb{l}{u})\, \sigma_{l\cdot}(u)\, {U}^{n,N}_{jk}(u) \ds W(u) \nonumber\\
	\widetilde{Z}^{n,N}_{jk,lm}(t) &=& \int_0^t \dd{jk}{u,t}\, \widehat{\rho}_{jk}^M(\thb{j}{u})\, \sigma_{j\cdot}(u)\, \widetilde{U}^{n,N,M}_{lm}(u) \ds W(u) .
\end{eqnarray}

Based the definition (\ref{def.U}), by It\^o's formula,
\begin{eqnarray}\label{rep.U-Z}
	U^{n,N}_{jk}(t)\,U^{n,N}_{lm}(t) &=& \int_0^t \dd{jk}{t,u}\,\dd{lm}{t,u}\, c_{km}(u) \ds u + Z^{n,N}_{jk,lm}(t) + Z^{n,N}_{lm,jk}(t) \nonumber\\
	U^{n,N}_{jk}(t)\,\widetilde{U}^{n,N,M}_{lm}(t) &=& \int_0^t \dd{jk}{t,u}\,\dd{lm}{u,t}\, \widehat{\rho}_{lm}^M(\thb{l}{u})\, c_{kl}(u)\ds u + \breve{Z}^{n,N}_{jk,lm}(t) + \mathring{Z}^{n,N}_{jk,lm}(t) \nonumber\\
	\widetilde{U}^{n,N,M}_{jk}(t)\,U^{n,N}_{lm}(t) &=& \int_0^t \dd{jk}{u,t}\,\dd{lm}{t,u}\, \widehat{\rho}_{jk}^M(\thb{j}{u})\, c_{jm}(u) \ds u + \breve{Z}^{n,N}_{lm,jk}(t) + \mathring{Z}^{n,N}_{lm,jk}(t)\nonumber\\
	\widetilde{U}^{n,N,M}_{jk}(t)\,\widetilde{U}^{n,N,M}_{lm}(t) &=& \int_0^t \dd{jk}{u,t}\,\dd{lm}{u,t}\, \widehat{\rho}_{jk}^M(\thb{j}{u})\,\widehat{\rho}_{lm}^M(\thb{l}{u})\, c_{jl}(u)\ds u \nonumber\\
	&& \hspace{6cm} + \widetilde{Z}^{n,N}_{jk,lm}(t) + \widetilde{Z}^{n,N}_{lm,jk}(t).
\end{eqnarray}

We have the following lemma about the magnitudes of quadratics.	
\begin{lem}\label{lem.U2.Z2}
Assume assumption \ref{a.SU} and (\ref{cond.g}), there exist some finite positive constant $K$ such that
\begin{eqnarray}\label{U2}
	\E\big[U^{n,N}_{jk}(t)U^{n,N}_{jk}(u)\big] &\le& K \int_0^{t\wedge u} \dd{jk}{t,v}\,\dd{jk}{u,v}\ds v \nonumber\\
	\E\big[U^{n,N}_{jk}(t)\widetilde{U}^{n,N,M}_{jk}(u)\big] &\le& K \int_0^{t\wedge u} \dd{jk}{t,v}\,\dd{jk}{v,u}\ds v \nonumber\\
	\E\big[\widetilde{U}^{n,N,M}_{jk}(t)\widetilde{U}^{n,N,M}_{jk}(u)\big] &\le& K \int_0^{t\wedge u} \dd{jk}{v,t}\,\dd{jk}{v,u}\ds v \\
	\E\big[\widehat{U}^{n,N,M}_{jk,T}(t,u)^2\big] &\le& K \int_0^u F^M\Big(\frac{t-\thb{j}{v}}{T}\Big)^2\,\dd{jk}{v,u}^2\ds v , \nonumber
\end{eqnarray}
and
\begin{eqnarray*}
	\E\big[Z^{n,N}_{jk,lm}(t)Z^{n,N}_{jk,lm}(u)\big] &\le& K \int_0^{t\wedge u} \dd{jk}{t,v}\,\dd{jk}{u,v}\ds v \int_0^v \dd{lm}{v,\vartheta}^2\ds \vartheta\\
	\E\big[\breve{Z}^{n,N}_{jk,lm}(t)\breve{Z}^{n,N}_{jk,lm}(u)\big] &\le& K \int_0^{t\wedge u} \dd{jk}{t,v}\,\dd{jk}{u,v}\ds v \int_0^v \dd{lm}{\vartheta,v}^2\ds \vartheta\\
	\E\big[\mathring{Z}^{n,N}_{jk,lm}(t)\mathring{Z}^{n,N}_{jk,lm}(u)\big] &\le& K \int_0^{t\wedge u} \dd{lm}{v,t}\,\dd{lm}{v,u}\ds v \int_0^v \dd{jk}{v,\vartheta}^2\ds \vartheta\\
	\E\big[\widetilde{Z}^{n,N}_{jk,lm}(t)\widetilde{Z}^{n,N}_{jk,lm}(u)\big] &\le& K \int_0^{t\wedge u} \dd{jk}{v,t}\,\dd{jk}{v,u}\ds v \int_0^v \dd{lm}{\vartheta,v}^2\ds \vartheta.
\end{eqnarray*}
\end{lem}
\begin{proof}
By (\ref{def.U}), It\^o's formula and Fubini's theorem,
\begin{eqnarray*}
	\E\big[U^{n,N}_{jk}(t)\,U^{n,N}_{jk}(u)\big] &=& \int_0^{t\wedge u} \dd{jk}{t,v}\,\dd{jk}{u,v}\,\E[c_{jk}(v)]\ds v \\
	\E\big[U^{n,N}_{jk}(t)\,\widetilde{U}^{n,N,M}_{jk}(u)\big] &=& \int_0^{t\wedge u} \dd{jk}{t,v}\,\dd{jk}{v,u}\,\E\big[\widehat{\rho}_{jk}^M(\thb{j}{v})\,c_{jk}(v)\big]\ds v \\
	\E\big[\widetilde{U}^{n,N,M}_{jk}(t)\,\widetilde{U}^{n,N,M}_{jk}(u)\big] &=& \int_0^{t\wedge u} \dd{jk}{v,t}\,\dd{jk}{v,u}\,\E\big[\widehat{\rho}_{jk}^M(\thb{j}{v})^2\,c_{jk}(v)\big]\ds v \\
	\E\big[\widehat{U}^{n,N,M}_{jk,T}(t,u)^2\big] &=& \int_0^u F^M\Big(\frac{t-\thb{j}{v}}{T}\Big)^2\dd{jk}{v,u}^2\,\E[c_jj(v)]\ds v,
\end{eqnarray*}
thereby the first claim follows from assumption \ref{a.SU} and (\ref{cond.g}).
	
According to It\^o's formula, (\ref{def.Z}), (\ref{rep.U-Z}),
\begin{multline*}
	\E\big[Z^{n,N}_{jk,lm}(t)Z^{n,N}_{jk,lm}(u)\big] = \E\int_0^{t\wedge u} \dd{jk}{t,v}\,\dd{jk}{u,v}\,c_{kk}(v)\,U^{n,N}_{lm}(v)^2\ds v \\
	\le K \int_0^{t\wedge u} \dd{jk}{t,v}\,\dd{jk}{u,v}\,\E\big[U^{n,N}_{lm}(v)^2\big]\ds v,
\end{multline*}
similarly,
\begin{eqnarray*}
	\E\big[\breve{Z}^{n,N}_{jk,lm}(t)\breve{Z}^{n,N}_{jk,lm}(u)\big] &\le& K \int_0^{t\wedge u} \dd{jk}{t,v}\,\dd{jk}{u,v}\,\E\big[\widetilde{U}^{n,N}_{lm}(v)^2\big]\ds v \\
	\E\big[\mathring{Z}^{n,N}_{jk,lm}(t)\mathring{Z}^{n,N}_{jk,lm}(u)\big] &\le& K \int_0^{t\wedge u} \dd{lm}{v,t}\,\dd{lm}{v,u}\,\E\big[U^{n,N}_{jk}(v)^2\big]\ds v \\
	\E\big[\widetilde{Z}^{n,N}_{jk,lm}(t)\widetilde{Z}^{n,N}_{jk,lm}(u)\big] &\le& K \int_0^{t\wedge u} \dd{jk}{v,t}\,\dd{jk}{v,u}\,\E\big[\widetilde{U}^{n,N}_{lm}(v)^2\big]\ds v,
\end{eqnarray*}
then the second claim follows from assumption \ref{a.U} and the first claim proved earlier.
\end{proof}
\section{Proof of proposition 2}\label{apdx:spot}
Recall the definitions (\ref{def.FShat}) and (\ref{def.Fhat})， based on (\ref{def.theta_n}) and (\ref{def.dirichlet.scale}), we have the expression:
\begin{equation*}
	\Fdx{j}{q-s}\times\Fdx{k}{s} = \int_0^T \ee{(q-s)}{\thb{j}{t}}\ds X_j(t) \int_0^T \ee{s}{\thb{k}{u}}\ds X_k(u),
\end{equation*}
by (\ref{def.X.mtg}) and It\^o's formula,
\begin{equation}\label{}
	\Fdx{j}{q-s}\times\Fdx{k}{s} = \phi^n_{q,s,jk} + \xi(0)^n_{q,s,jk} + \xi(1)^n_{q,s,jk},
\end{equation}
where
\begin{eqnarray*}
	\phi^n_{q,s,jk} &=& \int_0^T \ee{q}{\thb{j}{t}} \ei{s}{[\thb{j}{t}-\thb{k}{t}]} c_{jk}(t)\ds t\\
	\xi(0)^n_{q,s,jk} &=& \int_0^T \ee{q}{\thb{j}{t}}\ds X_j(t)\, \int_0^t \ei{s}{[\thb{j}{t}-\thb{k}{u}]}\ds X_k(u) \\
	\xi(1)^n_{q,s,jk} &=& \int_0^T \ds X_k(t)\, \int_0^t \ee{q}{\thb{j}{u}}\ei{s}{[\thb{j}{u}-\thb{k}{t}]}\ds X_j(u),
\end{eqnarray*}
then by (\ref{def.Fhat}), (\ref{def.Dirichlet}), (\ref{def.dirichlet.scale}),
\begin{equation}\label{rep.Fc}
	\Fc{jk}{q} = \Phi^{n,N}_{q,jk} + \Xi(0)^{n,N}_{q,jk} + \Xi(1)^{n,N}_{q,jk},
\end{equation}
where
\begin{eqnarray*}
	\Phi^{n,N}_{q,jk} &=& \int_0^T \ee{q}{\thb{j}{t}} \dd{jk}{t,t}\, c_{jk}(t)\ds t\\
	\Xi(0)^{n,N}_{q,jk} &=& \int_0^T \ee{q}{\thb{j}{t}}\ds X_j(t)\, \int_0^t \dd{jk}{t,u}\ds X_k(u) \\
	\Xi(1)^{n,N}_{q,jk} &=& \int_0^T \ds X_k(t)\, \int_0^t \ee{q}{\thb{j}{u}}\, \dd{jk}{u,t}\ds X_j(u).
\end{eqnarray*}

Thus, by (\ref{def.Fhat}) and (\ref{def.U}) we write
\begin{multline*}
	\chat{jk}{t} = \frac{1}{T}\int_0^T F^M\Big(\frac{t-\thb{j}{u}}{T}\Big)\,\dd{jk}{u,u}\,c_{jk}(u) \ds u \\
	+ \frac{1}{T}\int_0^T F^M\Big(\frac{t-\thb{j}{u}}{T}\Big)\,U^{n,N}_{jk}(u)\,\sigma_{j\cdot}(u) \ds W(u) + \frac{1}{T}\int_0^T \widehat{U}^{n,N}_{jk,T}(t,u)\,\sigma_{k\cdot}(u) \ds W(u)
\end{multline*}
therefore
\begin{equation}\label{decomp.chat.error}
	\chat{jk}{t} - c_{jk}(t) = Q(t,0)^{n,N,M}_{jk,T} + Q(t,1)^{n,N,M}_{jk,T} + \Omega(t)^{n,N,M}_{jk,T},
\end{equation}
where
\begin{eqnarray*}
	Q(t,0)^{n,N,M}_{jk,T} &=& \frac{1}{T}\int_0^T F^M\Big(\frac{t-\thb{j}{u}}{T}\Big)\,U^{n,N}_{jk}(u)\,\sigma_{j\cdot}(u) \ds W(u) \\ 
	Q(t,1)^{n,N,M}_{jk,T} &=& \frac{1}{T}\int_0^T \widehat{U}^{n,N,M}_{jk,T}(t,u)\,\sigma_{k\cdot}(u) \ds W(u) \\
	\Omega(t)^{n,N,M}_{jk,T} &=& \frac{1}{T}\int_0^T F^M\Big(\frac{t-\thb{j}{u}}{T}\Big)\,\dd{jk}{u,u}\,c_{jk}(u)\ds u - c_{jk}(t).
\end{eqnarray*}

By It\^o's formula and Fubini's theorem,
\begin{eqnarray*}
	\E\big[\big|Q(t,0)^{n,N,M}_{jk,T}\big|^2\big] &=& \frac{1}{T^2}\int_0^T F^M\Big(\frac{t-\thb{j}{u}}{T}\Big)^2\,\E\big[U^{n,N}_{jk}(u)^2\,c_{jj}(u)\big]\ds u\\
	\E\big[\big|Q(t,1)^{n,N,M}_{jk,T}\big|^2\big] &=& \frac{1}{T^2}\int_0^T \E\big[\widehat{U}^{n,N,M}_{jk,T}(t,u)^2\,c_{kk}(u)\big]\ds u,
\end{eqnarray*}
because of assumption \ref{a.SU} and lemma \ref{lem.U2.Z2}, 
\begin{multline*}
	\E\big[\big|Q(t,0)^{n,N,M}_{jk,T}\big|^2\big] \le K\int_0^T F^M\Big(\frac{t-\thb{j}{u}}{T}\Big)^2\ds u\, \int_0^u\dd{jk}{u,v}^2\ds v \\
	\le K \Big[\sup_{u\in[0,T]}\int_0^T\dd{jk}{u,v}^2\ds v\Big]\cdot\Big[\int_0^TF^M\Big(\frac{t-\thb{j}{u}}{T}\Big)^2\ds u\Big],
\end{multline*}
and by Fubini's theorem,
\begin{multline*}
	\E\big[\big|Q(t,1)^{n,N,M}_{jk,T}\big|^2\big]\\ \le K\int_0^T\ds u\, \int_0^uF^M\Big(\frac{t-\thb{j}{v}}{T}\Big)^2\,\dd{jk}{v,u}^2\ds v
	= K\int_0^T F^M\Big(\frac{t-\thb{j}{v}}{T}\Big)^2\ds v\, \int_v^T\,\dd{jk}{v,u}^2\ds u \\
	\le K \Big[\sup_{v\in[0,T]}\int_0^T\dd{jk}{v,u}^2\ds u\Big]\cdot\Big[\int_0^TF^M\Big(\frac{t-\thb{j}{u}}{T}\Big)^2\ds u\Big],
\end{multline*}
according to (\ref{Jackson}) and lemma \ref{lem.dp},
\begin{equation}\label{null.Qs}
	\big|Q(t,0)^{n,N,M}_{jk,T}\big|^2 + \big|Q(t,1)^{n,N,M}_{jk,T}\big|^2 \le K\frac{M}{N}.
\end{equation}

Notice
\begin{equation}\label{decomp.Omega}
	\Omega(t)^{n,N,M}_{jk,T} = \Omega(t,0)^{n,N,M}_{jk,T} + \Omega(t,1)^{n,M}_{jk,T} + \Omega(t,2)^M_{jk,T},
\end{equation}
where
\begin{eqnarray*}
	\Omega(t,0)^{n,N,M}_{jk,T} &=& \frac{1}{T}\int_0^T F^M\Big(\frac{t-\thb{j}{u}}{T}\Big)\, c_{jk}(u)\, \big[\dd{jk}{u,u}-1\big]\ds u \\
	\Omega(t,1)^{n,M}_{jk,T} &=& \frac{1}{T}\int_0^T F^M\Big(\frac{t-\thb{j}{u}}{T}\Big)\, c_{jk}(u)\ds u - \frac{1}{T}\int_0^T F^M\Big(\frac{t-u}{T}\Big)\, c_{jk}(u)\ds u \\
	\Omega(t,2)^M_{jk,T} &=& \frac{1}{T}\int_0^T F^M\Big(\frac{t-u}{T}\Big)\, c_{jk}(u)\ds u - c_{jk}(t).
\end{eqnarray*}
Based on the Taylor series of sine function, assumption \ref{a.SU}, (\ref{Fejer.integral}),
\begin{equation*}
	\E\Big(\sup_{t\in[0,T]}\big|\Omega(t,0)^{n,N,M}_{jk,T}\big|^2\Big) \le K \frac{N^4}{\underline{n}^4} \mathds{1}_{\{j\ne k\}};
\end{equation*}
according to assumption \ref{a.SU} and lemma \ref{lem.Fejer.Riemann}, we know
\begin{equation*}
	\E\Big(\sup_{t\in[0,T]}\big|\Omega(t,1)^{n,M}_{jk,T}\big|^2\Big) \le K \frac{M^2}{\underline{n}^2};
\end{equation*}
by lemma \ref{lem.Fejer.uc},
\begin{equation*}
\begin{array}{ll}
	\E\Big(\sup_{t\in[1/M,T-1/M]}\big|\Omega(t,2)^M_{jk,T}\big|^2\Big) \le K M^{-2\alpha}, &\text{if } c(0)\ne c(T)\\
	\E\Big(\sup_{t\in[0,T]}\big|\Omega(t,2)^M_{jk,T}\big|^2\Big) \le K M^{-2\alpha}, &\text{if } c(0)=c(T)
\end{array}
\end{equation*}
then proposition \ref{prop.msr} follows from (\ref{decomp.chat.error}), (\ref{null.Qs}), (\ref{decomp.Omega}), (\ref{cond.NM}).

\section{Proof of theorem 1, theorem 2}\label{apdx:thm2}
We can write, for $j,k=1,\cdots,d$
\begin{eqnarray}\label{decomp.uni}
	N^{1/2}\Big[\widehat{S}(g)^n_{jk,T} - S(g)_{jk,T}\Big] &=& \overline{S}(0)^{n,N,M,B}_{jk,T} + \overline{S}(1)^{n,N,M}_{jk,T} + \overline{S}(2)^{n,N,M}_{jk,T} \\
	N^{1/2}\Big[\widetilde{S}(g)^n_{jj,T} - S(g)_{jj,T}\Big] &=& \widetilde{S}(0)^{n,N,M}_{jj,T} + \overline{S}(1)^{n,N,M}_{jj,T} + \overline{S}(2)^{n,N,M}_{jj,T} \nonumber
\end{eqnarray}
where
\begin{eqnarray*}
	\overline{S}(0)^{n,N,M,B}_{jk,T} &\coloneqq& N^{1/2} \Bigg[\sum_{h=1}^{B}g\big(\chat{jk}{t_h}\big)\,T/B - \int_{t_0}^{t_{B}} g\big(\chat{jk}{t}\big) \ds t \Bigg] \\
	\widetilde{S}(0)^{n,N,M}_{jj,T} &\coloneqq& N^{1/2} \Bigg[\sum_{h=1}^{n^j}g\big(\chat{jj}{\tau_h}\big)\Delta^j_h - \int_{\ta{j}{0}}^{\ta{j}{n_j}} g\big(\chat{jj}{t}\big) \ds t \Bigg] \\
	&& \hspace{32mm} N^{1/2}\int_0^{\ta{0}{n_j}} g\big(\chat{jj}{t}\big) \ds t + N^{1/2}\int_{\ta{j}{n_j}}^T g\big(\chat{jj}{t}\big) \ds t \\
	\overline{S}(1)^{n,N,M}_{jk,T} &\coloneqq& N^{1/2} \int_0^T \Big\{g\big(\chat{jk}{t}\big) - g(c_{jk}(t)) - \partial g(c_{jk}(t))\Big[\chat{jk}{t} - c_{jk}(t)\Big] \Big\}\ds t \\
	\overline{S}(2)^{n,N,M}_{jk,T} &\coloneqq& N^{1/2} \int_0^T \partial g(c_{jk}(t)) \Big[\chat{jk}{t} - c_{jk}(t)\Big]\ds t,
\end{eqnarray*}
and $t_h=hT/B$.

By assumption \ref{a.U} and (\ref{cond.g}), we know $g\big(\chat{jk}{t}\big)\le K$, then
$N^{1/2} \big[\int_0^{\ta{0}{n_j}} g\big(\chat{jj}{t}\big) \ds t + \int_{\ta{j}{n_j}}^T g\big(\chat{jj}{t}\big) \ds t\big] = O_p(N^{1/2}/n)=o_p(1)$ by assumption \ref{a.T} and (\ref{cond.NM}). By Riemann summation and assumption \ref{a.T},
\begin{eqnarray*}
	\big\|\overline{S}(0)^{n,N,M,B}_{jk,T}\big\| &=& O_p(N^{1/2}/B) \\
	N^{1/2} \Bigg\|\Bigg[ \sum_{h=0}^{n_j}g\big(\chat{jj}{\tau_h}\big)\Delta^j_h - \int_{\ta{j}{0}}^{\ta{j}{n_j}} g\big(\chat{jj}{t}\big) \ds t\Bigg]\Bigg\| &=& O_p(N^{1/2}/n),
\end{eqnarray*}
in view of (\ref{cond.BL}),
\begin{equation*}
	\Big\|\overline{S}(0)^{n,N,M,B}_{jk,T}\Big\| + \left\|\widetilde{S}(0)^{n,N,M}_{jj,T}\right\| \overset{\mathbb{P}}{\longrightarrow}0.
\end{equation*}

By (\ref{cond.g}), $\big\|g\big(\chat{jk}{t}\big) - g(c_{jk}(t)) - \partial g(c_{jk}(t))\big[\chat{jk}{t} - c_{jk}(t)\big]\big\| = O_p\big(|\chat{jk}{t} - c_{jk}(t)|^2\big)$, hence we have
\begin{equation*}
	\Big\|\overline{S}(1)^{n,N,M}_{jk,T}\Big\| \le KN^{1/2}\int_0^T\big|\chat{jk}{t} - c_{jk}(t)\big|^2\ds t \le KTN^{1/2}\sup_{t\in[0,T]}\big|\chat{jk}{t} - c_{jk}(t)\big|^2,
\end{equation*}
following from proposition \ref{prop.msr}, we have $\E\big\|\overline{S}(1)^{n,N,M}_{jk,T}\big\|\le KTM/N^{1/2}$ under conditions of theorem \ref{thm.bi}. Then by (\ref{cond.g}) and Markov's inequality, 
\begin{equation*}
	\Big\|\overline{S}(1)^{n,N,M}_{jk,T}\Big\| \overset{\mathbb{P}}{\longrightarrow} 0.
\end{equation*}

It remains to show the stable convergence of $\overline{S}(2)^{n,N,M}_{jk,T} = N^{1/2} \int_0^T\partial g(c_{jk}(t))\Big[\overline{\chat{jk}{t}}-c_{jk}(t)\Big]\ds t$.

Let $\underline{T} = \min_{j}\ta{j}{n_j}$, note $T-\underline{T}\le \Delta(n)$, without loss of generality, we can assume $\underline{T}=T$, i.e., $\ta{j}{n^j}=T$, for $j=1,\cdots,d$.

\subsection{decomposition}
Let $\rho_{jk}(t) = \partial g(c_{jk}(t))$, by the definition (\ref{def.chat}), we have
\begin{multline*}
	\int_0^T\rho_{jk}(t)\overline{\chat{jk}{t}}\ds t 
	= \frac{1}{T}\sum_{|q|<M}\Big(1 - \frac{|q|}{M}\Big)\overline{\Fc{jk}{q}}\int_0^T \rho_{jk}(t)\ee{q}{t}\ds t \\
	= \frac{1}{T}\sum_{|q|<M}\Big(1 - \frac{|q|}{M}\Big)F(\rho_{jk})_q\,\overline{\Fc{jk}{q}},
\end{multline*}
by (\ref{def.fhat}) and (\ref{rep.Fc}),
\begin{multline*}
	\frac{N^{1/2}}{T}\sum_{|q|<M}\Big(1 - \frac{|q|}{M}\Big)F(\rho_{jk})_q\,\overline{\Fc{jk}{q}} = \frac{N^{1/2}}{T}\sum_{|q|<M}\Big(1 - \frac{|q|}{M}\Big)F(\rho_{jk})_q \Big[\overline{\Phi^{n,N}_{q,jk}} + \overline{\Xi(0)^{n,N}_{q,jk}} + \overline{\Xi(1)^{n,N}_{q,jk}}\Big] \\
	= N^{1/2}\int_0^T \widehat{\rho}_{jk}^M(\thb{j}{t})\, \dd{jk}{t,t}\, c_{jk}(t) \ds t + e(0)^{n,N,M}_{jk,T} + e(1)^{n,N,M}_{jk,T},
\end{multline*}
where
\begin{eqnarray}\label{def.es}
	e(0)^{n,N,M}_{jk,T} &=& N^{1/2} \int_0^T \widehat{\rho}_{jk}^M(\thb{j}{t})\, \sigma_{j\cdot}(t) \ds W(t) \int_0^t \dd{jk}{t,u}\,\sigma_{k\cdot}(u) \ds W(u) \nonumber\\
	e(1)^{n,N,M}_{jk,T} &=& N^{1/2} \int_0^T \sigma_{k\cdot}(t) \ds W(t) \int_0^t \widehat{\rho}_{jk}^M(\thb{j}{u})\, \dd{jk}{u,t}\,\sigma_{j\cdot}(u) \ds W(u).
\end{eqnarray}

Therefore, we have the following decomposition:
\begin{equation}\label{decomp.error_jk}
	N^{1/2}\int_0^T \rho_{jk}(t)\Big[\overline{\chat{jk}{t}}-c_{jk}(t)\Big]\ds t = o(0)^{n,M}_{jk,T} + o(1)^{n,N,M}_{jk,T} + e(0)^{n,N,M}_{jk,T} + e(1)^{n,N,M}_{jk,T},
\end{equation}
where
\begin{eqnarray}\label{def.os}
	o(0)^{n,M}_{jk,T} &=& N^{1/2} \int_0^T \big[\widehat{\rho}_{jk}^M(\thb{j}{t}) - \rho_{jk}(t)\big]c_{jk}(t)\ds t \nonumber\\
	o(1)^{n,N,M}_{jk,T} &=& N^{1/2}\int_0^T \widehat{\rho}_{jk}^M(\thb{j}{t})\, c_{jk}(t)\, \big[\dd{jk}{t,t}-1\big] \ds t.
\end{eqnarray}

\noindent\textbf{(1)} On one hand, by (\ref{Fejer.integral}), (\ref{rep.fhat}), Fubini's theorem, and lemma \ref{lem.Fejer.Riemann},
\begin{equation*}
	o(0)^{n,M}_{jk,T} = N^{1/2} \int_0^T c_{jk}(t)\ds t\, \frac{1}{T}\int_0^T F^M\Big(\frac{u-\thb{j}{t}}{T}\Big) \big[\rho_{jk}(u) - \rho_{jk}(t)\big] \ds u = \frac{N^{1/2}}{T} J^M_{jk,T} + O_p\big(N^{1/2}M/n\big),
\end{equation*}
where
	\[ J^M_{jk,T} = \int_0^T \int_0^T F^M\Big(\frac{u-t}{T}\Big) \big[\rho_{jk}(u) - \rho_{jk}(t)\big] c_{jk}(t) \ds t \ds u. \]
By symmetry of variables, $J^M_{jk,T} = \int_0^T \int_0^T F^M[(u-t)/T] \big[\rho_{jk}(t) - \rho_{jk}(u)\big] c_{jk}(u) \ds u \ds t$, hence
	\[ J^M_{jk,T} = -\frac{1}{2}\int_0^T\ds u \int_0^T F^M\Big(\frac{u-t}{T}\Big) \big[\rho_{jk}(u) - \rho_{jk}(t)\big] \big[c_{jk}(u) - c_{jk}(t)\big] \ds t. \]
By (\ref{cond.g}) the modulus of continuity of $\rho$ is determined by that of $c$, let $L^M_{jk,T}(u) \coloneqq \int_0^T F^M[(u-t)/T] \big[\rho_{jk}(u) - \rho_{jk}(t)\big] \big[c_{jk}(u) - c_{jk}(t)\big] \ds t$, by periodicity of $c$ and $\rho$, $L^M_{jk,T}(u) = \int_{u-T/2}^{u+T/2} F^M[(u-t)/T] \big[\rho_{jk}(u) - \rho_{jk}(t)\big] \big[c_{jk}(u) - c_{jk}(t)\big] \ds t$. Note
\begin{equation*}
	\big|L^M_{jk,T}(u)\big| \le \Big(\int_{|u-t|\le 1/M} + \int_{|u-t|>1/M}\Big) F^M\Big(\frac{u-t}{T}\Big) \big|\rho_{jk}(u) - \rho_{jk}(t)\big| \big|c_{jk}(u) - c_{jk}(t)\big| \ds t,
\end{equation*}
through an argument similar to the proof of lemma \ref{lem.Fejer.uc}, we have $\E\big|L^M_{jk,T}(u)\big| \le K\big[M^{-2\alpha} + M^{-(1+\alpha)}\big]$,
thus
\begin{equation*}
	\E\big|o(0)^{n,M}_{jk,T}\big| \le K\Big[\Big(\frac{N}{M^{4\alpha}}\Big)^{1/2} + \Big(\frac{N}{\underline{n}}\Big)^{1/2}\frac{M}{\underline{n}^{1/2}}\Big],
\end{equation*}
by (\ref{cond.NM}) and Markov's inequality, we have shown the asymptotic negligibility in probability of $o(0)^{n,M}_{jk,T}$, i.e.,
\begin{equation}\label{null.o(0)}
	o(0)^{n,M}_{jk,T} \overset{\mathbb{P}}{\longrightarrow} 0.
\end{equation}

\noindent\textbf{(2)} On the other hand, according to the definition (\ref{def.dirichlet.scale}) and the Taylor series of the sine function,
\begin{equation}\label{d-1}
	\dd{jk}{t,t} - 1 = -\frac{\pi^2}{6}(2N+1)^2\big[\thb{j}{t}-\thb{k}{t}\big]^2 + O_p(N^4\Delta(n)^4) 
\end{equation}
so
\begin{equation*}
	\E\big|o(1)^{n,N,M}_{jk,T}\big| \le KN^{5/2}\Delta(n)^2\mathds{1}_{\{j\ne k\}},
\end{equation*}
hence if we let $N\le\lfloor\underline{n}/2\rfloor-M+1$ in case $j=k$, and let $N=o(\underline{n}^{4/5})$ in case $j\ne k$, it follows
\begin{equation}\label{null.o(1)}
	o(1)^{n,N,M}_{jk,T} \overset{\mathbb{P}}{\longrightarrow} 0.
\end{equation}

Thus the asymptotics solely relies on $e(0)^{n,N,M}_{jk,T} + e(1)^{n,N,M}_{jk,T}$.

\subsection{stable convergence}
By (\ref{def.es}) and (\ref{def.U}), we can write
\begin{eqnarray}\label{rep.es}
e(0)^{n,N,M}_{jk,T} &=& N^{1/2}\int_0^T\widehat{\rho}_{jk}^M(\thb{j}{t})\,\sigma_{j\cdot}(t)\,U^{n,N}_{jk}(t)\ds W(t) \nonumber\\
e(1)^{n,N,M}_{jk,T} &=& N^{1/2}\int_0^T\sigma_{k\cdot}(t)\,\widetilde{U}^{n,N,M}_{jk}(t)\ds W(t).
\end{eqnarray}

To establish the stable convergence, according to \cite{j97}, \cite{jp98}, we need to consider the limits in probability of the brackets $\big\langle e(0)^{n,N,M}_{jk} + e(1)^{n,N,M}_{jk}, W_r\big\rangle_T$ and $\big\langle e(0)^{n,N,M}_{jk} + e(1)^{n,N,M}_{jk}, e(0)^{n,N,M}_{jk} + e(1)^{n,N,M}_{jk}\big\rangle_T$.

\noindent\textbf{(1)} First, let's consider, for $r=1\cdots,d'$,
\begin{eqnarray*}
	\big\langle e(0)^{n,N,M}_{jk},W_r\big\rangle_T &=& N^{1/2}\int_0^T \widehat{\rho}_{jk}^M(\thb{j}{t})\, \sigma_{jr}(t)\, U^{n,N}_{jk}(t)\ds t \\
	\big\langle e(1)^{n,N,M}_{jk},W_r\big\rangle_T &=& N^{1/2}\int_0^T \sigma_{kr}(t)\, \widetilde{U}^{n,N,M}_{jk}(t)\ds t.
\end{eqnarray*}
notice that
\begin{equation*}
	\big\langle e(0)^{n,N,M}_{jk},W_r\big\rangle_T^2 = N\int_0^T\int_0^T \widehat{\rho}_{jk}^M(\thb{j}{t})\,\widehat{\rho}_{jk}^M(\thb{j}{u})\, \sigma_{jr}(t)\,\sigma_{jr}(u) 
	\times U^{n,N}_{jk}(t)U^{n,N}_{jk}(u)\ds t\ds u,
\end{equation*}
according to lemma \ref{lem.U2.Z2}, Fubini's theorem, H\"older's inequality, 
\begin{multline*}
	\E\big[\big\langle e(0)^{n,N,M}_{jk},W_r\big\rangle_T^2\big] \le KN\int_0^T\int_0^T\ds t\ds u\, \Big(\int_0^{t\wedge u} \dd{jk}{t,v}\,\dd{jk}{u,v}\ds v\Big) \\
	\le KN \int_0^T\ds v\, \Big(\int_v^T \big|\dd{jk}{t,v}\big|\ds t\Big) \Big(\int_v^T \big|\dd{jk}{u,v}\big|\ds u\Big) \\
	\le KN \int_0^T\ds v\, \Big(\int_v^T \big|\dd{jk}{t,v}\big|\ds t\Big)^2 \le KT^{\frac{3p-2}{p}} N \Big(\int_0^T\big|\dd{jk}{t,v}\big|^p\ds t\Big)^{2/p},
\end{multline*}
similarly,
\begin{equation*}
	\big\langle e(1)^{n,N,M}_{jk},W_r\big\rangle_T^2 = N\int_0^{\tau_n}\int_0^{\tau_n}\sigma_{kr}(t)\sigma_{kr}(u) \times \widetilde{U}^{n,N,M}_{jk}(t)\widetilde{U}^{n,N,M}_{jk}(u)\ds t\ds u,
\end{equation*}
by a similar argument applied to $\E[\langle e(0)^{n,N,M}_{jk},W_r\rangle_T^2]$,
\begin{equation*}
	\E\big[\big\langle e(1)^{n,N,M}_{jk},W_r\big\rangle_T^2\big] \le KT^{\frac{3p-2}{p}} N \Big(\int_0^T\big|\dd{jk}{v,t}\big|^p\ds t\Big)^{2/p}.
\end{equation*}

By Jensen's inequality and Markov's inequality, we have the following lemma.
\begin{lem}\label{lem.stable}
Under the assumptions of theorem \ref{thm.bi}, $\forall j,k=1,\cdots,d$ and $\forall r=1,\cdots,d'$,
\begin{equation*}
	\big\langle e(0)^{n,N,M}_{jk}+e(1)^{n,N,M}_{jk},W_r\big\rangle_T \overset{\mathbb{P}}{\longrightarrow} 0.
\end{equation*}
\end{lem}

\noindent\textbf{(2)} Second, let's consider
\begin{eqnarray*}
	\big\langle e(0)^{n,N,M}_{jk},e(0)^{n,N,M}_{jk}\big\rangle_T &=& N\int_0^T \widehat{\rho}_{jk}^M(\thb{j}{t})^2\,c_{jj}(t)\, U^{n,N}_{jk}(t)^2\ds t \\
	\big\langle e(0)^{n,N,M}_{jk},e(1)^{n,N,M}_{jk}\big\rangle_T &=& N\int_0^T \widehat{\rho}_{jk}^M(\thb{j}{t})\,c_{jk}(t)\, U^{n,N}_{jk}(t)\,\widetilde{U}^{n,N,M}_{jk}(t)\ds t \\
	\big\langle e(1)^{n,N,M}_{jk},e(1)^{n,N,M}_{jk}\big\rangle_T &=& N\int_0^T c_{kk}(t)\, \widetilde{U}^{n,N,M}_{jk}(t)^2\ds t,
\end{eqnarray*}
in view of (\ref{rep.U-Z}),
\begin{eqnarray*}
	\big\langle e(0)^{n,N,M}_{jk},e(0)^{n,N,M}_{jk}\big\rangle_T &=& 2O(0)^{n,N,M}_{jk,T} + V(0)^{n,N,M}_{jk,T} \\
	\big\langle e(0)^{n,N,M}_{jk},e(1)^{n,N,M}_{jk}\big\rangle_T &=& O(1)^{n,N,M}_{jk,T} + O(2)^{n,N,M}_{jk,T} + V(1)^{n,N,M}_{jk,T} \\
	\big\langle e(1)^{n,N,M}_{jk},e(1)^{n,N,M}_{jk}\big\rangle_T &=& 2O(3)^{n,N,M}_{jk,T} + V(2)^{n,N,M}_{jk,T},
\end{eqnarray*}
where
\begin{eqnarray*}
	O(0)^{n,N,M}_{jk,T} &=& N\int_0^T \widehat{\rho}_{jk}^M(\thb{j}{t})^2\, c_{jj}(t)\, Z^{n,N}_{jk}(t)\ds t \\
	O(1)^{n,N,M}_{jk,T} &=& N\int_0^T \widehat{\rho}_{jk}^M(\thb{j}{t})\, c_{jk}(t)\, \breve{Z}^{n,N}_{jk}(t)\ds t \\
	O(2)^{n,N,M}_{jk,T} &=& N\int_0^T \widehat{\rho}_{jk}^M(\thb{j}{t})\, c_{jk}(t)\, \mathring{Z}^{n,N}_{jk}(t)\ds t \\
	O(3)^{n,N,M}_{jk,T} &=& N\int_0^T c_{kk}(t)\, \widetilde{Z}^{n,N}_{jk}(t)\ds t
\end{eqnarray*}
and
\begin{eqnarray*}
	V(0)^{n,N,M}_{jk,T} &=& \int_0^T \widehat{\rho}_{jk}^M(\thb{j}{t})^2\,c_{jj}(t)\ds t\, \Big[N\int_0^t \dd{jk}{t,u}^2 \,c_{kk}(u)\ds u\Big] \\
	V(1)^{n,N,M}_{jk,T} &=& \int_0^T \widehat{\rho}_{jk}^M(\thb{j}{t})\,c_{jk}(t)\ds t\, \Big[N\int_0^t \dd{jk}{t,u}\,\dd{jk}{u,t}\, \widehat{\rho}_{jk}^M(\thb{j}{u})\, c_{jk}(u)\ds u\Big] \\
	V(2)^{n,N,M}_{jk,T} &=& \int_0^T c_{kk}(t)\ds t\, \Big[N\int_0^t \dd{jk}{u,t}^2\, \widehat{\rho}_{jk}^M(\thb{j}{u})^2\, c_{jj}(u)\ds u\Big].
\end{eqnarray*}

Let consider the asymptotically negligible terms,
\begin{eqnarray*}
	\big| O(0)^{n,N,M}_{jk,T}\big|^2 &=& N^2\int_0^T\int_0^T\widehat{\rho}_{jk}^M(\thb{j}{t})^2\,\widehat{\rho}_{jk}^M(\thb{j}{u})^2\, c_{jj}(t)\,c_{jj}(u)\, Z^{n,N}_{jk}(t)\,Z^{n,N}_{jk}(u)\ds t\ds u \\
	\big|O(1)^{n,N,M}_{jk,T}\big|^2 &=& N^2\int_0^T\int_0^T \widehat{\rho}_{jk}^M(\thb{j}{t})\,\widehat{\rho}_{jk}^M(\thb{j}{u})\, c_{jk}(t)\,c_{jk}(u)\,  \breve{Z}^{n,N}_{jk}(t)\,\breve{Z}^{n,N}_{jk}(u)\ds t\ds u \\
	\big|O(2)^{n,N,M}_{jk,T}\big|^2 &=& N^2\int_0^T\int_0^T \widehat{\rho}_{jk}^M(\thb{j}{t})\,\widehat{\rho}_{jk}^M(\thb{j}{u})\, c_{jk}(t)\,c_{jk}(u)\,  \mathring{Z}^{n,N}_{jk}(t)\,\mathring{Z}^{n,N}_{jk}(u)\ds t\ds u \\
	\big|O(3)^{n,N,M}_{jk,T}\big|^2 &=& N^2\int_0^T\int_0^T c_{kk}(t)\,c_{kk}(u)\,  \widetilde{Z}^{n,N}_{jk}(t)\,\widetilde{Z}^{n,N}_{jk}(u)\ds t\ds u,
\end{eqnarray*}
by (\ref{cond.g}), assumption \ref{a.SU}, lemma \ref{lem.U2.Z2},
\begin{equation*}
	\E\big(\big|O(0)^{n,N,M}_{jk,T}\big|^2\big) 
	\le KN^2\int_0^T\int_0^T\ds t\ds u\, \Big[\int_0^{t\wedge u} \dd{jk}{t,v}\,\dd{jk}{u,v}\ds v\, \int_0^v \dd{jk}{v,\vartheta}^2\ds \vartheta\Big]
\end{equation*}
then by (\ref{bdd.dd}), 
\begin{equation}
	\E\big(\big|O(0)^{n,N,M}_{jk,T}\big|^2\big) \le KT^{\frac{3p-2}{p}}N^{1-2/p}.
\end{equation}

By similar arguments, we can show the same upper bound applies to $\E\big(\big|O(1)^{n,N,M}_{jk,T}\big|^2\big)$, $\E\big(\big|O(2)^{n,N,M}_{jk,T}\big|^2\big)$ and $\E\big(\big|O(3)^{n,N,M}_{jk,T}\big|^2\big)$ as well. Let $p\in(1,2)$ and use Jensen's inequality and Markov's inequality, we have
\begin{lem}\label{lem.null.bi}
Under the assumptions of theorem \ref{thm.bi}, $\forall j,k=1,\cdots,d$,
\begin{equation*}
	\max_{r=0,1,2,3}O(r)^{n,N,M}_{jk,T} \overset{\mathbb{P}}{\longrightarrow} 0
\end{equation*}
\end{lem}

Now, let's consider the terms which contribute to the asymptotic variance. By (\ref{approximation}) and lemma \ref{lem.dd}, we have the following lemma:
\begin{lem}\label{lem.AVAR.bi}
Under the assumptions of theorem \ref{thm.bi}, $\forall j,k=1,\cdots,d$,
\begin{eqnarray*}
	V(0)^{n,N,M}_{jk,T} &\overset{\mathbb{P}}{\longrightarrow}& \int_0^T \rho_{jk}(t)^2\, \tilde{\theta}_{jk,jk}(t)\, c_{jj}(t)\,c_{kk}(t) \ds t \\
	V(1)^{n,N,M}_{jk,T} &\overset{\mathbb{P}}{\longrightarrow}& \int_0^T \rho_{jk}(t)^2\, \check{\theta}_{jk,jk}(t)\, c_{jk}(t)^2\ds t \\
	V(2)^{n,N,M}_{jk,T} &\overset{\mathbb{P}}{\longrightarrow}& \int_0^T \rho_{jk}(t)^2\, \grave{\theta}_{jk,jk}(t)\, c_{jj}(t)\,c_{kk}(t) \ds t.
\end{eqnarray*}
\end{lem}
Theorem \ref{thm.bi} then follows from (\ref{decomp.error_jk}), (\ref{null.o(0)}), (\ref{null.o(1)}), and lemma \ref{lem.stable}, \ref{lem.null.bi}, \ref{lem.AVAR.bi}.

\section{Proof of theorem 3, proposition 3, proposition 4}\label{apdx:thm3}
We can write
\begin{eqnarray}\label{decomp.multi}
	N^{1/2}\big[\widehat{S}(g)^{n}_T - S(g)_T\big] &=& \overline{S}(0)^{n,N,M,B}_T + \overline{S}(1)^{n,N,M}_T + \overline{S}(2)^{n,N,M}_T \\
	N^{1/2}\big[\widehat{S}(g)^{n}_T - \underline{S}(g)^{n,N}_T\big] &=& \overline{S}(0)^{n,N,M,B}_T + \overline{\underline{S}}(1)^{n,N,M}_T + \overline{\underline{S}}(2)^{n,N,M}_T, \nonumber
\end{eqnarray}
where
\begin{eqnarray*}
	\overline{S}(0)^{n,N,M,B}_T &\coloneqq& N^{1/2} \left[\sum_{h=1}^{B}g\big(\chat{}{hT/B}\big)\,T/B - \int_{0}^{T} g\big(\chat{}{t}\big) \ds t \right] \\
	\overline{S}(1)^{n,N,M}_T &\coloneqq& N^{1/2} \int_0^T \Big\{g\big(\chat{}{t}\big) - g(c(t)) - \sum_{j,k=1}^d\partial_{jk} g(c(t))\Big[\chat{jk}{t} - c_{jk}(t)\Big] \Big\}\ds t \\
	\overline{S}(2)^{n,N,M}_T &\coloneqq& \sum_{j,k=1}^d N^{1/2} \int_0^T \partial_{jk} g(c(t)) \Big[\chat{jk}{t} - c_{jk}(t)\Big]\ds t,
\end{eqnarray*}
and
\begin{eqnarray*}
	\overline{\underline{S}}(1)^{n,N,M}_T &\coloneqq& N^{1/2} \int_0^T \Big\{g\big(\chat{}{t}\big) - g\big(\underline{c}^{n,N}(t)\big) - \sum_{j,k=1}^d\partial_{jk} g\big(\underline{c}^{n,N}(t)\big)\Big[\chat{jk}{t} - \underline{c}^{n,N}_{jk}(t)\Big] \Big\}\ds t \\
	\overline{\underline{S}}(2)^{n,N,M}_T &\coloneqq& \sum_{j,k=1}^d N^{1/2} \int_0^T \partial_{jk}g\big(\underline{c}^{n,N}(t)\big) \Big[\chat{jk}{t} - \dd{jk}{t,t}\,c_{jk}(t)\Big]\ds t,
\end{eqnarray*}

Use an argument similar to that on $\overline{S}(0)^{n,N,M,B}_{jk,T}$ in appendix \ref{apdx:thm2}, it follows
\begin{equation}\label{null.S(0)}
	\big\|\overline{S}(0)^{n,N,M,B}_T\big\| \overset{\mathbb{P}}{\longrightarrow} 0.
\end{equation}

By (\ref{cond.g}), 
\begin{eqnarray*}
	\Big\|g\big(\chat{}{t}\big) - g(c(t)) - \sum_{j,k}\partial_{jk} g(c(t))\big[\chat{jk}{t} - c_{jk}(t)\big]\Big\| &=& O_p\big(\|\chat{}{t} - c(t)\|^2\big) \\
	\Big\|g\big(\chat{}{t}\big) - g\big(\underline{c}(t)\big) - \sum_{j,k}\partial_{jk} g\big(\underline{c}^{n,N}(t)\big)\big[\chat{jk}{t} - \underline{c}^{n,N}_{jk}(t)\big]\Big\| &=& O_p\big(\|\chat{}{t} - \underline{c}^{n,N}(t)\|^2\big),
\end{eqnarray*}
therefore
\begin{eqnarray*}
	\big\|\overline{S}(1)^{n,N,M}_T\big\| &\le& KTN^{1/2}\sup_{t\in[0,T]}\big\|\chat{}{t} - c(t)\big\|^2 \\
	\big\|\overline{\underline{S}}(1)^{n,N,M}_T\big\| &\le& KTN^{1/2}\sup_{t\in[0,T]}\big\|\chat{}{t} - \underline{c}^{n,N}(t)\big\|^2.
\end{eqnarray*}
According to proposition \ref{prop.msr}, under conditions of theorem \ref{thm.multi},
	\[ \E\big\|\overline{S}(1)^{n,N,M}_T\big\| \le KT \frac{M}{N^{1/2}}. \]
Notice that
\begin{equation*}
	\chat{jk}{t} - \underline{c}_{jk}(t) = Q(t,0)^{n,N,M}_{jk,T} + Q(t,1)^{n,N,M}_{jk,T} + \underline{\Omega}(t,1)^{n,N,M}_{jk,T} + \underline{\Omega}(t,2)^{N,M}_{jk,T},
\end{equation*}
where $Q(t,0)^{n,N,M}_{jk,T}$ and $Q(t,1)^{n,N,M}_{jk,T}$ are defined by (\ref{decomp.chat.error}) and
\begin{eqnarray*}
	\underline{\Omega}(t,1)^{n,N,M}_{jk,T} &=& \frac{1}{T}\int_0^T \Big[F^M\Big(\frac{t-\thb{j}{u}}{T}\Big) - F^M\Big(\frac{t-u}{T}\Big)\Big]\,\dd{jk}{u,u}\,c_{jk}(u)\ds u \\
	\underline{\Omega}(t,2)^{N,M}_{jk,T} &=& \frac{1}{T}\int_0^T F^M\Big(\frac{t-u}{T}\Big)\,\dd{jk}{u,u}\,c_{jk}(u)\ds u - \dd{jk}{t,t}\,c_{jk}(t),
\end{eqnarray*}
by a similar proof to that of proposition \ref{prop.msr},
	\[ \E\big\|\overline{\underline{S}}(1)^{n,N,M}_T\big\| \le KT \frac{M}{N^{1/2}}. \]

Thus by (\ref{cond.NM}) and Markov's inequality,
\begin{equation}\label{null.S(1)}
	\big\|\overline{S}(1)^{n,N,M}_T\big\| + \big\|\overline{\underline{S}}(1)^{n,N,M}_T\big\| \overset{\mathbb{P}}{\longrightarrow} 0.
\end{equation}

\subsection{stable convergence}
Let $\rho_{jk}(t) = \partial_{jk}g(c(t))$ and $\underline{\rho}_{jk}(t) = \partial_{jk}g\big(\underline{c}^{n,N}(t)\big)$, we need to study
\begin{eqnarray*}
	\overline{S}(2)^{n,N,M}_T &=& \sum_{j,k=1}^d N^{1/2} \int_0^T\rho_{jk}(t)\Big[\overline{\chat{jk}{t}}-c_{jk}(t)\Big]\ds t \\
	\overline{\underline{S}}(2)^{n,N,M}_T &=& \sum_{j,k=1}^d N^{1/2} \int_0^T\underline{\rho}_{jk}(t)\Big[\overline{\chat{jk}{t}} - \dd{jk}{t,t}\,c_{jk}(t)\Big]\ds t.
\end{eqnarray*}

Based on (\ref{decomp.error_jk}),
\begin{equation*}
	\overline{S}(2)^{n,N,M}_T = \sum_{j,k=1}^d \Big[o(0)^{n,M}_{jk,T} + o(1)^{n,N,M}_{jk,T} + e(0)^{n,N,M}_{jk,T} + e(1)^{n,N,M}_{jk,T}\Big],
\end{equation*}
where $o(0)^{n,M}_{jk,T}$, $o(1)^{n,N,M}_{jk,T}$, $e(0)^{n,N,M}_{jk,T}$, $e(1)^{n,N,M}_{jk,T}$ are defined in (\ref{def.es}) and (\ref{def.os}).

Similarly,
\begin{equation*}
	\overline{\underline{S}}(2)^{n,N,M}_T = \sum_{j,k=1}^d \Big[\underline{o}(0)^{n,M}_{jk,T} + \underline{e}(0)^{n,N,M}_{jk,T} + \underline{e}(1)^{n,N,M}_{jk,T}\Big].
\end{equation*}
where
\begin{eqnarray*}
	\underline{o}(0)^{n,M}_{jk,T} &=& N^{1/2} \sum_{h=1}^{n_j}\int_{I^j_h} \big[\underline{\widehat{\rho}}^M_{jk}\big(\tau^j_h\big) - \underline{\rho}_{jk}(t)\big]\dd{jk}{t,t}\,c_{jk}(t)\ds t \\
	\underline{e}(0)^{n,N,M}_{jk,T} &=& N^{1/2} \int_0^T \underline{\widehat{\rho}}^M_{jk}(\thb{j}{t})\, \sigma_{j\cdot}(t) \ds W(t) \int_0^t \dd{jk}{t,u}\,\sigma_{k\cdot}(u) \ds W(u) \nonumber\\
	\underline{e}(1)^{n,N,M}_{jk,T} &=& N^{1/2} \int_0^T \sigma_{k\cdot}(t) \ds W(t) \int_0^t \underline{\widehat{\rho}}^M_{jk}(\thb{j}{u})\, \dd{jk}{u,t}\,\sigma_{j\cdot}(u) \ds W(u).
\end{eqnarray*}

Here we show stable convergence of $\overline{S}(2)^{n,N,M}_T$. The asymptotic analysis of $\overline{\underline{S}}(2)^{n,N,M}_T$ goes along similar lines.

By (\ref{d-1}),
\begin{equation*}
	o(1)^{n,N,M}_{jk,T} = -\frac{\pi^2}{6T^2} (2N+1)^2N^{1/2}\int_0^T \widehat{\rho}_{jk}^M(\thb{j}{t})\, c_{jk}(t)\, \big[\thb{j}{t}-\thb{k}{t}\big]^2 \ds t + O_p\big(N^{9/2}\,\underline{n}^{-4}\big),
\end{equation*}
then If $N=\lfloor\kappa\underline{n}^{4/5}\rfloor$, by (\ref{approximation}) and assumption \ref{a.theta},
\begin{equation}\label{limit.o(1)}
	o(1)^{n,N,M}_{jk,T} \overset{\mathbb{P}}{\longrightarrow} -\frac{2\pi^2\kappa^{5/2}}{3T^2} \int_0^T \partial_{jk}g\big(c(t)\big)\, c_{jk}(t)\, \varrho_{jk}(t) \ds t.
\end{equation}

Now, let's study $\Psi^{n,N,M}_T \coloneqq \sum_{j,k=1}^d \big[e(0)^{n,N,M}_{jk,T} + e(1)^{n,N,M}_{jk,T}\big]$. Because of lemma \ref{lem.stable}, it remains to study the limit of the bracket $\big\langle \Psi^{n,N,M},\Psi^{n,N,M}\big\rangle_T$ in probability,
\begin{multline*}
	\big\langle \Psi^{n,N,M},\Psi^{n,N,M}\big\rangle_T = \sum_{j,k,l,m=1}^d \Big[\big\langle e(0)^{n,N,M}_{jk},e(0)^{n,N,M}_{lm}\big\rangle_T + \big\langle e(0)^{n,N,M}_{jk},e(1)^{n,N,M}_{lm}\big\rangle_T \\
	+ \big\langle e(1)^{n,N,M}_{jk},e(0)^{n,N,M}_{lm}\big\rangle_T + \big\langle e(1)^{n,N,M}_{jk},e(0)^{n,N,M}_{lm}\big\rangle_T\Big],
\end{multline*}
and by (\ref{rep.es}),
\begin{eqnarray*}
	\big\langle e(0)^{n,N,M}_{jk},e(0)^{n,N,M}_{lm} \big\rangle_T &=& N\int_0^T \widehat{\rho}_{jk}^M(\thb{j}{t})\,\widehat{\rho}_{lm}^M(\thb{l}{t})\, c_{jl}(t)\, U^{n,N}_{jk}(t)\,U^{n,N}_{lm}(t) \ds t \\
	\big\langle e(0)^{n,N,M}_{jk},e(1)^{n,N,M}_{lm} \big\rangle_T &=& N\int_0^T \widehat{\rho}_{jk}^M(\thb{j}{t})\, c_{jm}(t)\, U^{n,N}_{jk}(t)\,\widetilde{U}^{n,N,M}_{lm}(t) \ds t \\
	\big\langle e(1)^{n,N,M}_{jk},e(0)^{n,N,M}_{lm} \big\rangle_T &=& N\int_0^T \widehat{\rho}_{lm}^M(\thb{l}{t})\, c_{kl}(t)\, \widetilde{U}^{n,N,M}_{jk}(t)\,U^{n,N}_{lm}(t) \ds  t \\
	\big\langle e(1)^{n,N,M}_{jk},e(1)^{n,N,M}_{lm} \big\rangle_T &=& N\int_0^T c_{km}(t)\, \widetilde{U}^{n,N}_{jk}(t)\,\widetilde{U}^{n,N,M}_{lm}(t) \ds  t,
\end{eqnarray*}
so by (\ref{rep.U-Z}),
\begin{equation}\label{rep.AVAR.multi}
	\big\langle \Psi^{n,N,M},\Psi^{n,N,M}\big\rangle_T = \sum_{j,k,l,m=1}^d\left[\sum_{r=0}^3O(r)^{n,N,M}_{jk,lm,T} + \sum_{r=0}^3 V(r)^{n,N,M}_{jk,lm,T} \right],
\end{equation}
where
\begin{eqnarray*}
	O(0)^{n,N,M}_{jk,lm,T} &=& N\int_0^T \widehat{\rho}_{jk}^M(\thb{j}{t})\,\widehat{\rho}_{lm}^M(\thb{l}{t})\, c_{jl}(t)\, \big[Z^{n,M}_{jk,lm}(t) + Z^{n,M}_{lm,jk}(t)\big] \ds t \\
	O(1)^{n,N,M}_{jk,lm,T} &=& N\int_0^T \Big[\widehat{\rho}_{jk}^M(\thb{j}{t})\,c_{jm}(t)\,\breve{Z}^{n,N}_{jk,lm}(t) + \widehat{\rho}_{lm}^M(\thb{l}{t})\,c_{kl}(t)\,\breve{Z}^{n,N}_{lm,jk}(t)\Big] \ds t \\
	O(2)^{n,N,M}_{jk,lm,T} &=& N\int_0^T \Big[\widehat{\rho}_{jk}^M(\thb{j}{t})\,c_{jm}(t)\,\mathring{Z}^{n,N}_{jk,lm}(t) + \widehat{\rho}_{lm}^M(\thb{l}{t})\,c_{kl}(t)\,\mathring{Z}^{n,N}_{lm,jk}(t)\Big] \ds  t \\
	O(3)^{n,N,M}_{jk,lm,T} &=& N\int_0^T c_{km}(t)\, \big[\widetilde{Z}^{n,M}_{jk,lm}(t) + \widetilde{Z}^{n,M}_{lm,jk}(t)\big] \ds t ,
\end{eqnarray*}
and
\begin{eqnarray*}
	 V(0)^{n,N,M}_{jk,lm,T} &=& \int_0^T \widehat{\rho}_{jk}^M(\thb{j}{t})\,\widehat{\rho}_{lm}^M(\thb{l}{t})\, c_{jl}(t) \ds t\, \Big[N\int_0^t\, \dd{jk}{t,u}\,\dd{lm}{t,u}\, c_{km}(u) \ds u\Big]\\
	 V(1)^{n,N,M}_{jk,lm,T} &=& \int_0^T \widehat{\rho}_{jk}^M(\thb{j}{t})\, c_{jm}(t) \ds t\, \Big[N\int_0^t \dd{jk}{t,u}\,\dd{lm}{u,t}\, \widehat{\rho}_{lm}^M(\thb{l}{u})\, c_{kl}(u)\ds u\Big] \\
	 V(2)^{n,N,M}_{jk,lm,T} &=& \int_0^T \widehat{\rho}_{lm}^M(\thb{l}{t})\, c_{kl}(t) \ds  t\, \Big[N\int_0^t \dd{jk}{u,t}\,\dd{lm}{t,u}\, \widehat{\rho}_{jk}^M(\thb{j}{u})\, c_{jm}(u) \ds u\Big] \\
	 V(3)^{n,N,M}_{jk,lm,T} &=& \int_0^T c_{km}(t) \ds  t\, \Big[N\int_0^t \dd{jk}{u,t}\,\dd{lm}{u,t}\, \widehat{\rho}_{jk}^M(\thb{j}{u})\,\widehat{\rho}_{lm}^M(\thb{l}{u})\, c_{jl}(u)\ds u\Big]. \\
\end{eqnarray*}

To show the asymptotic negligibility of $O(r)^{n,N,M}_{jk,lm,T}$ for $r=0,\cdots,3$, by symmetry, it suffices to study the following terms:
\begin{eqnarray*}
	\phi(0)^{n,N,M}_{jk,lm,T} &\coloneqq& N\int_0^T \widehat{\rho}_{jk}^M(\thb{j}{t})\,\widehat{\rho}_{lm}^M(\thb{l}{t})\, c_{jl}(t)\, Z^{n,N}_{jk,lm}(t) \ds t \\
	\phi(1)^{n,N,M}_{jk,lm,T} &\coloneqq& N\int_0^T \widehat{\rho}_{jk}^M(\thb{j}{t})\,c_{jm}(t)\,\breve{Z}^{n,N}_{jk,lm}(t) \ds t \\
	\phi(2)^{n,N,M}_{jk,lm,T} &\coloneqq& N\int_0^T \widehat{\rho}_{jk}^M(\thb{j}{t})\,c_{jm}(t)\,\mathring{Z}^{n,N}_{jk,lm}(t) \ds  t \\
	\phi(3)^{n,N,M}_{jk,lm,T} &\coloneqq& N\int_0^T c_{km}(t)\, \widetilde{Z}^{n,N}_{jk,lm}(t) \ds t .
\end{eqnarray*}
Note
\begin{eqnarray*}
	\big|\phi(0)^{n,N,M}_{jk,lm,T}\big|^2 &=& N^2\int_0^T\int_0^T \widehat{\rho}_{jk}^M(\thb{j}{t})\,\widehat{\rho}_{lm}^M(\thb{l}{t})\,\widehat{\rho}_{jk}^M(\thb{j}{u})\,\widehat{\rho}_{lm}^M(\thb{l}{u}) \\ 
	&&\hspace{5.8cm} c_{jl}(t)\,c_{jl}(u)\, Z^{n,N}_{jk,lm}(t)\,Z^{n,N}_{jk,lm}(u)\ds t\ds u \\
	\big|\phi(1)^{n,N,M}_{jk,lm,T}\big|^2 &=& N^2\int_0^T\int_0^T \widehat{\rho}_{jk}^M(\thb{j}{t})\,\widehat{\rho}_{jk}^M(\thb{j}{u})\, c_{jm}(t)\,c_{jm}(u)\,  \breve{Z}^{n,N}_{jk,lm}(t)\,\breve{Z}^{n,N}_{jk,lm}(u)\ds t\ds u \\
	\big|\phi(2)^{n,N,M}_{jk,lm,T}\big|^2 &=& N^2\int_0^T\int_0^T \widehat{\rho}_{jk}^M(\thb{j}{t})\,\widehat{\rho}_{jk}^M(\thb{j}{u})\, c_{jm}(t)\,c_{jm}(u)\,  \mathring{Z}^{n,N}_{jk,lm}(t)\,\mathring{Z}^{n,N}_{jk,lm}(u)\ds t\ds u \\
	\big|\phi(3)^{n,N,M}_{jk,lm,T}\big|^2 &=& N^2\int_0^T\int_0^T c_{km}(t)\,c_{km}(u)\,  \widetilde{Z}^{n,N}_{jk,lm}(t)\,\widetilde{Z}^{n,N}_{jk,lm}(u)\ds t\ds u,
\end{eqnarray*}
by (\ref{cond.g}), assumption \ref{a.SU}, lemma \ref{lem.U2.Z2},
\begin{equation*}
	\E\big(\big|\phi(0)^{n,N,M}_{jk,lm,T}\big|^2\big) 
	\le KN^2\int_0^T\int_0^T\ds t\ds u\, \Big[\int_0^{t\wedge u} \dd{jk}{t,v}\,\dd{jk}{u,v}\ds v \int_0^v \dd{lm}{v,\vartheta}^2\ds \vartheta\Big],
\end{equation*}
by (\ref{bdd.dd}), 
\begin{equation}
	\E\big(\big|\phi(0)^{n,N,M}_{jk,lm,T}\big|^2\big) \le KT^{\frac{3p-2}{p}}N^{1-2/p}.
\end{equation}

By similar arguments, we can show the same upper bound also applies to $\E\big(\big|\phi(1)^{n,N,M}_{jk,lm,T}\big|^2\big)$, $\E\big(\big|\phi(2)^{n,N,M}_{jk,lm,T}\big|^2\big)$, $\E\big(\big|\phi(3)^{n,N,M}_{jk,lm,T}\big|^2\big)$. Thus by taking $p\in(1,2)$ and using Jensen's inequality and Markov's inequality, we can prove that $\phi(r)^{n,N,M}_{jk,lm,T},\,r=0,1,2,3$ all converge to 0 in probability.
\begin{lem}\label{lem.null.multi}
Under the assumptions of theorem \ref{thm.multi}, $\forall j,k,l,m=1,\cdots,d$,
\begin{equation*}
	\max_{r=0,1,2,3}\phi(r)^{n,N,M}_{jk,lm,T} \overset{\mathbb{P}}{\longrightarrow} 0.
\end{equation*}
\end{lem}

By (\ref{approximation}) and lemma \ref{lem.dd}, we have
\begin{lem}\label{lem.AVAR.multi}
Under the assumptions of theorem \ref{thm.multi}, $\forall j,k,l,m=1,\cdots,d$,
\begin{eqnarray*}
	V(0)^{n,N,M}_{jk,lm,T} &\overset{\mathbb{P}}{\longrightarrow}& \int_0^T \rho_{jk}(t)\,\rho_{lm}(t)\, \tilde{\theta}_{jk,lm}(t)\, c_{jl}(t)\,c_{km}(t) \ds t \nonumber\\
	V(1)^{n,N,M}_{jk,lm,T} &\overset{\mathbb{P}}{\longrightarrow}& \int_0^T \rho_{jk}(t)\,\rho_{lm}(t)\, \acute{\theta}_{jk,lm}(t)\, c_{jm}(t)\,c_{kl}(t) \ds t \nonumber\\
	V(2)^{n,N,M}_{jk,lm,T} &\overset{\mathbb{P}}{\longrightarrow}& \int_0^T \rho_{jk}(t)\,\rho_{lm}(t)\, \check{\theta}_{jk,lm}(t)\, c_{jm}(t)\,c_{kl}(t) \ds t \nonumber\\
	V(3)^{n,N,M}_{jk,lm,T} &\overset{\mathbb{P}}{\longrightarrow}& \int_0^T \rho_{jk}(t)\,\rho_{lm}(t)\, \grave{\theta}_{jk,lm}(t)\, c_{jl}(t)\,c_{km}(t) \ds t.
\end{eqnarray*}
\end{lem}

In view of (\ref{rep.AVAR.multi}), theorem \ref{thm.multi} follows from (\ref{null.o(0)}), (\ref{null.o(1)}), lemma \ref{lem.stable}, \ref{lem.null.multi} and \ref{lem.AVAR.multi}; proposition \ref{prop.noncenter} follows from (\ref{null.o(0)}), (\ref{limit.o(1)}), lemma \ref{lem.stable}, \ref{lem.null.multi} and \ref{lem.AVAR.multi}. 

The stable convergence and the asymptotic variance of $\overline{\underline{S}}(2)^{n,N,M}_T$ can be shown by an analogous derivation, from which proposition \ref{prop.sqrtn} follows.

\bibliographystyle{apa}
\bibliography{../Reference}
\end{document}